\documentclass[11pt,a4paper]{article}

\usepackage[a4paper,left=2.5cm,right=2.5cm,top=2.6cm,bottom=2.6cm]{geometry}
\usepackage{xeCJK}
\IfFileExists{C:/Windows/Fonts/SimSun.ttc}
  {\setCJKmainfont{SimSun}}
  {\setCJKmainfont{FandolSong-Regular.otf}[BoldFont=FandolSong-Bold.otf]}
\usepackage{amsmath,amssymb,amsthm,mathrsfs,mathtools}
\usepackage{bm}
\allowdisplaybreaks[4]

\usepackage{tikz-cd}
\usetikzlibrary{arrows.meta}
\usepackage{booktabs}
\usepackage{longtable}
\usepackage{array}
\usepackage{caption}
\usepackage[colorlinks,linkcolor=blue!55!black,citecolor=teal!70!black,
            urlcolor=blue!55!black,bookmarksnumbered]{hyperref}

\newtheorem{theorem}{Theorem}[section]
\newtheorem{proposition}[theorem]{Proposition}
\newtheorem{lemma}[theorem]{Lemma}
\newtheorem{corollary}[theorem]{Corollary}
\theoremstyle{definition}
\newtheorem{definition}[theorem]{Definition}
\newtheorem{example}[theorem]{Example}
\newtheorem{remark}[theorem]{Remark}

\newcommand{\K}{\mathbb{K}}

\newcommand{\R}{\mathbb{R}}
\newcommand{\C}{\mathbb{C}}
\newcommand{\N}{\mathbb{N}}

\newcommand{\cc}{\mathfrak{c}}
\newcommand{\ba}{\mathrm{ba}}

\newcommand{\VF}{\mathbf{VF}}
\newcommand{\NF}{\mathbf{NF}}
\newcommand{\BF}{\mathbf{BF}}
\newcommand{\iNF}{\mathbf{iNF}}
\newcommand{\iBF}{\mathbf{iBF}}
\newcommand{\Ban}{\mathbf{Ban}}
\newcommand{\iBan}{\mathbf{iBan}}

\DeclareMathOperator{\id}{id}
\DeclareMathOperator{\Ker}{Ker}

\DeclareMathOperator{\Span}{span}
\DeclareMathOperator{\Dir}{Dir}
\DeclareMathOperator{\Inv}{Inv}
\DeclareMathOperator{\Sh}{Sh}
\DeclareMathOperator{\Cl}{Cl}
\DeclareMathOperator{\dens}{dens}
\DeclareMathOperator{\Ob}{Ob}

\newenvironment{pquote}{\begin{quote}\small\itshape}{\end{quote}}

\title{\bfseries Basic Theory of Dual Chains
       \\[3pt] \large with a Discussion of Several Errors in Three Papers by Ugle\v{s}i\'c}
\author{Quanfeng Liu\thanks{Contact: \texttt{cfj262@outlook.com}; Independent Researcher, People's Republic of China}\\[6pt]
        \small Nikica Uglešić}
\date{September 2026}

\begin{document}
\maketitle

\begin{abstract}
\noindent For a Banach space $X$, we study the completed direct limit of its even iterated duals along the
canonical embeddings and the bounded inverse limit of the corresponding odd duals. We give explicit
constructions, establish their duality relation, and describe the comparison maps associated with
iterated dualization. We then examine several assertions in three preprints of Ugle\v{s}i\'c, giving
counterexamples concerning dimension stability and the commutation of bidualization with
direct limits. The finite quotient shape (the category whose objects are the finite-dimensional
quotients, with all continuous linear maps between them) is classified by the algebraic dimension of
the continuous dual. We also compute the dimensions of the finite iterated duals and the completed
hyperdual space of $\ell^{\infty}(\N)$, and prove that, inside the hyperdual space, the canonical copy of a higher
layer need not be the range of an $M$-projection. Throughout, uncompleted
limits, Banach completions, and the different meanings of isomorphism are distinguished explicitly.

\end{abstract}

\tableofcontents
\bigskip

%=====================================================================
\section{Introduction}\label{sec:intro}

\subsection{Origin of the problem}\label{sub:origin}

Let $\K$ be the field of real numbers $\R$ or the field of complex numbers $\C$. For a normed
space $X$ over $\K$, denote by $D(X)=X^{*}$ its dual space (that is, the Banach space of all
continuous linear functionals on $X$ endowed with the operator norm), and let $D^{n}(X)$ denote
the $n$-th iterated dual, $X^{(n)}=D^{n}(X)$. This paper discusses an elementary
question: \emph{what do we get if we keep taking duals of a Banach space?} 

Classical functional analysis gives two basic answers to the question.

\textbf{(1) Non-reflexivity is hereditary.} By a short argument due to Goldberg \cite{goldberg1960}
(see Theorem \ref{thm:goldberg} below), $X$ is reflexive if and only if $X^{*}$ is reflexive.
Hence, as soon as $X$ is non-reflexive, every space in the dual chain is non-reflexive, and the two
chains
\begin{equation}\label{eq:chains-intro}
  X\ \subset\ X^{**}\ \subset\ X^{(4)}\ \subset\ \cdots ,
  \qquad
  X^{*}\ \subset\ X^{(3)}\ \subset\ X^{(5)}\ \subset\ \cdots
\end{equation}
(the embeddings being the canonical embeddings and their iterates) are \emph{strictly}
increasing: they never ``stabilize'' at a reflexive space.

\textbf{(2) Shape and dimension alone do not suffice to determine the behavior of the chains.}
James \cite{james1951} constructed a non-reflexive space $J$ that is isometrically isomorphic
to its own second dual: $J\equiv J^{**}$. Thus, whether the chain
$X\subset X^{**}\subset\cdots$ is strict cannot be determined from the type, the shape or the
algebraic dimension of the individual levels alone; one has to examine \emph{how the embeddings are
chosen}.

The question then naturally becomes: what do these two strictly increasing chains ``tend to''?
In algebraic terms: what are their limits as \emph{direct systems} and \emph{inverse systems}
(see Definition \ref{def:limits})? This is precisely the subject of the present paper.

\paragraph{A change of terminology.}
Before proceeding, a comment on terminology is in order. The old name ``tower system'' used in the early version of this paper is not
unambiguous in the mathematical literature, and even less so outside mathematics: some editors of some
journals in the biomedical and environmental sciences took it for a device of the
fermentation-tower type many times and sent the first author too many calls for papers unrelated to mathematics. Such misunderstandings are mathematically irrelevant, yet they show that the old name is
not informative enough for interdisciplinary communication. We have therefore adopted the new
Chinese and English names listed in the table below (the table is taken from
\cite{concrete-dual-chains}).

\begin{center}
\renewcommand{\arraystretch}{1.25}
\begin{tabular}{lll}
\toprule
Original name & New Chinese name & New English name\\
\midrule
Tower System & （偶阶）对偶链 & (Even-Order) Dual Chain\\
Tower Space ($\widetilde{X}$) & 超对偶空间 & Hyperdual Space\\
Inverse Tower Space ($\overleftarrow{X}$) & 逆超对偶空间 & Inverse Hyperdual Space\\
\bottomrule
\end{tabular}
\end{center}

\noindent The mathematical notation is unchanged. To facilitate comparison with the published
literature, the old terms ``tower system'', ``tower space'' and ``inverse tower space'' are kept
when titles are quoted or the sources are paraphrased. The change of terminology does not affect
any mathematical content: definitions, theorems and proofs correspond verbatim, and only a more
descriptive name has been adopted.

\subsection{The state of the art}\label{sub:status}

\subsubsection{From the perspective of shape theory}

Shape theory was introduced by Borsuk \cite{borsuk1968} (a systematic account is given in his
monograph \cite{borsuk1975}) in order to extend homotopy theory from
polyhedra to general compact metric spaces (and to still more general spaces). Its basic
technique is to represent a space as an \emph{inverse system} of ``nice'' spaces and then to
replace the ``homotopy type'' by the inverse limit. This method was organized into the
``inverse system approach'' by Marde\v{s}i\'c and Segal \cite{ms1982} and by
Dydak--Segal \cite{dydaksegal1978}, and was given a categorical treatment by
Cordier--Porter \cite{cordierporter1989}; Marde\v{s}i\'c \cite{mardesic1997} gives a survey of
thirty years of development.

An important lesson from shape theory is this: \emph{a ``limit''
is meaningful only when taken in the correct category}, and the choice of the
connecting maps (arbitrary continuous maps, or surjections, or isomorphisms) determines the
whole character of the theory. The discussion in Section \ref{sec:cat} of this paper is
precisely a manifestation of this lesson in the case of normed spaces.

\subsubsection{From the perspective of normed spaces and quotient shapes}

Since 2016, Ugle\v{s}i\'c has carried over the notion of the ``quotient shape'' from shape
theory to the category of vector spaces: reference \cite{uglesic2016} develops the theory in a
general concrete category, \cite{uglesic2017} treats the case of vector spaces and normed
spaces, \cite{uglesic2018t} the case of general topological spaces, and \cite{uglesic2018lp}
the case of $l_p$ and $L_p$ spaces. On this basis, the following three arXiv preprints (all
submitted in 2019)
\begin{itemize}
  \item[\textbf{[U1]}] N.~Ugle\v{s}i\'c, \emph{The quotient shapes of normed spaces and
    application}, arXiv:1903.06456v1 [math.FA], \cite{U1};
  \item[\textbf{[U2]}] N.~Ugle\v{s}i\'c, \emph{On the duals of normed spaces and quotient
    shapes}, arXiv:1903.06459v1 [math.FA], \cite{U2};
  \item[\textbf{[U3]}] N.~Ugle\v{s}i\'c, \emph{An extension of the normed dual functors},
    arXiv:1903.06467v2 [math.FA], \cite{U3},
\end{itemize}
give the impression that ``the classification of the (finite) quotient shapes of
normed spaces has been completely settled'': here [U3] uses the technique of direct limits to
associate with every normed space $X$ a ``hyperdual space'' $\widetilde X$ which satisfies
$D^{2}(\widetilde X)\cong\widetilde X$ and contains $X$ and all even-order duals $D^{2n}(X)$;
these spaces form an increasing sequence under the canonical isometric embeddings. This is
precisely the object that we call the hyperdual space of this paper (Definition \ref{def:tower-space}).
The uncompleted direct-limit object used in [U3] is denoted here by $U_X$; our hyperdual space
$\widetilde X$ is its Banach completion. This distinction is substantive: Example
\ref{ex:cat-canonical} shows that $U_X$ is incomplete for every non-reflexive Banach space $X$.

\subsubsection{From the perspective of limits of Banach spaces}

Semadeni and Zidenberg \cite{semadeni1965} were the first to make a systematic study of the
inductive and inverse limits in the category of Banach spaces (see \cite{semadeni1971}).
Gwizdek \cite{gwizdek2023} (see also his doctoral dissertation \cite{gwizdek2024}) explicitly
uses the name ``\emph{tower system}'' for the sequence of iterated second duals, and introduces
two functors $\Dir$ and $\Inv$ for normed spaces and bounded linear operators; he proves that
$\Dir$ preserves direct sum decompositions, the spectrum of an operator and the Fredholm
property, and finally applies this machinery to the problem of the support of representing
measures in function algebras. In this paper we adopt instead the name
``(even-order) dual chain'', with $\widetilde{(\cdot)}$ and $\overleftarrow{(\cdot)}$ denoting $\Dir$ and
$\Inv$ (see Definitions \ref{def:tower-space} and \ref{def:inv-tower}).

\subsubsection{The contribution of this paper}

Section \ref{sec:prelim} gives explicit proofs of the limit constructions and estimates used in the
corrections. Standard operator-theoretic properties summarized in Proposition \ref{prop:preserved}
are quoted from Gwizdek \cite{gwizdek2023}. The corrections in Section \ref{sec:corrections} rely
on the results and proofs established independently at each step, rather than on the disputed
assertions of the three preprints. The third part (Section \ref{sec:open}) addresses two further
questions about the objects of the theory and answers both in the negative.

\subsection{Problems in the three preprints}\label{sub:problems}

While writing \cite{dualchain}, and in order to find out what was already known about the
problem of the ``limit of the dual chain'', the main conclusions of [U1], [U2] and
[U3] were checked one by one, and \emph{explicit counterexamples} to several of them were found. These errors
fall into three classes, corresponding to the three subsections of Section
\ref{sec:corrections} of this paper:

\begin{enumerate}
  \item[\textbf{(E1)}] \textbf{Dimension conclusions.} Theorem 4, Theorem 5, Corollary 4,
    Corollary 5 and Remark 3 of [U2], as well as Theorem 2(iii) of [U3], all rest on the
    assumption that the algebraic (Hamel) dimension $\dim X^{*}$ of the dual space
    is determined by $\dim X$. In §\ref{sec:dim} we show that this picture is wrong: the
    density character $\dens(X)$ only gives a \emph{general upper bound} for the cardinality of
    the dual, $\dim X\le\dim X^{*}\le|X^{*}|\le\cc^{\dens(X)}$, and does not determine
    $\dim X^{*}$.
  \item[\textbf{(E2)}] \textbf{Category and functors.} The statements in Lemma 4 and in
    Theorem 1 and Theorem 2 of [U3] concerning ``the hyperdual space is a Banach space'' and
    ``the second dual functor commutes with the direct limit'' depend on the category in which
    the limit is taken. In §\ref{sec:cat} we use commutative diagrams to describe the correct
    picture: the limit must be taken in $\Ban_1$ (or in $\iBan$), and ``first take the limit
    and then the second dual'' and ``first take the second dual and then the limit'' are in
    general \emph{not} isomorphic.
  \item[\textbf{(E3)}] \textbf{Wording.} The formulation of Theorem 2(ii) of [U3] about the
    direct sum decomposition is loosely worded; we record the topological direct-sum
    interpretation---``$\dotplus$'' denotes the topological direct sum, and ``isometrically''
    qualifies the two embeddings---and show that it holds unconditionally; we explain why
    the stronger isometric interpretations cannot hold (§\ref{sec:wording}), and we
    list a number of misprints in the three papers and give a comparison of
    terminology and numbering.
\end{enumerate}

Two remarks are in order. First, the problems above have nothing to do with the
\emph{remaining} conclusions of [U1], [U2] and [U3]: a large number of technical lemmas in the
three papers (on the adjointness of the canonical embedding, on the cofinite modification of
inverse systems, on $D$ turning retractions into embeddings, and so on) are correct and can
continue to be used; in §\ref{sec:cat} and §\ref{sec:wording} of this paper we point out
explicitly which conclusions are unaffected. Second, the aim of this paper is not to assign
blame. The present paper grew out of a correspondence between the two authors in
September 2026 about the problems above: Ugle\v{s}i\'c readily acknowledged these observations
and, no longer having the energy to go through the old manuscripts again, agreed that the
correct theory be written up in the present form.

\subsection{Structure of the paper and main results}\label{sub:summary}

The paper consists of three parts. The first part (Section \ref{sec:prelim}) is a
self-contained preparation: we establish the basic theory of dual chains independently of the
three preprints mentioned above. The second part (Section
\ref{sec:corrections}) is the corrective part; according to (E1)--(E3) of \S\ref{sub:problems}
it splits into three subsections, each of which proceeds in the order ``quote the original
statement $\to$ give a counterexample $\to$ write down the correct statement (or explain that
the question is open)''. The third part (Section \ref{sec:open}) adds three facts about the
objects of the theory: the exact dimension value $\dim\ell^{\infty**}=2^{2^{\cc}}$, together
with negative answers to two further questions, \textbf{(O1)} and \textbf{(O2)}.

The contents of the first part are as follows. We first fix the category
(\S\ref{sub:cat-prelim}: $\Ban$, $\Ban_1$, $\iBan$ and, after dropping completeness, $\NF$,
$\iNF$); this is a prerequisite for the later discussion of whether $D^{2}$ commutes with the
direct limit, and it also fixes what the word ``limit'' means here, since a limit depends on the category in which it is taken. Then we
introduce the two central objects of this paper---the hyperdual space and the
inverse hyperdual space (\S\ref{sub:tower}):
\[
  \widetilde X=\Dir(X)=\varinjlim_n X^{(2n)},
  \qquad
  \overleftarrow X=\Inv(X)=\varprojlim_n X^{(2n+1)} .
\]
Here the objects of the direct limit are taken to be the explicitly constructed completions
(this is a key difference in convention between this paper and [U3]; see
\S\ref{sub:cat-category}), so that $\widetilde X$ is a Banach space by construction, and we prove
that $\overleftarrow X\equiv\widetilde X^{\,*}$ (Theorem \ref{thm:dir-dual-inv}). We also point
out that there is a third natural object on the chain, namely
$\varinjlim_n X^{(2n+1)}=\widetilde{X^{*}}$, which is in general \emph{not} to be identified with
$\overleftarrow X$ (Remark \ref{rem:odd-direct}); this distinction becomes visible only when
$X$ is non-reflexive. Afterwards we give the basic properties of the limit functors
(\S\ref{sub:functors}), illustrate the exact shape of the limit by taking the James space as a
concrete example (\S\ref{sub:james}), and finally establish the basic facts about the Hamel
dimension of (hyper)dual spaces (\S\ref{sub:dim-basic}), the core of which is
\begin{equation}
  \dim X\ \le\ \dim X^{*}\ \le\ |X^{*}|\ \le\ \cc^{\dens(X)} ,
\end{equation}
that is, the \emph{density character} $\dens(X)$ yields a general upper bound for $\dim X^{*}$,
whereas $\dim X$ does not determine it. The strictness of the second chain is independent of
Goldberg's theorem
(Theorem \ref{thm:goldberg}), but in order to understand what it ``tends to'' one has to use
the first part.

The three subsections of the second part give the following results. \S\ref{sec:dim} deals
with dimensions: several assertions in the three preprints (Theorem 4, Theorem 5, Corollary 4,
Corollary 5 and Remark 3 of [U2], Lemma 4(ii) and Question 1 of [U1], and Theorem 2(iii) of
[U3]) are based on the picture that ``$\dim X^{*}$ is determined by $\dim X$'', and this
picture is wrong (Theorem \ref{thm:dim-counter}): when $\dim X=\cc$, $\dim X^{*}$ can be
either $\cc$ (as for $\ell^{1}$, $c_{0}$ and the James space) or $2^{\cc}$ (as for
$\ell^{\infty}$). The correct quantitative picture is Proposition \ref{prop:dim-correct}, the
root cause of the error being that the Hamel dimension was mistaken for an input of the dual operator (Remark
\ref{rem:dim-root}); for the hyperdual space we use the general lower bound
$\dim\widetilde X\ge\sup_{n}\dim X^{(2n)}$ (Proposition \ref{prop:dim-tower-lower}), while for
$\ell^{\infty}$ the exact values are known and are given
by the beth numbers, $\dim\widetilde{\ell^{\infty}}=\beth_{\omega+1}$ (Remark \ref{rem:beth-tower}).
\S\ref{sec:cat} deals
with categories: the assertion in [U3] that ``the second dual functor commutes with the direct
limit'' depends on the category in which the direct limit is taken; the correct procedure is
to take the objects to be the explicitly constructed completions (in agreement with Gwizdek
\cite{gwizdek2023} and Semadeni--Zidenberg \cite{semadeni1965}), and then Theorem
\ref{prop:colim-embed} replaces the equality in [U3, Theorem 1 (i)] by an \emph{isometric
embedding} \eqref{eq:cat-correct}, so that [U3, Theorem 2 (i)] can only be downgraded to
$\widetilde X\hookrightarrow D^{2}(\widetilde X)$
(\eqref{eq:cat-hyperdual-embed}). In particular \S\ref{sec:cat} separates the completeness of the
uncompleted limit from the preservation of a completed limit under bidualization. The general
comparison map is isometric and need not be surjective. For the dual chain, the composite
constructed in \S\ref{sub:cat-correct} is precisely $J_{\widetilde X}$, whose surjectivity is
equivalent to reflexivity of $X$ by Proposition \ref{prop:limit-reflexive}. The existence of an
arbitrary bounded or isometric isomorphism $\widetilde X\to\widetilde X^{**}$ is a separate
assertion, and this comparison-map argument does not settle it for all $X$. An earlier version of the present paper used the parareflexivity terminology of [U3]: in the usage of [U3] parareflexivity requires
only for the \emph{existence} of some isometric isomorphism $T(X)\to T(X)^{**}$, whereas the
surjectivity of a \emph{specified} embedding is in general a strictly
stronger requirement; in the classical case, where the specified map is
the canonical embedding $J_{X}$, surjectivity is exactly reflexivity
(Remark \ref{rem:reflexive-vs-iso}, Example \ref{ex:james}); [U3, Remark
1 (b)] itself lists the surjectivity question as undecided. \S\ref{sec:wording} deals with the wording: the formulation of [U3, Theorem 2
(ii)] about $\widetilde X\cong D^{2k}(X)\dotplus N(r_{2k})$ is loosely worded, and we record
the topological direct-sum interpretation ($\dotplus$ denotes the topological direct sum, and
``isometrically'' qualifies only the two embeddings), which holds unconditionally;
we explain why the stronger isometric interpretations cannot hold, and we also show that the
conclusion follows directly from Proposition \ref{prop:tower-layer-retractions}, established
independently;
at the end of that subsection
we also list a number of misprints in the three preprints and give a comparison of terminology
and of the numbering of the references.

The third part (Section \ref{sec:open}) completes the picture with three facts about the
objects of the theory: the exact dimension value $\dim\ell^{\infty**}=2^{2^{\cc}}$
\eqref{eq:open-dim}, which replaces the upper bound used in \S\ref{sec:dim} and shows that the
general lower bound $\dim\widetilde X\ge\sup_{n}\dim X^{(2n)}$ of Proposition
\ref{prop:dim-tower-lower} can be strict; and answers to the two questions \textbf{(O1)} and
\textbf{(O2)} of \S\ref{sec:open}, both of which are \emph{negative}: $\Sh_{\bar0}(\ell^{\infty})$
differs from the unique $l_{p}$-type $\Sh_{\bar0}(l_{p})$ (Theorem
\ref{thm:shape-classification}), and the projection onto a layer of a hyperdual space cannot in
general be chosen to be an $M$-projection (Proposition \ref{prop:no-M-projection}, with the
pair $X=c_{0}$, $k=1$). The weaker question whether a two-contracted retraction with
$\|r_{2k}\|\le1$ and $\|\id-r_{2k}\|\le1$ can always be found is a separate question, and it is
not settled here. Neither answer depends on the choices we made with regard to the three
preprints.

\subsection{AI usage statement}

The research and writing of this paper made use of DeepSeek-V4.1-Flash. GPT-6 Pro and Opus 5.5
were used to check this paper. 

%=====================================================================
\section{Preliminaries}\label{sec:prelim}

In this section we fix the definitions and basic facts that are used repeatedly throughout the
paper. Conventions: $\K$ denotes the field of real numbers $\R$ or the field of complex
numbers $\C$, and a vector space always means a vector space over $\K$; letters such as $X,Y$
denote normed linear spaces, the norm is always written $\|\cdot\|$, and the same symbol is
used for the norm on the dual space (no confusion can arise from the context). In accordance
with the character of this paper, this section is a \emph{self-contained} preparation:
whenever we make an assertion about dual chains we give an independent definition and proof,
without relying on any of the preprints under review; in fact, this section does not quote the
three preprints that will be examined in Section \ref{sec:corrections}.

Throughout the construction of the dual chain, the initial space $X$ is assumed to be a Banach space.
Statements about dual cardinalities and finite quotient shapes are formulated separately for
arbitrary normed spaces when appropriate. We write
\[
X_n=X^{(2n)},\qquad
\kappa_{n,n+1}=J_{X_n}:X_n\longrightarrow X_{n+1},\qquad
\kappa_{n,\infty}:X_n\longrightarrow\widetilde X.
\]
After identifying the layers by their connecting isometries, let
\[
U_X=\bigcup_{n\geq0}\kappa_{n,\infty}(X_n),
\qquad \widetilde X=\overline{U_X}.
\]
Thus $U_X$ denotes the uncompleted direct limit, whereas $\widetilde X$ denotes its Banach
completion. For a fixed $k\geq1$, a retraction onto the $k$-th layer is a bounded linear map
\[
r_{2k}:\widetilde X\longrightarrow X_k,
\qquad r_{2k}\kappa_{k,\infty}=\id_{X_k}.
\]
The corresponding projection on the hyperdual space is
\[
P_k=\kappa_{k,\infty}r_{2k}:\widetilde X\longrightarrow\widetilde X.
\]
In particular, the complementary operator is $\id_{\widetilde X}-P_k$. The adjacent-layer embedding
$J_{X_k}$ and the structure map $\kappa_{k,\infty}$ have different codomains and are kept
distinct throughout.

The material of this section is arranged in three layers. The first layer
(\S\ref{sub:dual-prelim}--\S\ref{sub:examples-prelim}) is basic functional analysis: dual
spaces, the canonical embedding, reflexivity, the distinction between the three kinds of
``isomorphism'', Goldberg's theorem, and the propagation of non-reflexivity along the dual
chain. The second layer (\S\ref{sub:cat-prelim}) is basic category theory: categories,
functors, direct and inverse systems and their limits, and why the limit of the dual chain
must be taken in $\Ban_1$ (indeed in $\iBan$). The third layer
(\S\ref{sub:tower}--\S\ref{sub:dim-basic}) is the dual chain itself: the construction,
the universal properties and the mutual relations of the two limit spaces (the hyperdual space
$\widetilde X$ and the inverse chain space $\overleftarrow X$), the properties of the two
limit functors $\Dir$ and $\Inv$, and the basic facts about the Hamel dimension of
(hyper)dual spaces.

The central conclusions of this section are: for a non-reflexive Banach space $X$ the two
dual chains are strictly increasing and never ``stabilize'' (Theorem \ref{thm:chain},
Corollary \ref{cor:strict-chain}); after gluing the two chains into a direct limit and an
inverse limit respectively, the resulting spaces are still non-reflexive (Proposition
\ref{prop:limit-reflexive}); and the two limit spaces are dual to each other, that is,
$\overleftarrow X\equiv\widetilde X^{*}$ (Theorem \ref{thm:dir-dual-inv}).

\subsection{Dual spaces, the canonical embedding and reflexivity}\label{sub:dual-prelim}

\begin{definition}[Dual space and higher duals]\label{def:dual}
Let $(X,\|\cdot\|)$ be a normed linear space. The set of all continuous linear functionals on
$X$,
\[
  X^{*}=\bigl\{\,f:X\to\K\ \bigm|\ f\ \text{is linear and continuous}\,\bigr\}
\]
with pointwise addition and scalar multiplication, is a vector space over $\K$; it is called
the \emph{dual space} (conjugate space) of $X$ and is endowed with the \emph{dual norm}
\[
  \|f\|_{X^{*}}=\sup\bigl\{\,|f(x)|\bigm| x\in X,\ \|x\|\le1\,\bigr\}.
\]
In the sequel we abbreviate $\|f\|_{X^{*}}$ to $\|f\|$. Inductively we define the
\emph{$n$-th dual space} of $X$ by
\[
  X^{(0)}=X,\qquad X^{(n+1)}=\bigl(X^{(n)}\bigr)^{*}\quad(n\ge0),
\]
and we call $D(X):=X^{*}$ the \emph{dual functor} and $D^{n}(X)=X^{(n)}$ the $n$-fold dual.
\end{definition}

The dual space is always ``better'' than the original space: it is automatically complete.
This fact is used repeatedly in the inductive arguments below.

\begin{lemma}[The dual space is a Banach space]\label{lem:dual-banach}
Let $Y$ be a normed linear space. Then $Y^{*}$ is complete with respect to the dual norm,
that is, $Y^{*}$ is a Banach space. In particular, the dual space of any normed linear space
is a Banach space.
\end{lemma}

\begin{proof}
Let $(f_n)$ be a Cauchy sequence in $Y^{*}$. For every $y\in Y$ we have
$|f_n(y)-f_m(y)|\le\|f_n-f_m\|\,\|y\|$, so $(f_n(y))$ is a Cauchy sequence in $\K$; write
$f(y)=\lim_{n\to\infty}f_n(y)$.

First we show that $f\in Y^{*}$. The linearity of $f$ follows from the fact that limits
commute with linear operations. Since $(f_n)$ is a Cauchy sequence, $\|f_n\|$ is bounded; let
$M$ be a bound for it. From $|f_n(y)|\le M\|y\|$, letting $n\to\infty$ we get
$|f(y)|\le M\|y\|$, so $f$ is continuous and $\|f\|\le M$.

Next we show that $f_n\to f$. Given $\varepsilon>0$, choose $N$ such that
$\|f_n-f_m\|\le\varepsilon$ whenever $m,n\ge N$. For fixed $y$, letting $m\to\infty$ gives
$|f_n(y)-f(y)|\le\varepsilon\|y\|$, and this holds for all $y\in Y$ and all $n\ge N$; hence
$\|f_n-f\|\le\varepsilon$ for $n\ge N$. Thus $f_n\to f$ and $Y^{*}$ is complete.
\end{proof}

\begin{definition}[The second dual and the canonical embedding]\label{def:canonical}
Let $X$ be a normed linear space. For every $x\in X$, the map $J_X(x):X^{*}\to\K$ defined by
\[
  \bigl(J_X(x)\bigr)(f)=f(x),\qquad f\in X^{*},
\]
is a continuous linear functional on $X^{*}$, that is, $J_X(x)\in X^{**}:=(X^{*})^{*}$.
The linear map
\[
  J_X:X\longrightarrow X^{**}
\]
is called the \emph{canonical embedding} (natural embedding, canonical map) of $X$ into its
\emph{second dual space}. Applying this construction level by level, we obtain spaces and
embeddings
\[
  X^{(m)}\xrightarrow{\ J_{X^{(m)}}\ }X^{(m+2)},\qquad m\ge0 .
\]
\end{definition}

\begin{lemma}[The canonical embedding is isometric]\label{lem:isometry}
For any normed linear space $X$ and any $x\in X$ we have $\|J_X(x)\|=\|x\|$.
Consequently $J_X$ is injective and $J_X(X)$ is a linear subspace of $X^{**}$; if $X$ is
complete, then $J_X(X)$ is moreover a \emph{closed} linear subspace of $X^{**}$.
\end{lemma}

\begin{proof}
For $\|f\|\le1$ we have $|\bigl(J_X(x)\bigr)(f)|=|f(x)|\le\|x\|$, hence
$\|J_X(x)\|\le\|x\|$. The reverse inequality comes from the Hahn--Banach theorem: if $x\ne0$,
then there is $f\in X^{*}$ with $\|f\|=1$ and $f(x)=\|x\|$, whence
$\|J_X(x)\|\ge|f(x)|=\|x\|$ (for $x=0$ the assertion is trivial). Therefore
$\|J_X(x)\|=\|x\|$.

An isometric map is clearly injective. Moreover an isometric map carries Cauchy sequences to
Cauchy sequences and preserves completeness: if $X$ is complete, then $J_X(X)$ is complete,
while $X^{**}=(X^{*})^{*}$ is complete by Lemma \ref{lem:dual-banach}; a complete subspace of
a complete space is closed, so $J_X(X)$ is closed in $X^{**}$.
\end{proof}

\begin{definition}[Reflexivity]\label{def:reflexive}
Let $X$ be a normed linear space. If the canonical embedding $J_X:X\to X^{**}$ is surjective
(equivalently, $J_X(X)=X^{**}$), then $X$ is called \emph{reflexive}. By Lemma
\ref{lem:isometry}, in that case $J_X$ is an isometric isomorphism, so a reflexive normed
linear space is always isometrically isomorphic to its second dual space.
\end{definition}

\begin{remark}[Two points on reflexivity]\label{rem:reflexive-points}
\begin{enumerate}
  \item[(i)] The requirement that the canonical embedding be surjective is essential:
    reflexivity means that the \emph{particular canonical map} $J_X$ is surjective, and not
    that ``there exists some isometric isomorphism $X\to X^{**}$''. Example \ref{ex:james}
    provides a Banach space $J$ that is \emph{isometrically isomorphic} to $J^{**}$ and yet
    \emph{non-reflexive}. Hence in the sequel ``non-reflexive'' always means that $J_X$ is not
    surjective.
  \item[(ii)] Since $X^{**}$ is complete (Lemma \ref{lem:dual-banach}) and $J_X$ is
    isometric, it follows that a reflexive normed linear space is necessarily a Banach space.
    In other words, \emph{an incomplete normed linear space is never reflexive}; conversely, a
    Banach space need not be reflexive (Example \ref{ex:c0}).
\end{enumerate}
\end{remark}

\begin{corollary}[All higher duals are Banach spaces]\label{cor:higher-dual-banach}
For any normed linear space $X$ and every $m\ge1$, the space $X^{(m)}$ is a Banach space; in
particular $X^{*},X^{**},X^{***},\dots$ are all complete.
\end{corollary}

\begin{proof}
Induct on $m$. For $m=1$, $X^{(1)}=X^{*}$ is a Banach space by Lemma \ref{lem:dual-banach}.
If $X^{(m)}$ ($m\ge1$) is a Banach space, then $X^{(m+1)}=(X^{(m)})^{*}$ is again the dual
space of some normed linear space, so it is a Banach space by Lemma \ref{lem:dual-banach}.
\end{proof}

\subsection{The three kinds of ``isomorphism'' and their differences}\label{sub:iso-prelim}

In what follows we use both topological and linear notions of isomorphism, and reflexivity
lies precisely between the two; it is therefore necessary to distinguish them rigorously.

\begin{definition}[Homeomorphism, linear homeomorphism, isometric isomorphism]\label{def:isomorphisms}
Let $X,Y$ be normed linear spaces and let $T:X\to Y$ be a bijection.
\begin{enumerate}
  \item[(i)] If $T$ and $T^{-1}$ are both continuous, then $T$ is called a
    \emph{topological isomorphism} (homeomorphism), and $X$ and $Y$ are said to be
    \emph{topologically equivalent}, written $X\cong Y$;
  \item[(ii)] If $T$ is linear and a homeomorphism, then $T$ is called a \emph{linear
    homeomorphism} (linear isomorphism). Linearity alone is not enough: the open mapping
    theorem says that a \emph{bounded} (equivalently, continuous) linear bijection between
    Banach spaces is automatically a homeomorphism, so for Banach spaces there is no difference
    between ``linear homeomorphism'' and ``bounded linear bijection''. The word ``bounded''
    cannot be dropped: if $\{e_n\}$ is the standard unit vector basis of $\ell^2$ and one
    extends it to a Hamel basis of $\ell^2$, then the linear map $T$ sending $e_n$ to $ne_n$
    and fixing every remaining vector of that Hamel basis is an algebraic bijection of $\ell^2$
    onto itself which is unbounded and discontinuous, hence not a homeomorphism;
  \item[(iii)] If $T$ is linear and $\|Tx\|=\|x\|$ for all $x\in X$, then $T$ is called an
    \emph{isometric isomorphism} (isometric homeomorphism), written $X\equiv Y$.
\end{enumerate}
The three requirements are successively stronger: isometric isomorphism $\Longrightarrow$
linear homeomorphism $\Longrightarrow$ homeomorphism, and the converse implications fail in
general. Almost all spaces discussed in the sequel are Banach spaces, and for these the open mapping theorem supplies the continuity of the inverse, so a
\emph{bounded} linear bijection is a homeomorphism; hence, for simplicity,
from now on $\cong$ always denotes a \emph{linear homeomorphism} between Banach
spaces and $\equiv$ denotes an isometric isomorphism (in each case the operator realizing it is
bounded, hence continuous, by definition).
\end{definition}

\begin{remark}[Reflexivity and the three kinds of isomorphism]\label{rem:reflexive-vs-iso}
By Definition \ref{def:reflexive} and Lemma \ref{lem:isometry},
\[
  X\ \text{is reflexive}\iff J_X\ \text{is surjective}\iff J_X\ \text{is an isometric isomorphism},
\]
that is to say: reflexivity requires that the isometric isomorphism between $X$ and $X^{**}$
be realized by the \emph{canonical} map, and not merely that some such isomorphism ``exists''.
If one only required $X\equiv X^{**}$ (some isometric isomorphism exists), or, more weakly,
$X\cong X^{**}$ (a linear homeomorphism), the condition would be much less restrictive;
Example \ref{ex:james} (the James space) shows that these weaker conditions are not sufficient
to imply reflexivity.
\end{remark}

\subsection{Goldberg's theorem: propagation of non-reflexivity along the dual chain}\label{sub:goldberg-prelim}

The following theorem is the key tool of the whole paper. It asserts that for Banach spaces
reflexivity is preserved in both directions by the operation of ``taking duals'', so that its
contrapositive expresses the propagation of non-reflexivity.

\begin{theorem}[Goldberg \cite{goldberg1960}]\label{thm:goldberg}
Let $X$ be a Banach space. Then
\[
  X\ \text{is reflexive}\iff X^{*}\ \text{is reflexive}.
\]
\end{theorem}

\begin{proof}
($\Longrightarrow$) Suppose that $X$ is reflexive, that is, $J_X:X\to X^{**}$ is a surjective
isometric map. Take any $x^{***}\in X^{***}$ and put
\[
  x^{*}=x^{***}\circ J_X,\qquad\text{that is,}\qquad
  x^{*}(x)=x^{***}\bigl(J_X(x)\bigr),\quad x\in X .
\]
Then $x^{*}\in X^{*}$. We claim that $J_{X^{*}}(x^{*})=x^{***}$: for any
$x^{**}\in X^{**}$, since $J_X$ is surjective we may choose $x\in X$ with $x^{**}=J_X(x)$, and
then
\[
  \bigl(J_{X^{*}}(x^{*})\bigr)(x^{**})
  =x^{**}(x^{*})=(J_X(x))(x^{*})=x^{*}(x)
  =x^{***}\bigl(J_X(x)\bigr)=x^{***}(x^{**}).
\]
Hence $J_{X^{*}}$ is surjective and $X^{*}$ is reflexive.

($\Longleftarrow$) We prove the contrapositive: if $X$ is non-reflexive, then $X^{*}$ is
non-reflexive. In this case $J_X(X)\subsetneq X^{**}$. By Lemma \ref{lem:isometry} and the
completeness of $X$, the space $J_X(X)$ is a closed proper subspace of $X^{**}$. Choose
$z\in X^{**}\setminus J_X(X)$ and put $d:=\operatorname{dist}\bigl(z,J_X(X)\bigr)>0$; the
positivity of $d$ is exactly where completeness is used. On the subspace
$M:=J_X(X)+\K z$ define a linear functional $g$ by
\[
  g(m+az)=a,\qquad m\in J_X(X),\ a\in\K .
\]
This is well defined since $z\notin J_X(X)$. For $a\ne0$ and $m\in J_X(X)$ we have
$m+az=a\bigl(z-(-m/a)\bigr)$ with $-m/a\in J_X(X)$, whence
$\|m+az\|\ge|a|\,d$ and therefore
\[
  |g(m+az)|=|a|\le\frac{\|m+az\|}{d};
\]
the same estimate is trivial for $a=0$. Hence $g$ is bounded on $M$ with $\|g\|\le1/d$, and
$g\ne0$ because $g(z)=1$. By the Hahn--Banach theorem $g$ extends to a functional
$x^{***}\in(X^{**})^{*}=X^{***}$ with $\|x^{***}\|=\|g\|$; replacing $x^{***}$ by
$x^{***}/\|x^{***}\|$ we may assume $\|x^{***}\|=1$, and by construction
\[
  \|x^{***}\|=1,\qquad x^{***}\bigl(J_X(x)\bigr)=0\quad\text{for all }x\in X .
\]
Suppose that $X^{*}$ were reflexive. Then $J_{X^{*}}:X^{*}\to X^{***}$ is surjective, so there
is $x^{*}\in X^{*}$ with $x^{***}=J_{X^{*}}(x^{*})$; then for all $x\in X$,
\[
  x^{*}(x)=(J_X(x))(x^{*})=\bigl(J_{X^{*}}(x^{*})\bigr)\bigl(J_X(x)\bigr)=x^{***}\bigl(J_X(x)\bigr)=0,
\]
that is, $x^{*}=0$ and hence $x^{***}=0$, contradicting $\|x^{***}\|=1$. Therefore $X^{*}$ is
non-reflexive.
\end{proof}

\begin{remark}[Completeness cannot be dropped]\label{rem:goldberg-complete}
The two directions of Theorem \ref{thm:goldberg} do not depend on completeness in the same
way. The direction ``$\Longrightarrow$'' uses only the surjectivity of $J_X$, and holds for an
arbitrary normed linear space; the key step of the direction ``$\Longleftarrow$'' is the
statement that ``$J_X(X)$ is a \emph{closed} proper subspace of $X^{**}$'', which requires $X$
to be complete. If in the theorem the Banach space is replaced by a general normed linear
space, the contrapositive form ``$X$ non-reflexive $\Longrightarrow$ $X^{*}$ non-reflexive''
fails; see Example \ref{ex:einf}.
\end{remark}

\begin{definition}[The dual chain and its strictness]\label{def:chain}
Let $X$ be a normed linear space. Along the canonical embeddings
\[
  X^{(m)}\xrightarrow{\ J_{X^{(m)}}\ }X^{(m+2)}
\]
we identify $X^{(m)}$ with the subspace $J_{X^{(m)}}(X^{(m)})$ of $X^{(m+2)}$, and thereby
obtain the two chains
\[
\begin{tikzcd}[row sep=3.0em, column sep=3.2em]
  X \arrow[r, "J_X"] \arrow[rr, bend left=20, "J_{X^{(2)}}\circ J_X"]
    & X^{(2)} \arrow[r, "J_{X^{(2)}}"] & X^{(4)} \arrow[r, "J_{X^{(4)}}"] & \cdots \\[-4pt]
  X^{*} \arrow[r, "J_{X^{*}}"] & X^{(3)} \arrow[r, "J_{X^{(3)}}"] & X^{(5)} \arrow[r, "J_{X^{(5)}}"] & \cdots
\end{tikzcd}
\]
If every inclusion in these two chains is a proper inclusion, we say that the two
\emph{dual chains are strictly increasing}; if there is $m$ such that
$X^{(k+2)}=J_{X^{(k)}}(X^{(k)})$ for all $k\ge m$, we say that the dual chain
\emph{stabilizes} at $X^{(m)}$. (The compatibility of this identification is guaranteed by the
naturality of the canonical embedding: for $T\in\mathcal{L}(X,Y)$ one has
$J_Y\circ T=T^{**}\circ J_X$, and taking $T=J_{X^{(m)}}$ gives
$J_{X^{(m+2)}}\circ J_{X^{(m)}}=(J_{X^{(m)}})^{**}\circ J_{X^{(m)}}$, that is, ``going from
$X^{(m)}$ to $X^{(m+4)}$ via $X^{(m+2)}$'' and ``going from $X^{(m)}$ to $X^{(m+4)}$ in one
step'' are the same map.)
\end{definition}

\begin{theorem}[Non-reflexivity propagates along the dual chain]\label{thm:chain}
Let $X$ be a Banach space. If $X$ is non-reflexive, then $X^{(n)}$ is non-reflexive for every
$n\ge1$.
\end{theorem}

\begin{proof}
By Corollary \ref{cor:higher-dual-banach}, the spaces $X^{(n)}$ ($n\ge1$) are Banach spaces,
so Theorem \ref{thm:goldberg} applies to them. Induct on $n$.

Base case ($n=1$): $X^{(0)}=X$ is a Banach space and non-reflexive. If $X^{*}$ were reflexive,
then the direction ``$\Longleftarrow$'' of Theorem \ref{thm:goldberg} would give that $X$ is
reflexive, a contradiction. Hence $X^{*}$ is non-reflexive.

Inductive step: suppose that $X^{(k)}$ ($k\ge1$) is a Banach space and non-reflexive. Since
$X^{(k+1)}=(X^{(k)})^{*}$, applying Theorem \ref{thm:goldberg} to the Banach space $X^{(k)}$
gives: if $X^{(k+1)}$ were reflexive, then $X^{(k)}$ would be reflexive, contradicting the
induction hypothesis. Hence $X^{(k+1)}$ is non-reflexive.
\end{proof}

\begin{corollary}[The two chains are strictly increasing and never stabilize]\label{cor:strict-chain}
Let $X$ be a Banach space and suppose that $X$ is non-reflexive. Then, in the sense of
Definition \ref{def:chain}, the two dual chains
\[
  X\subsetneq X^{(2)}\subsetneq X^{(4)}\subsetneq\cdots,\qquad
  X^{*}\subsetneq X^{(3)}\subsetneq X^{(5)}\subsetneq\cdots
\]
are both strictly increasing. In particular, the dual chain does not ``stabilize'' at a
reflexive space.
\end{corollary}

\begin{proof}
By Theorem \ref{thm:chain}, for every $k\ge0$ the spaces $X^{(2k)}$ (including $X^{(0)}=X$)
and $X^{(2k+1)}$ are non-reflexive; and by Definition \ref{def:reflexive}, $X^{(m)}$ is
non-reflexive if and only if $J_{X^{(m)}}\bigl(X^{(m)}\bigr)\subsetneq X^{(m+2)}$. Taking
$m=2k$ and $m=2k+1$ yields the two proper inclusion chains. As for stabilization: if the chain
stabilized at $X^{(m)}$, then from some step on every canonical embedding would be surjective,
in particular $X^{(m)}$ would be reflexive, contradicting Theorem \ref{thm:chain}.
\end{proof}

\begin{corollary}\label{cor:one-reflexive}
Let $X$ be a Banach space. If $X^{(n)}$ is reflexive for some $n\ge0$, then $X$ is reflexive.
Equivalently,
\[
  X\ \text{is non-reflexive}\iff \text{all}\ X^{(n)}\ (n\ge0)\ \text{are non-reflexive}.
\]
\end{corollary}

\begin{proof}
Suppose that $X^{(n)}$ is reflexive. By Theorem \ref{thm:goldberg} we argue backwards step by
step: $X^{(n-1)}$ and $X^{(n)}$ are dual to each other, hence $X^{(n-1)}$ is reflexive;
repeating this $n$ times gives that $X=X^{(0)}$ is reflexive. Combined with Theorem
\ref{thm:chain} this yields the equivalent form.
\end{proof}

\begin{remark}[Strictness of the chain is not the same as the spaces being non-isomorphic]\label{rem:strict-vs-iso}
The word ``strict'' in Corollary \ref{cor:strict-chain} refers only to the level of the
\emph{images of the canonical embeddings}; it does not mean that the various $X^{(m)}$ are
mutually non-isomorphic. The James space $J$ of Example \ref{ex:james} satisfies
$J\equiv J^{**}$: two neighboring levels are even \emph{isometrically isomorphic}, and yet
$J_J(J)\subsetneq J^{**}$ still holds, so that the two chains are still strictly increasing and
never stabilize. Thus ``$X$ is non-reflexive'' and ``$X$ is not isomorphic to $X^{**}$''---where ``isomorphic'' may mean \emph{isometrically} isomorphic or merely
a linear homeomorphism---are two distinct conditions.
\end{remark}

\subsection{Three basic examples}\label{sub:examples-prelim}

Neither the word ``Banach'' in Theorem \ref{thm:goldberg} nor the word ``canonical'' in
Definition \ref{def:reflexive} can be relaxed. The three examples below show, in turn,
that non-reflexive Banach spaces do exist; that the theorem fails once completeness is
dropped; and that the theorem fails once ``reflexive'' is silently replaced by
``isometrically isomorphic to the second dual''. These three examples will recur in
the discussion of Section \ref{sec:corrections}.

\begin{example}[A non-reflexive Banach space: $c_0$]\label{ex:c0}
Take $X=c_0$ (the space of all sequences converging to $0$, with the supremum norm). Then
\[
  c_0^{*}=\ell^{1},\qquad c_0^{**}=(c_0^{*})^{*}=(\ell^{1})^{*}=\ell^{\infty},
\]
while $J_{c_0}(c_0)=c_0$ is a proper subspace of $\ell^{\infty}$, so $c_0$ is non-reflexive.
By the contrapositive form of Theorem \ref{thm:goldberg}, $\ell^{1}$ and $\ell^{\infty}$ are
non-reflexive as well \cite[Chapter IV, Exercise 2]{rudin1991} (for a general textbook
account see also \cite{conway1990}).
\end{example}

\begin{example}[Dropping completeness: $E^{\infty}$]\label{ex:einf}
Take
\[
  X=E^{\infty}=\bigl\{\,x=(\xi_1,\xi_2,\dots)\ \bigm|\ \text{only finitely many }\xi_k\ne0\,\bigr\},
\]
endowed with the $\ell^{2}$ norm $\|x\|=\bigl(\sum_{k\ge1}|\xi_k|^{2}\bigr)^{1/2}$.
Then $X$ is a dense linear subspace of $\ell^{2}$; in particular, $X$ is incomplete; by Remark
\ref{rem:reflexive-points}(ii), $X$ is non-reflexive. On the other hand:
\begin{enumerate}
  \item[(1)] $X^{*}\equiv(\ell^{2})^{*}=\ell^{2}$. Indeed, take any $f\in X^{*}$; by the
    Hahn--Banach theorem there is $F\in(\ell^{2})^{*}$ with $\|F\|=\|f\|$ and $F|_{X}=f$; since
    $X$ is dense in $\ell^{2}$, the $F\in(\ell^{2})^{*}$ satisfying $F|_{X}=f$ is unique, so
    $f\mapsto F$ is an isometric isomorphism $X^{*}\to(\ell^{2})^{*}$.
  \item[(2)] By (1), $X^{*}\equiv\ell^{2}$ is reflexive while $X$ is non-reflexive; hence
    ``$X$ non-reflexive $\Longrightarrow$ $X^{*}$ non-reflexive'' fails for incomplete $X$.
  \item[(3)] Moreover $X^{**}\equiv\ell^{2}$, while $J_X(X)=X$ is dense but not closed in
    $\ell^{2}$, hence $J_X(X)\subsetneq X^{**}$, which is consistent with the
    non-reflexivity of $X$.
\end{enumerate}
This example shows that the completeness hypothesis in Theorem \ref{thm:goldberg} cannot be
omitted.
\end{example}

\begin{example}[Isometrically isomorphic to the second dual without being reflexive: the James space]\label{ex:james}
James \cite{james1951} constructed the following real Banach space $J$:
\[
  J=\Bigl\{\,x=(\xi_1,\xi_2,\dots)\ \Bigm|\ \lim_{n\to\infty}\xi_n=0,\ \|x\|<+\infty\,\Bigr\},
\]
where
\[
  \|x\|=\sup\Bigl[\sum_{i=1}^{n}\bigl(\xi_{k_{2i-1}}-\xi_{k_{2i}}\bigr)^{2}
  +\xi_{k_{2n+1}}^{2}\Bigr]^{1/2},
\]
the supremum being taken over all $n\ge1$ and all tuples of positive integers
$k_1<k_2<\cdots<k_{2n+1}$. The space $(J,\|\cdot\|)$ is a Banach space, and
\begin{enumerate}
  \item[(1)] $J$ is separable and $\{e_n\}$ (where $e_n$ has $1$ in the $n$-th position and
    $0$ elsewhere) is a shrinking monotone basis of $J$;
  \item[(2)] $\{e_n\}$ is not boundedly complete, so $J$ is non-reflexive;
  \item[(3)] $J$ is \emph{isometrically isomorphic} to $J^{**}$, and all conjugate spaces of
    $J$ are separable.
\end{enumerate}
By (3), there exists an isometric isomorphism between $J$ and $J^{**}$, but none of them is the
canonical embedding: by (2) and Definition \ref{def:reflexive},
$J_J(J)\subsetneq J^{**}$. This shows that ``$X$ is isometrically isomorphic to $X^{**}$''
never implies ``$X$ is reflexive''; reflexivity concerns only the \emph{canonical embedding}
(Remark \ref{rem:reflexive-vs-iso}).
\end{example}

%=====================================================================
\subsection{Categorical background: direct and inverse limits}\label{sub:cat-prelim}

The two directions of the dual chain yield a direct system and an inverse system respectively;
hence the question ``what is the limit of the chain'' depends first of all on ``in which
category the limit is taken''. Since readers coming from analysis may not be familiar with
category theory (for a systematic account see \cite{maclane1998}), in this subsection we list the
notions we need in the form ``concrete category $+$ explicit construction'', and use the
language of universal properties only where necessary.

\begin{definition}[Categories, functors and natural transformations]\label{def:category}
A \emph{category} $\mathcal{C}$ consists of the following data: a class of objects
$\Ob(\mathcal{C})$; for every pair of objects $X,Y$ a set of morphisms
$\mathcal{C}(X,Y)$; a composition law
$\circ:\mathcal{C}(Y,Z)\times\mathcal{C}(X,Y)\to\mathcal{C}(X,Z)$ which is associative; and for
every object $X$ an identity morphism $\id_X\in\mathcal{C}(X,X)$ which acts as a unit for
composition.

Let $\mathcal{C},\mathcal{D}$ be categories. A \emph{covariant functor}
$F:\mathcal{C}\to\mathcal{D}$ sends an object $X$ to an object $F(X)$ and a morphism
$f\in\mathcal{C}(X,Y)$ to $F(f)\in\mathcal{D}\bigl(F(X),F(Y)\bigr)$, and preserves identities
and composition:
\[
  F(\id_X)=\id_{F(X)},\qquad F(g\circ f)=F(g)\circ F(f).
\]
A \emph{contravariant functor} $F:\mathcal{C}\to\mathcal{D}$ assigns to $f:X\to Y$ a morphism
$F(f):F(Y)\to F(X)$ and satisfies $F(g\circ f)=F(f)\circ F(g)$, together with preservation of
identity morphisms.

Let $F,G:\mathcal{C}\to\mathcal{D}$ be functors and let
$\eta=(\eta_X)_{X\in\Ob(\mathcal{C})}$ be a family of morphisms with
$\eta_X:F(X)\to G(X)$. If $\eta_Y\circ F(f)=G(f)\circ\eta_X$ holds for every
$f\in\mathcal{C}(X,Y)$, that is, if the following diagram commutes,
\[
\begin{tikzcd}[column sep=3.6em,row sep=2.6em]
  F(X) \arrow[r,"F(f)"] \arrow[d,"\eta_X"'] & F(Y) \arrow[d,"\eta_Y"]\\
  G(X) \arrow[r,"G(f)"'] & G(Y)
\end{tikzcd}
\]
then $\eta$ is called a \emph{natural transformation} from $F$ to $G$, written
$\eta:F\rightsquigarrow G$; if every $\eta_X$ is an isomorphism, then $\eta$ is called a
\emph{natural isomorphism}.
\end{definition}

\begin{definition}[Direct and inverse systems, direct and inverse limits]\label{def:limits}
Let $\langle I,\le\rangle$ be a directed set (that is, $\le$ is a preorder and any two elements
have an upper bound); in this paper only the case $I=\N$ (with the natural order) is used, and
below we fix the index set $I$ for ease of exposition.

\emph{(Direct systems and direct limits.)} A \emph{direct system} (or inductive system) in
$\mathcal{C}$ is a family of objects $(X_i)_{i\in I}$ together with a family of morphisms
$\bigl(f_{ij}:X_i\to X_j\bigr)_{i\le j}$ satisfying
\[
  f_{ii}=\id_{X_i},\qquad f_{ik}=f_{jk}\circ f_{ij}\quad (i\le j\le k).
\]
An object $X$ together with a family of morphisms $f_i:X_i\to X$ ($i\in I$) is called a
\emph{cocone} of this direct system if $f_j\circ f_{ij}=f_i$ ($i\le j$). The cocone $(X,f_i)$
is called the \emph{direct limit} of the direct system (also called inductive limit, forward
limit or colimit) if it satisfies the following \emph{universal property}: for every cocone
$(Y,g_i)$ there is a unique morphism $u:X\to Y$ with $u\circ f_i=g_i$ for all $i\in I$:
\[
\begin{tikzcd}[column sep=small,row sep=small]
  X_i \arrow[r,"f_i"] \arrow[dr,"g_i"'] & X \arrow[d,dashed,"\exists!\,u"]\\
  & Y
\end{tikzcd}
\]
If a direct limit exists, it is unique up to isomorphism and is denoted by
$\varinjlim_iX_i$.

\emph{(Inverse systems and inverse limits.)} Now reverse all arrows. An \emph{inverse
system} (or projective system) in $\mathcal{C}$ is a family of objects $(Y_i)_{i\in I}$
together with a family of morphisms $\bigl(\rho_{ij}:Y_j\to Y_i\bigr)_{i\le j}$ satisfying
\[
  \rho_{ii}=\id_{Y_i},\qquad \rho_{ik}=\rho_{ij}\circ\rho_{jk}\quad (i\le j\le k).
\]
An object $Y$ together with a family of morphisms $\rho_i:Y\to Y_i$
($\rho_{ij}\circ\rho_j=\rho_i$, $i\le j$) is called a \emph{cone}. The cone $(Y,\rho_i)$ is
called the \emph{inverse limit} of the inverse system (also called the projective limit)
if for every cone $(Z,g_i)$ there is a unique $u:Z\to Y$ with
$\rho_i\circ u=g_i$ for all $i\in I$:

\[
\begin{tikzcd}[column sep=small,row sep=small]
  Z \arrow[d,dashed,"\exists!\,u"'] \arrow[dr,"g_i"] & \\
  Y \arrow[r,"\rho_i"'] & Y_i
\end{tikzcd}
\]
If an inverse limit exists, it too is unique up to isomorphism and is denoted by
$\varprojlim_iY_i$.
\end{definition}

\begin{remark}[The intuitive difference between the two kinds of limits]\label{rem:limit-picture}
A direct limit ``glues a family of objects into a single whole along the arrows''; when the
connecting maps are injective---which is the case throughout this paper---it is realized as the
increasing union $\bigcup_iX_i$ (\S\ref{sub:cat-prelim}), and is therefore in general larger than
each of its members, the differences between the members being ``joined together'' in the process
of gluing. An inverse limit ``consists of the mutually compatible elements of a family of
objects'', that is, of the coherent threads running through the whole system; there is no
\emph{a priori} size restriction on such a family of threads, and in particular an inverse limit
need \emph{not} be smaller than each of its members. The inverse limit of the finite-dimensional
spaces $\C^{n}$ along the truncation maps is $\ell^{\infty}$ (Remark \ref{rem:separable-caveat}),
which is vastly larger than each $\C^{n}$. Whether an inverse limit is comparable in size with the
members of the system has to be decided case by case, and does not follow from the universal
property alone. We shall see below that the
two directions of the dual chain give exactly these two constructions (\S\ref{sub:tower}), and
that the two are dual to each other.
\end{remark}

\begin{example}[Direct limits in vector spaces]\label{ex:colim-vect}
In the category of vector spaces over $\K$, the direct limit can be realized as the quotient of
the disjoint union $\bigsqcup_{i\in I}X_i=\bigcup_i\bigl(X_i\times\{i\}\bigr)$ by the
equivalence relation
\[
  (x_i,i)\sim(x_j,j)\iff \exists\,k\ge i,j:\ f_{ik}(x_i)=f_{jk}(x_j),
\]
the linear operations on the equivalence classes being defined by ``first push both elements to
the same level and then add them'', and with $f_i(x_i):=[(x_i,i)]$. When all the $f_{ij}$ are
injective (which is the case in this paper), the direct limit can be viewed intuitively as the
increasing union $\bigcup_iX_i$, and the equivalence classes are the ``eventual images'' of the
elements.
\end{example}

\begin{definition}[The three categories of Banach spaces]\label{def:ban-categories}
Let
\[
  \Ban:\ \text{objects are Banach spaces, morphisms are all bounded linear operators};
\]
\[
  \Ban_1:\ \text{same objects, morphisms are contractions, i.e. } \|T\|\le1;
\]
\[
  \iBan:\ \text{same objects, morphisms are isometric linear maps (not required to be surjective)}.
\]
The three categories have the same class of objects while their classes of morphisms are
successively smaller; they are thus three different ``resolutions'' of the same universe of
objects. In the literature $\Ban_1$ and $\iBan$ are also frequently written $(\BF)_1$ and
$\iBF$ (the $\mathbf i$ standing for ``isometric''); if the completeness requirement is
dropped, one writes correspondingly $\NF$ (normed spaces) and $\iNF$ (isometric normed spaces).
\end{definition}

\begin{theorem}[Existence of limits in $\Ban_1$ {\cite{semadeni1965,semadeni1971}}]\label{thm:ban1-limits}
In $\Ban_1$, every direct system indexed by a directed set has a direct limit, and every
inverse system has an inverse limit.
\end{theorem}

\begin{remark}[Why not take the limit in $\Ban$]\label{rem:why-ban1}
The universal property of the direct limit is equivalent to
\begin{equation}\label{eq:colim-hom}
  \mathcal{L}_1\Bigl(\varinjlim_i X_i,\ Y\Bigr)\cong\varprojlim_i\mathcal{L}_1(X_i,Y),
  \qquad \mathcal{L}_1(X,Y):=\{T\in\mathcal{L}(X,Y):\|T\|\le1\},
\end{equation}
that is: ``contractions out of the direct limit'' are exactly ``compatible families of
contractions''. The completion of an increasing union does have the required factorization
property for compatible families of \emph{contractions}, and more generally for compatible
families of operators carrying a common norm bound. In the category $\Ban$, however, a
compatible family of bounded operators need not have uniformly bounded norms \emph{a priori}.
For example, let $X_n=(\K^{n},\|\cdot\|_\infty)$ with zero-padding inclusions
$i_{n,n+1}(x_1,\dots,x_n)=(x_1,\dots,x_n,0)$ and define
\begin{equation}\label{eq:no-uniform-bound}
  f_n(x_1,\ldots,x_n)=\sum_{j=1}^n jx_j \qquad(n\ge1).
\end{equation}
Each $f_n$ is bounded and the family $(f_n)$ is compatible, namely $f_m\circ i_{n,m}=f_n$ for
$n\le m$; but the induced functional $f$ on the algebraic union $c_{00}$ satisfies $f(e_j)=j$ and
therefore admits no bounded extension to $c_0$. Thus the completed union $c_0$ does not
satisfy the colimit universal property in $\Ban$ for this compatible family. Note that the
structure maps of the system into $c_0$ \emph{remain isometric}: what fails is the required
bounded factorization of an arbitrary compatible family of test maps, not the isometry of the
connecting maps. This is precisely why, when dealing with
(co)limits of Banach spaces, the literature invariably restricts the category to $\Ban_1$
(and, when necessary, further to the isometric subcategory $\iBan$): contractivity
automatically provides the required uniform norm bound (for a systematic discussion of
(co)limits and local presentability of the category of Banach spaces, see \cite{castillo2010}).
The dual chain of this paper goes one step further: all connecting maps are isometries, so that
one can compute directly in $\iBan$, see Lemma \ref{lem:isometric-colim}.
\end{remark}

\begin{lemma}[Direct limits of isometric direct systems]\label{lem:isometric-colim}
Let $(X_n,\kappa_{n,m})$ be an isometric direct system in $\iBan$ indexed by $\N$. Then
\begin{enumerate}
  \item[(i)] On $\mathbb{X}:=\bigsqcup_{n\ge0}X_n\big/\sim$ (the equivalence relation being as
    in Example \ref{ex:colim-vect}, which here takes the form
    $x_n\sim x_m\iff\kappa_{n,k}(x_n)=\kappa_{m,k}(x_m)$, $k=\max\{n,m\}$) there is a unique
    linear structure for which $\kappa_{n,\infty}(x_n):=[(x_n,n)]$ is a linear map;
  \item[(ii)] $\|[(x_n,n)]\|:=\|x_n\|$ does not depend on the choice of the representative, so
    it is a norm on $\mathbb{X}$, and every $\kappa_{n,\infty}$ is isometric;
  \item[(iii)] The direct limit $\varinjlim_n X_n$ is isometrically isomorphic to the
    completion of $(\mathbb{X},\|\cdot\|)$, and $\bigcup_n\kappa_{n,\infty}(X_n)$ is dense in it.
\end{enumerate}
\end{lemma}

\begin{proof}
(i), (ii): if $n\le m$ and $x_n\sim x_m$, then $x_m=\kappa_{n,m}(x_n)$ (or conversely), and
since $\kappa_{n,m}$ is isometric we get $\|x_m\|=\|x_n\|$; hence the norm does not depend on
the choice of the representative. The linear structure is defined by ``pushing two elements to
the same level and adding them there'', and the isometry property guarantees that the result is
independent of how the elements are pushed.

(iii): let $Y$ be a Banach space and let a compatible family of contractions
$f_n:X_n\to Y$ be given (that is, $f_m\circ\kappa_{n,m}=f_n$ for $n\le m$). On the dense
subspace $\bigcup_n\kappa_{n,\infty}(X_n)$ put
$f\bigl(\kappa_{n,\infty}(x_n)\bigr):=f_n(x_n)$. By compatibility $f$ is well defined and
linear; moreover
$\bigl\|f\bigl(\kappa_{n,\infty}(x_n)\bigr)\bigr\|=\|f_n(x_n)\|\le\|x_n\|=\bigl\|\kappa_{n,\infty}(x_n)\bigr\|$,
so $f$ is a contraction and therefore extends uniquely to a contraction on the completion;
uniqueness follows from density. This is precisely the universal property of the direct limit.

If in addition the compatible cocone maps $f_n$ are all \emph{isometries}, then the induced map
$f$ preserves the norm on the algebraic union $\bigcup_n\kappa_{n,\infty}(X_n)$ and hence, by
density and continuity of $f$ and of the norm, on the completion as well. The same completed
object therefore also satisfies the colimit universal property in $\iBan$ for this isometric
direct system.
\end{proof}

\begin{remark}[Explicit description of the inverse limit]\label{rem:inv-limit-form}
For an inverse system $(Y_n,\rho_{n,m})$ in $\Ban_1$ indexed by $\N$
($\rho_{n,m}:Y_m\to Y_n$, $n\le m$), the inverse limit can be realized as
\[
  \varprojlim_n Y_n=\Bigl\{\,y=(y_n)_{n\ge0}\in\prod_{n\ge0}Y_n:\ 
  y_n=\rho_{n,m}(y_m)\ (n\le m),\ \|y\|:=\sup_n\|y_n\|<\infty\,\Bigr\},
\]
with the linear operations defined pointwise. This description is particularly convenient when
all the $\rho_{n,m}$ are isometric; when the $\rho_{n,m}$ are merely contractions the $\sup$
norm is indispensable, and the limit then consists of the bounded compatible families; no
comparison in size with the individual levels follows from the description alone. In
\S\ref{sub:tower} we shall see that the inverse system of the dual chain is
exactly of the latter kind, and that this is precisely the watershed between reflexivity and
non-reflexivity.
\end{remark}

\begin{remark}[The dual chain is an isometric direct system]\label{rem:tower-isometric}
By Lemma \ref{lem:isometry} the canonical embeddings
$J_{X^{(2m)}}:X^{(2m)}\to X^{(2m+2)}$ are all isometric, so the direct system
$\langle X^{(2n)},\kappa_{n,m}\rangle$ given by the dual chain (where
$\kappa_{n,m}:=J_{X^{(2m-2)}}\circ\cdots\circ J_{X^{(2n)}}$) is an \emph{isometric} direct
system in $\iBan$; whereas the connecting maps $\kappa^{*}_{n,m}$ of the inverse system given
by the odd-order dual chain are in general contractions and not isometries. In other words, the
two chains are different from the outset in their categorical ``shape'': the even chain is an
isometric direct system, while the odd chain is only a contractive inverse system. This is
precisely why, in what follows, we take a direct limit and an inverse limit, respectively.
\end{remark}

The chains of this section admit one further simplification. By Remark
\ref{rem:tower-isometric}, the connecting maps of the even-order dual chain are all isometric,
so that we actually work in the smaller isometric category $\iBan$; there the direct limit has
a very simple description (Lemma \ref{lem:isometric-colim}): one need not deal with seminorms;
one simply takes the completion of the increasing union, the norm being the original norm on
any one level. In the sequel, ``direct limit'' and ``inverse limit'' always mean the limits
taken in $\Ban_1$ (and, for the even chain, in $\iBan$).

\subsection{The two limit spaces}\label{sub:tower}

\subsubsection{Dual chains and the hyperdual space (direct limit)}

\begin{definition}[Dual chain]\label{def:tower-system}
Let $X$ be a Banach space. Put $X_n:=X^{(2n)}$ ($n\ge0$) and
$\kappa_{n,n+1}:=J_{X^{(2n)}}:X_n\to X_{n+1}$, and for $n\le m$ let
$\kappa_{n,m}:=\kappa_{m-1,m}\circ\cdots\circ\kappa_{n,n+1}$,
$\kappa_{n,n}:=\id_{X_n}$. The system $\langle X_n,\kappa_{n,m}\rangle$ is called the
\emph{(even-order) dual chain} of $X$.
\end{definition}

\begin{definition}[Hyperdual space]\label{def:tower-space}
Let $X$ be a Banach space. The direct limit of the dual chain,
\[
  \widetilde{X}:=\varinjlim_{n\to\infty}\bigl(X^{(2n)},\kappa_{n,m}\bigr),
\]
is called the \emph{hyperdual space} of $X$. $\widetilde X$ is the large space obtained by
gluing together the even-order duals of $X$, level by level, along the canonical embeddings.
Gwizdek \cite{gwizdek2023} writes it as $\Dir(X)$, and keeps the name ``tower (wie\.za)'' in
\cite{gwizdek2024}. We always use the notation $\widetilde X$.
\end{definition}

\begin{theorem}[Existence of the hyperdual space and its universal property]\label{thm:tower-exists}
Let $X$ be a Banach space. Then the hyperdual space $\widetilde X$ exists and is a Banach space,
$\kappa_{0,\infty}:X\to\widetilde X$ is an isometric embedding, and
$\bigcup_{n\ge0}\kappa_{n,\infty}(X^{(2n)})$ is dense in $\widetilde X$. Moreover
$\widetilde X$ is characterized by the following universal property: for every Banach space $Y$
and every family of contractions $\bigl(f_n:X^{(2n)}\to Y\bigr)$ satisfying
$f_{n+1}\circ\kappa_{n,n+1}=f_n$ ($n\ge0$), there is a \emph{unique} contraction
$f:\widetilde X\to Y$ with $f\circ\kappa_{n,\infty}=f_n$ for all $n$. In particular, a space
satisfying this universal property is unique up to isometric isomorphism.
\end{theorem}

\begin{proof}
By Remark \ref{rem:tower-isometric} the dual chain is an isometric direct system in $\iBan$,
so Lemma \ref{lem:isometric-colim} yields existence and all the assertions except uniqueness,
while the universal property is exactly what was verified in the proof of part (iii) of that
lemma. As for uniqueness, suppose that $(Z,\kappa^{Z}_{n,\infty})$ satisfies the same universal
property as $(\widetilde X,\kappa_{n,\infty})$. Applying the universal property of
$\widetilde X$ to the cocone $\kappa^{Z}_{n,\infty}$ gives a contraction
$u:\widetilde X\to Z$ with $u\circ\kappa_{n,\infty}=\kappa^{Z}_{n,\infty}$. Applying the
universal property of $Z$ to the cocone $\kappa_{n,\infty}$ gives a contraction
$v:Z\to\widetilde X$ with $v\circ\kappa^{Z}_{n,\infty}=\kappa_{n,\infty}$. Uniqueness gives
$v\circ u=\id_{\widetilde X}$ and $u\circ v=\id_Z$. Since both maps are contractions, they are
inverse isometries.
\end{proof}

\subsubsection{The odd chain and the inverse hyperdual space (inverse limit)}

\begin{definition}[Inverse hyperdual space]\label{def:inv-tower}
Let $X$ be a Banach space. The inverse limit of the inverse system
$\bigl(X^{(2n+1)},\kappa^{*}_{n,m}\bigr)_{n\le m}$,
\[
  \overleftarrow{X}:=\varprojlim_{n\to\infty}\bigl(X^{(2n+1)},\kappa^{*}_{n,m}\bigr),
\]
is called the \emph{inverse hyperdual space} of $X$; Gwizdek \cite{gwizdek2023} writes it as
$\Inv(X)$. By Remark \ref{rem:inv-limit-form} it can be written explicitly as
\[
  \overleftarrow{X}=\Bigl\{\,\varphi=(\varphi_0,\varphi_1,\dots)\ \Bigm|\ 
  \varphi_n\in X^{(2n+1)},\ \varphi_n=\kappa^{*}_{n,m}(\varphi_m)\ (n\le m),\
  \|\varphi\|=\sup_{n}\|\varphi_n\|<\infty\,\Bigr\}.
\]
\end{definition}

\begin{theorem}[The dual of the direct limit is the inverse limit]\label{thm:dir-dual-inv}
Let $X$ be a Banach space. Then there is an isometric isomorphism
\[
  \bigl(\widetilde{X}\bigr)^{*}=\Bigl(\varinjlim_{n}X^{(2n)}\Bigr)^{*}
  \equiv\varprojlim_{n}X^{(2n+1)}=\overleftarrow{X},
\]
the isomorphism being given by
$\Phi(\varphi):=\bigl(\varphi\circ\kappa_{n,\infty}\bigr)_{n\ge0}$.
More generally, for every Banach space $Y$ and every $T\in\mathcal{L}(X,Y)$ one has
$\Dir(T)^{*}=\Inv(T)$.
\end{theorem}

\begin{proof}
(The argument is taken from \cite[Theorem 2.9]{gwizdek2023}.)
Let $\varphi\in\widetilde X^{*}$ and put $\varphi_n:=\varphi\circ\kappa_{n,\infty}$. Since
$\kappa_{n+1,\infty}\circ\kappa_{n,n+1}=\kappa_{n,\infty}$, for any $x_n\in X^{(2n)}$ we have
\begin{align*}
  \bigl\langle\kappa^{*}_{n,n+1}(\varphi_{n+1}),\,x_n\bigr\rangle
  &=\bigl\langle\varphi_{n+1},\,\kappa_{n,n+1}(x_n)\bigr\rangle
  =\bigl\langle\varphi,\,\kappa_{n+1,\infty}\bigl(\kappa_{n,n+1}(x_n)\bigr)\bigr\rangle\\
  &=\bigl\langle\varphi,\,\kappa_{n,\infty}(x_n)\bigr\rangle
  =\langle\varphi_n,x_n\rangle,
\end{align*}
hence $\varphi_n=\kappa^{*}_{n,n+1}(\varphi_{n+1})$; also
$\|\varphi_n\|\le\|\varphi\|\cdot\|\kappa_{n,\infty}\|=\|\varphi\|$. Consequently
$\Phi(\varphi)\in\overleftarrow X$ and $\|\Phi(\varphi)\|\le\|\varphi\|$.

Conversely, let $\psi=(\psi_n)_{n\ge0}\in\overleftarrow X$ and define, on the dense subspace
$\bigcup_n\kappa_{n,\infty}\bigl(X^{(2n)}\bigr)$,
\[
  \Psi(\psi)\bigl(\kappa_{n,\infty}(x_n)\bigr):=\psi_n(x_n).
\]
By the compatibility $\psi_n=\kappa^{*}_{n,m}(\psi_m)$ ($n\le m$): if
$\kappa_{n,\infty}(x_n)=\kappa_{m,\infty}(x_m)$, we may assume $n\le m$, so that
$x_m=\kappa_{n,m}(x_n)$ and hence
$\psi_m(x_m)=\psi_m\bigl(\kappa_{n,m}(x_n)\bigr)
=\bigl(\kappa^{*}_{n,m}(\psi_m)\bigr)(x_n)=\psi_n(x_n)$;
thus $\Psi(\psi)$ is well defined. It is clearly linear. For $\psi$ in the inverse limit, the
induced functional on the algebraic union has norm at most $\|\psi\|$, and therefore extends
uniquely to a functional $\Psi(\psi)\in\widetilde X^{*}$ with $\|\Psi(\psi)\|\le\|\psi\|$. Thus
$\Psi$, as a map from the inverse limit to $\widetilde X^{*}$, is a contraction. Finally, $\Phi$
and $\Psi$ are mutually inverse and both are contractions, so
they give an isometric isomorphism $\widetilde X^{*}\equiv\overleftarrow X$. The assertion
about operators follows by applying the same argument level by level to $J_n(T)$; for the
details see \cite[Theorem 2.9]{gwizdek2023}.
\end{proof}

\begin{remark}[Comparison of terminology and notation]\label{rem:naming}
The two limit spaces are described differently in the literature; the following table
compares the terminology of this paper with that of \cite{gwizdek2023,gwizdek2024} (a
comparison with the terminology used in the three preprints is given in \S\ref{sec:wording}),
in order to avoid confusion.
\begin{center}
\begin{tabular}{p{0.46\textwidth}p{0.46\textwidth}}
\toprule
this paper & references \cite{gwizdek2023,gwizdek2024}\\
\midrule
(even-order) dual chain $\langle X^{(2n)},\kappa_{n,m}\rangle$ & tower system $\langle X_n,\kappa_{n,m}\rangle$\\
hyperdual space $\widetilde X$ & $\Dir(X)$\\
inverse hyperdual space $\overleftarrow X$ & $\Inv(X)$\\
\bottomrule
\end{tabular}
\end{center}
In the table, $\kappa_{n,m}$ denote the connecting maps of the dual chain (Definition
\ref{def:tower-system}), and $\Dir$, $\Inv$ are the two limit functors introduced in the next
subsection.
\end{remark}

\begin{remark}[The third natural object: the direct limit of the odd chain]\label{rem:odd-direct}
Besides the two limits above, the odd chain of duals has a further natural \emph{direct} limit
\[
  \varinjlim_{n}X^{(2n+1)}=\widetilde{X^{*}},
\]
namely the hyperdual space of $X^{*}$ (because $\bigl(X^{*}\bigr)^{(2n)}=X^{(2n+1)}$). Here the tilde is attached to $X^{*}$, whereas in the dual $\widetilde X^{*}$ of the hyperdual space of $X$ it is attached to $X$; the two notations are kept apart throughout.
The direct system on the odd duals uses the canonical embeddings $J_{X^{(2n+1)}}$, whereas the
inverse system uses the adjoints of the even-order dual chain's connecting maps. These are two different
diagrams constructed from the same sequence of objects. The inverse-limit construction gives
$\overleftarrow X\equiv\widetilde X^{*}$ (Theorem \ref{thm:dir-dual-inv}), and this
identification is with $\widetilde X^{*}$, not with the direct limit $\widetilde{X^{*}}$ above.
The main objects of this paper are $\widetilde X$ and $\overleftarrow X$; the space
$\widetilde{X^{*}}$ is mentioned only when a comparison is needed, and no identification between
$\widetilde{X^{*}}$ and $\widetilde X^{*}$ is used here.
\end{remark}

\subsection{The limit functors and their properties}\label{sub:functors}

Both limit constructions can be promoted to functors, and they are ``conjugate'' to each other
in the sense of duality.

\begin{theorem}[The limit functors {\cite{gwizdek2023}}]\label{thm:functors}
For $T\in\mathcal{L}(X,Y)$ put $J_0(T):=T$ and $J_{n+1}(T):=J_n(T)^{**}$ ($n\ge0$), and define
\[
  \Dir(T):=\varinjlim_n J_n(T):\ \widetilde X=\varinjlim_n X^{(2n)}\longrightarrow
  \varinjlim_n Y^{(2n)}=\widetilde Y,
\]
\[
  \Inv(T):=\varprojlim_n J_n(T)^{*}:\ \overleftarrow Y=\varprojlim_n Y^{(2n+1)}
  \longrightarrow\varprojlim_n X^{(2n+1)}=\overleftarrow X .
\]
Then
\begin{enumerate}
  \item[(i)] $\Dir$ is a covariant functor on $\Ban$ and $\Inv$ is a contravariant functor
    on $\Ban$, and $\|\Dir(T)\|=\|T\|$, $\|\Inv(T)\|=\|T\|$;
  \item[(ii)] there exists a natural transformation $\eta:\mathrm{Id}\rightsquigarrow\Dir$
    whose components $\eta_X(x):=[(x,0)]$ are isometric embeddings, and
    $\Dir(T)\circ\eta_X=\eta_Y\circ T$;
  \item[(iii)] (interchange of limits and duality) $\Dir(T)^{*}=\Inv(T)$; in particular
    $\widetilde X^{*}\equiv\overleftarrow X$ (Theorem \ref{thm:dir-dual-inv});
  \item[(iv)] $\Dir(X)\equiv\Dir(X^{**})$ (isometrically isomorphic); and if $X$ is reflexive,
    then $\eta_X:X\to\widetilde X$ is an isometric isomorphism, i.e. $\widetilde X\equiv X$.
\end{enumerate}
\end{theorem}

\begin{proof}
(i), (ii): assume first that $\|T\|\le1$. The functoriality of $\Dir$ comes from the universal
property of the direct limit: given $S\in\mathcal{L}(Y,Z)$ and $T\in\mathcal{L}(X,Y)$, the
operators $\Dir(ST)$ and $\Dir(S)\Dir(T)$ both factor through $\kappa^{Z}_{n,\infty}\circ
J_n(ST)$, so their equality follows from the uniqueness part of the universal property; the
contravariance of $\Inv$ reverses the arrows, and the argument is the same. In the general case
one normalizes $T$ by $\|T\|$ and then rescales. Finally, from
$\Dir(T)\circ\kappa_{0,\infty}=\kappa_{0,\infty}\circ T$ and the fact that $\kappa_{0,\infty}$
is isometric we get $\|\Dir(T)\|\ge\|T\|$, while $\|\Dir(T)\|\le\|T\|$ is immediate from the
definition, so $\|\Dir(T)\|=\|T\|$; the other equality follows from (iii) together with
$\|\Dir(T)^{*}\|=\|\Dir(T)\|$.
The second identity of (ii) is precisely the naturality assertion, and it is checked by following
$x\mapsto[(x,0)]\mapsto[(Tx,0)]$; that $\eta_X$ is an isometric embedding follows from
$\|[(x,0)]\|=\|x\|$ (Lemma \ref{lem:isometric-colim}(ii)).

(iii): apply Theorem \ref{thm:dir-dual-inv} level by level to $J_n(T)$ (see the last paragraph
of the proof of that theorem).

(iv): the direct limit of a dual chain depends only on the ``tail'' of the chain: shifting the
index by one gives an isometric isomorphism
$\varinjlim_nX^{(2n)}\equiv\varinjlim_nX^{(2n+2)}$, which is exactly
$\Dir(X)\equiv\Dir(X^{**})$. If $X$ is reflexive, then by Theorem \ref{thm:goldberg} and its
repeated use every $J_{X^{(2n)}}$ is a surjective isometry, so
$\kappa_{n,\infty}(X^{(2n)})=\kappa_{0,\infty}(X)$ holds for every $n$;
by Theorem \ref{thm:tower-exists}, $\kappa_{0,\infty}(X)$ is dense in $\widetilde X$, and since
it is complete ($X$ is complete and $\kappa_{0,\infty}$ is isometric) it is closed, hence equal
to $\widetilde X$, that is, $\eta_X$ is an isometric isomorphism.
\end{proof}

\begin{proposition}[The natural map between the even-order dual functors and the direct limit]\label{prop:colim-embed}
Let $(X_n,\kappa_{n,m})$ be an isometric direct system in $\iBan$ with direct limit
$X_\infty:=\varinjlim_nX_n$, and fix $k\ge0$. Then
$\bigl(D^{2k}[X_n],\,D^{2k}(\kappa_{n,m})\bigr)$ is again an isometric direct system in $\iBan$;
let $\lambda_n:D^{2k}[X_n]\to\varinjlim_nD^{2k}[X_n]$ denote its structure maps. There is a unique
isometric embedding
\[
  \iota_k:\ \varinjlim_n D^{2k}[X_n]\ \longrightarrow\
  D^{2k}\bigl(X_\infty\bigr)=D^{2k}\Bigl(\varinjlim_n X_n\Bigr),
\]
characterized by
\begin{equation}\label{eq:colim-embed-iota}
  \iota_k\bigl(\lambda_n x\bigr)=D^{2k}(\kappa_{n,\infty})\,x,
  \qquad x\in D^{2k}[X_n],
\end{equation}
where $\kappa_{n,\infty}:X_n\to X_\infty$ is the structure map of the direct limit $X_\infty$. For
$k\ge1$ the map $\iota_k$ need not be surjective, so the equality form of ``the even-order dual functor
commutes with the direct limit'' does not hold; for a counterexample see \S\ref{sec:cat}. For
$k=0$, the map $\iota_0:\varinjlim_nX_n\to X_\infty$ is the canonical identification of the direct
limit with itself, hence an isometric isomorphism.
\end{proposition}

\begin{proof}
Since $\kappa_{n,m}$ is isometric, so is $D^{2k}(\kappa_{n,m})$; hence
$\bigl(D^{2k}[X_n],D^{2k}(\kappa_{n,m})\bigr)$ is an isometric direct system and its direct limit
exists, with isometric structure maps $\lambda_n$ (Lemma \ref{lem:isometric-colim}).
By functoriality,
$D^{2k}(\kappa_{n,\infty})=D^{2k}(\kappa_{m,\infty})\circ D^{2k}(\kappa_{n,m})$ ($n\le m$), so
the family $\bigl(D^{2k}(\kappa_{n,\infty})\bigr)_n$ is compatible, and the universal property of
the direct limit yields a unique contraction $\iota_k$ satisfying \eqref{eq:colim-embed-iota}.
Finally, for $x\in D^{2k}[X_n]$, since $\kappa_{n,\infty}$ is isometric, $D^{2k}(\kappa_{n,\infty})$
is isometric as well (the even-order dual of an isometric embedding is again an isometric
embedding), so, by \eqref{eq:colim-embed-iota} and the isometry of $\lambda_n$,
\[
  \bigl\|\iota_k\bigl(\lambda_n x\bigr)\bigr\|
  =\bigl\|D^{2k}(\kappa_{n,\infty})(x)\bigr\|=\|x\|=\bigl\|\lambda_n x\bigr\|,
\]
i.e. $\iota_k$ is isometric on the dense subspace
$\bigcup_n\lambda_n\bigl(D^{2k}[X_n]\bigr)$, and hence is an isometric embedding.
\end{proof}

\begin{proposition}[Operator properties preserved by the limit {\cite{gwizdek2023}}]\label{prop:preserved}
Let $X$ be a Banach space.
\begin{enumerate}
  \item[(i)] (direct-sum decomposition) if $X=U\oplus V$ (a direct sum of closed subspaces),
    then $\widetilde X\cong\widetilde U\oplus\widetilde V$ (the direct sum of Banach spaces, a
    linear homeomorphism); by duality (Theorem \ref{thm:dir-dual-inv}), $\overleftarrow X$
    decomposes correspondingly as the direct sum of $\overleftarrow U$ and
    $\overleftarrow V$;
  \item[(ii)] (spectrum) for $T\in\mathcal{L}(X)$ one has
    $\sigma(T)=\sigma\bigl(\Dir(T)\bigr)=\sigma\bigl(\Inv(T)\bigr)$;
  \item[(iii)] (Fredholm property) if $X$ has the approximation property and
    $T\in\mathcal{L}(X)$ is a Fredholm operator, then $\Dir(T)$ and $\Inv(T)$ are Fredholm
    operators too, and their indices satisfy
    $i(T)=i\bigl(\Dir(T)\bigr)=-i\bigl(\Inv(T)\bigr)$.
\end{enumerate}
\end{proposition}

\begin{proposition}[Non-reflexivity does not disappear under the limit]\label{prop:limit-reflexive}
Let $X$ be a Banach space. Then
\[
  \widetilde X\ \text{is reflexive}\iff \overleftarrow X\ \text{is reflexive}\iff X\ \text{is reflexive}.
\]
Consequently, if $X$ is non-reflexive, then both limit spaces of the two dual chains are
non-reflexive.
\end{proposition}

\begin{proof}
Consider first $\widetilde X$. ``$\Longleftarrow$'': if $X$ is reflexive, then by
Theorem \ref{thm:goldberg} and its repeated use every even-order dual of $X$ is reflexive and
every canonical embedding $J_{X^{(2n)}}$ is a surjective isometry, so all the connecting maps
of the dual chain are surjective, and hence $\widetilde X\equiv X$ is reflexive
(Theorem \ref{thm:functors}(iv)).
``$\Longrightarrow$'': if $\widetilde X$ is reflexive, then by Theorem \ref{thm:tower-exists},
$\kappa_{0,\infty}$ embeds $X$ isometrically into $\widetilde X$, and its image
$\kappa_{0,\infty}(X)$ is a closed subspace of $\widetilde X$ ($X$ is complete and the image
under an isometry is complete); a closed subspace of a reflexive Banach space is reflexive, so
$\kappa_{0,\infty}(X)$ is reflexive, and hence $X$ is reflexive.

Consider now $\overleftarrow X$: by Theorem \ref{thm:dir-dual-inv} we have
$\overleftarrow X\equiv\widetilde X^{*}$, and by Theorem \ref{thm:goldberg} ($X$ reflexive
$\iff$ $X^{*}$ reflexive) together with the previous paragraph, $\widetilde X^{*}$ is reflexive
if and only if $\widetilde X$ is reflexive, if and only if $X$ is reflexive.
\end{proof}

\begin{remark}\label{rem:limit-vs-chain}
Proposition \ref{prop:limit-reflexive} shows that the phenomenon established in this section --- that
the dual chain never ``stabilizes'' --- persists after passing to the limit: even if the whole
chain is glued into a single space, no reflexive space arises. This is contrary to the naive
expectation (``the limit should be closer to being reflexive than each of its members''), but
the reason is simple: non-reflexivity is ``inherited'' by the hyperdual space from $X$ itself ($X$
is a closed subspace of it), and reflexivity is precisely what is inherited by closed subspaces.
Likewise, $\overleftarrow X$, being the dual of $\widetilde X$, is non-reflexive as well.
\end{remark}

\begin{proposition}[Separability]\label{prop:limit-separable}
Let $X$ be a Banach space. If every $X^{(2n)}$ ($n\ge0$) is separable, then $\widetilde X$ is
separable.
\end{proposition}

\begin{proof}
By Theorem \ref{thm:tower-exists}, $\bigcup_{n\ge0}\kappa_{n,\infty}\bigl(X^{(2n)}\bigr)$ is
dense in $\widetilde X$; and each $\kappa_{n,\infty}\bigl(X^{(2n)}\bigr)$ is an isometric image
of a separable space, hence separable, so it contains a countable dense subset $D_n$. Then
$D:=\bigcup_nD_n$ is countable and dense in
$\bigcup_n\kappa_{n,\infty}(X^{(2n)})$; the latter is dense in $\widetilde X$, so $D$ is dense in
$\widetilde X$.
\end{proof}

\begin{remark}\label{rem:separable-caveat}
No equally simple assertion can be made for the inverse chain space $\overleftarrow X$. On the
one hand $\overleftarrow X\equiv\widetilde X^{*}$, and the dual of a separable Banach space need
not be separable (for instance $\ell^{1}$ is separable while $(\ell^{1})^{*}=\ell^{\infty}$ is
not); on the other hand, even if every level of the inverse system is separable and all the
connecting maps are contractions, the inverse limit may still fail to be separable: take
$X_n=(\C^{n},\|\cdot\|_\infty)$ with the supremum norm
$\|(\xi_1,\dots,\xi_n)\|_\infty=\max_{1\le j\le n}|\xi_j|$, and let
$\rho_{n,m}:(\C^{m},\|\cdot\|_\infty)\to(\C^{n},\|\cdot\|_\infty)$ be the \emph{truncation}
$\rho_{n,m}(\xi_1,\dots,\xi_m)=(\xi_1,\dots,\xi_n)$; then the inverse limit
consists of the compatible families of the form $x_n=(\xi_1,\dots,\xi_n)$ (the first $n$ terms of
the same sequence $\xi$), so
\[
  \varprojlim_n(\C^{n},\|\cdot\|_\infty)\equiv\ell^{\infty},
\]
which is not separable. Thus Proposition \ref{prop:limit-separable} holds only for
$\widetilde X$; as for $\overleftarrow X$, this paper only points out that it is dual to
$\widetilde X$ (Theorem \ref{thm:dir-dual-inv}) and makes no general assertion about its
separability.
\end{remark}

\subsection{The James space: the concrete shape of the limits}\label{sub:james}

The real James space $J$ is quasi-reflexive of order one:
\[
  \dim\bigl(J^{**}\big/J_J(J)\bigr)=1 .
\]
It is non-reflexive and separable, isometrically isomorphic to its bidual, and all its
finite iterated duals are separable; the ``deficiency'' $\dim J^{**}\big/J_J(J)=1$ is minimal
(\cite[Example 2.3]{gwizdek2023}). We use these properties directly, without introducing an
additional meaning of the term ``almost reflexive''. For this
very reason the dual chain of $J$ can be computed completely, and it represents the
``slowest-growing'' end of the dual chain.

\begin{theorem}[The dual chain and the hyperdual space of the James space {\cite[Example 2.3]{gwizdek2023}}]\label{thm:james-tower}
Write $J_n:=J^{(2n)}$ ($n\ge0$), so that $J_0=J$ and the connecting maps are
$\iota_n:=\kappa_{n,n+1}=J_{J_n}:J_n\to J_{n+1}$. Up to a linear homeomorphism of each level, and with all direct sums below understood as
topological direct sums, the dual chain of the James space is given by
\[
  J_1=J^{**}=J\oplus\R u_1,\quad
  J_2=J_1^{**}=J_1\oplus\R u_2=J\oplus\R u_1\oplus\R u_2,\quad\ldots,\quad
  J_n=J\oplus\R u_1\oplus\cdots\oplus\R u_n,
\]
where (inside each $J_n$) the first direct summand is the canonical image of $J$. Hence
\[
  \widetilde J=\Dir(J)=\kappa_{0,\infty}(J)\oplus R,\qquad
  R=\overline{\Span\{u_k:k\ge1\}},
\]
where $R$ is the completion of the direct limit of the linear spans of $\{u_k:k\ge1\}$ under the
layer norms $\|\cdot\|_{J_n}$, that is, the norm of a vector $\lambda_1u_1+\cdots+\lambda_nu_n$
is computed inside $J_n$. The first summand $\kappa_{0,\infty}(J)$ is the range of a bounded
projection of $\widetilde J$; in particular $\widetilde J$ is separable.
\end{theorem}

\begin{proof}
Recall that $J$ is a \emph{real} Banach space (Example \ref{ex:james}), so all spaces below are
real and all operators are real-linear. $J$ is separable and $\kappa_{0,1}(J)=J_J(J)$ is a
closed subspace of $J^{**}=J_1$ of codimension $1$ (Example \ref{ex:james}(2)(3) and Definition
\ref{def:reflexive}). A closed subspace of finite codimension in a Banach space is complemented
by a bounded projection; equivalently, there is a bounded left inverse
$\rho_0:J_1\to J_0=J$ of $\iota_0$, that is, $\rho_0\iota_0=\id_{J_0}$, and we fix one. Choose a
nonzero $u_1\in\ker\rho_0$, so that
\[
  J_1=\kappa_{0,1}(J_0)\oplus\ker\rho_0=\kappa_{0,1}(J_0)\oplus\R u_1 .
\]
For $n\ge1$ the space $J_n$ is itself a dual space, with predual $J^{(2n-1)}$
($J_n=(J^{(2n-1)})^{*}$), and $\iota_n=J_{J_n}$ is the canonical embedding of a dual space into
its second dual. Hence the standard retraction
\[
  \rho_n:=\bigl(J_{J^{(2n-1)}}\bigr)^{*}:J_{n+1}=(J_n)^{**}\longrightarrow J_n
\]
satisfies $\rho_n\iota_n=\id_{J_n}$ and $\|\rho_n\|=1$; the identity used here is
$J_F^{*}J_{F^*}=\id_{F^*}$, with $F:=J^{(2n-1)}$, $F^{*}=J^{(2n)}=J_n$. Since $J_n\equiv J$ for
every $n$ (Example \ref{ex:james}(3)) and $\dim J^{**}\big/J_J(J)=1$, each $J_{n+1}$ is a
codimension-one extension of $J_n$, so $\ker\rho_n$ is one-dimensional; choose a nonzero
$u_{n+1}\in\ker\rho_n$, so that
\[
  J_{n+1}=\iota_n(J_n)\oplus\R u_{n+1}.
\]
Inductively, identifying $u_k$ ($k\le n$) with its image $\kappa_{k,n}(u_k)\in J_n$,
\[
  J_n=\kappa_{0,n}(J_0)\oplus\R u_1\oplus\cdots\oplus\R u_n\qquad(n\ge0),
\]
which is the asserted decomposition of the dual chain.

Set $q_0:=\id_{J_0}$ and $q_n:=\rho_0\rho_1\cdots\rho_{n-1}:J_n\to J_0$ for $n\ge1$. Then
$q_n\kappa_{0,n}=\id_{J_0}$, while the compatibility
\[
  q_{n+1}\iota_n=\rho_0\rho_1\cdots\rho_{n-1}(\rho_n\iota_n)=q_n
\]
holds because $\rho_n\iota_n=\id_{J_n}$; and
\[
  \|q_n\|\le\|\rho_0\|\prod_{k=1}^{n-1}\|\rho_k\|=\|\rho_0\|\quad(n\ge1),\qquad\|q_0\|=1 .
\]
Thus the family $(q_n)$ is compatible and uniformly bounded, with $\sup_n\|q_n\|\le C_0$, where
$C_0:=\max\{1,\|\rho_0\|\}<\infty$. By compatibility the formula
$q\bigl(\kappa_{n,\infty}(x_n)\bigr):=q_n(x_n)$ is well defined on the dense subspace
$\bigcup_n\kappa_{n,\infty}(J_n)$ of $\widetilde J$: if
$\kappa_{n,\infty}(x_n)=\kappa_{m,\infty}(x_m)$ with $n\le m$, then $x_m=\kappa_{n,m}(x_n)$ and
the compatibility gives $q_m(x_m)=q_n(x_n)$. It is linear and bounded by $C_0$, hence extends
uniquely to a bounded linear $q:\widetilde J\to J_0$ with $\|q\|\le C_0$. Since
$q\kappa_{0,\infty}=q_0=\id_{J_0}$, the map
\[
  P:=\kappa_{0,\infty}q:\widetilde J\longrightarrow\widetilde J
\]
is a bounded projection: $P^2=\kappa_{0,\infty}(q\kappa_{0,\infty})q=\kappa_{0,\infty}q=P$. Its
range is $\kappa_{0,\infty}(J_0)$ and its kernel is $\ker P=\ker q$ (here we use that
$\kappa_{0,\infty}$ is injective). Therefore
\[
  \widetilde J=\kappa_{0,\infty}(J)\oplus\ker q
\]
is a topological direct sum. This is the analytic step that upgrades the levelwise
decompositions to the limit, and it is precisely what a naive ``glue the complements together''
argument omits: the boundedness of $P$ requires a \emph{uniform} bound on the $q_n$.

It remains to identify $\ker q$ with $R$. From $q_n=q_{n-1}\rho_{n-1}$ and
$J_n=\iota_{n-1}(J_{n-1})\oplus\R u_n$ (with $u_n\in\ker\rho_{n-1}$) we get
$\ker q_n=\iota_{n-1}(\ker q_{n-1})\oplus\R u_n$, hence
$\ker q_n=\R u_1\oplus\cdots\oplus\R u_n$. Consequently
$\ker q\cap\kappa_{n,\infty}(J_n)=\kappa_{n,\infty}(\ker q_n)\subset\Span\{u_1,\dots,u_n\}$,
which gives $\Span\{u_k:k\ge1\}\subset\ker q$; since $\ker q$ is closed, also
$\overline{\Span\{u_k:k\ge1\}}\subset\ker q$. Conversely, let $z\in\ker q$ and choose $x_m$ in the union of the layers with $x_m\to z$. For
each $m$, choose $n_m$ with $x_m\in\kappa_{n_m,\infty}(J_{n_m})$ and put
$y_m:=(\id_{\widetilde J}-P)x_m$. Since
$Px_m\in\kappa_{0,\infty}(J)\subseteq\kappa_{n_m,\infty}(J_{n_m})$, we have
\[
  y_m\in\kappa_{n_m,\infty}(\ker q_{n_m})\subseteq\Span\{u_1,\ldots,u_{n_m}\};
\]
moreover
\[
  \|y_m-z\|=\bigl\|(\id_{\widetilde J}-P)(x_m-z)\bigr\|\le(1+\|P\|)\,\|x_m-z\|\longrightarrow0 .
\]
Thus $z\in\overline{\Span\{u_k:k\ge1\}}$, so $\ker q=R$, which completes the description of the
complement.

Separability: each $J_n$ is separable (Example \ref{ex:james}(3): all the conjugate spaces of
$J$ are separable), so the dense subspace $\bigcup_n\kappa_{n,\infty}(J_n)$ of Lemma
\ref{lem:isometric-colim}(iii) is a union of countably many separable spaces, and hence
$\widetilde J$ is separable.
\end{proof}

By Proposition \ref{prop:limit-reflexive}, both $\widetilde J$ and
$\overleftarrow J\equiv\widetilde J^{*}$ are non-reflexive. The James space thus ties together
the two threads of this section:

(i) the dual chain itself never ``stabilizes'' (\S\ref{sub:goldberg-prelim}--\S\ref{sub:examples-prelim}):
every higher-order dual of $J$ is non-reflexive;

(ii) passing to the direct limit produces no reflexive space either (Proposition
\ref{prop:limit-reflexive}): both $\widetilde J$ and $\overleftarrow J$ are non-reflexive;

(iii) adjacent layers are even \emph{isometrically} isomorphic ($J\equiv J^{**}$), yet the dual chain
system uses the \emph{canonical} embedding, and each step upward adds one new direction: since
$u_1\notin\kappa_{0,1}(J)$, its image $\kappa_{1,\infty}(u_1)$ lies outside
$\kappa_{0,\infty}(J)$. Thus $\kappa_{0,\infty}(J)$ is a proper subspace of $\widetilde J$. Both spaces are
nevertheless infinite-dimensional and separable, and hence both have Hamel dimension $\cc$;
in particular the proper inclusion must not be read as ``a larger Hamel dimension'' (see
\S\ref{sub:dim-basic}).

This shows that ``$X$ is isometrically isomorphic to $X^{**}$'' does not mean that
non-reflexivity is ``smoothed out'' along the chain --- the distinction emphasized in this
section between the ``canonical embedding'' and ``the existence of some isometric isomorphism''
remains crucial at the level of limits as well.
On the other hand, the example of $J$ also shows that a hyperdual space may grow ``only a little'':
$\widetilde J$ is merely the direct sum of $J$ and a separable space $R$, and by Proposition
\ref{prop:limit-separable} (all conjugate spaces of $J$ are separable) it is still separable.

\begin{remark}[The other extreme: $C(K)$]\label{rem:CK}
The extreme opposite of the James space is the space of continuous functions. By
\cite[Example 2.4]{gwizdek2023}, for an infinite compact Hausdorff space $K$ every level of the
dual chain of $C(K)$ is again the space of continuous functions on some (successively larger)
compact Hausdorff space (this is precisely Dixmier's classical correspondence
$C(K)^{**}\equiv C(K')$, see \cite{dixmier1951}; for finer descriptions of the elements of
$C(X)^{**}$ and of the ideal structure see \cite{kaplan1959}). To describe the limits correctly,
however, one must keep track of the connecting maps between the layers. Write
$A_n:=C(K)^{(2n)}$ and fix an isometric identification $A_n\equiv C(K_n)$, where $K_n$ denotes
the Gelfand spectrum of the commutative unital $C^{*}$-algebra $A_n$. The canonical embedding
$\iota_n:=\kappa_{n,n+1}:A_n\to A_{n+1}$ is a unital injective $*$-homomorphism, so under
Gelfand duality it corresponds to a \emph{continuous surjection} $p_n:K_{n+1}\to K_n$, and on
functions the embedding is the pull-back
\[
  \iota_n(f)=f\circ p_n .
\]
The arrows on function algebras and on their spectra thus point in opposite directions. Hence the
hyperdual space of $C(K)$ is the space of continuous functions on the \emph{inverse limit} of the
spectra along the connecting maps, and \emph{not} on the Stone--\v{C}ech compactification of a
disjoint union:
\[
  \widetilde{C(K)}=\Dir\bigl(C(K)\bigr)\equiv C(K_\infty),\qquad
  K_\infty=\varprojlim(K_n,p_n).
\]
Indeed, if $\pi_n:K_\infty\to K_n$ denotes the corresponding coordinate projection, then each $\pi_n$ is
surjective (because the $p_n$ are), so $f\mapsto f\circ\pi_n$ is an isometric embedding of
$C(K_n)$ into $C(K_\infty)$, and the images form an increasing subalgebra
$\mathcal A:=\bigcup_n\pi_n^{*}C(K_n)\subset C(K_\infty)$ that contains the constants, is closed
under complex conjugation, and separates the points of $K_\infty$: two distinct compatible
sequences differ in some coordinate, and on the corresponding compact space the two values are
separated by a continuous function. By the Stone--Weierstrass theorem
$\overline{\mathcal A}=C(K_\infty)$, and this is exactly the completed direct limit of the dual chain.

For comparison, the operation produced by the description above, namely the
Stone--\v{C}ech compactification of the disjoint union of the layers, corresponds instead to
choosing a function on each layer \emph{independently}:
\[
  C\Bigl(\beta\Bigl(\textstyle\bigsqcup_n K_n\Bigr)\Bigr)
  \equiv\Bigl\{(f_n):f_n\in C(K_n),\ \sup_n\|f_n\|<\infty\Bigr\},
\]
a bounded product in which the identifications $\iota_n$ between the layers play no role. The
simplest example is the constant system $\C\to\C\to\cdots$: its direct limit is $\C$, whereas the
disjoint union of the one-point spectra is $\N$ and its Stone--\v{C}ech compactification gives
$C(\beta\N)=\ell^{\infty}$.

Finally, the hyperdual space of $C(K)$ is still far from being separable when $K$ is infinite, and this
conclusion has a direct proof that does not use the above description of the underlying space.
Choose countably many distinct points $t_1,t_2,\dots\in K$ and put
$Y:=\overline{\Span}\{\delta_{t_1},\delta_{t_2},\dots\}\subset C(K)^{*}$; choosing a continuous
function of sup-norm $1$ with prescribed values at finitely many distinct points shows that
$Y$ is isometrically isomorphic to $\ell^1$. The map $P\mu:=\sum_{i\ge1}\mu(\{t_i\})\delta_{t_i}$ is a
projection of $C(K)^{*}$ onto $Y$ of norm $1$ (indeed $P\delta_{t_i}=\delta_{t_i}$ and
$\|P\mu\|\le\sum_i|\mu(\{t_i\})|\le\|\mu\|$), so passing to adjoints exhibits an isometric,
complemented copy of $\ell^{\infty}$ inside $C(K)^{**}\equiv C(K')$. Since $C(K)^{**}$ is a
level of the dual chain and embeds isometrically into $\widetilde{C(K)}$, the latter is non-separable.
Thus $J$ and $C(K)$ represent, respectively, the ``slowest'' and the ``fastest'' growth
patterns of the dual chain (Chapter 3 of \cite{gwizdek2024} discusses this in
detail).
\end{remark}

\subsection{\texorpdfstring{Basic facts about the Hamel dimension of (hyper)dual spaces}
{Basic facts about the Hamel dimension of (hyper)dual spaces}}\label{sub:dim-basic}

The preceding subsections dealt with the ``shape'' of the dual chain (the two limit spaces and
the limit functors). This subsection deals with one of its quantitative aspects: how the
\emph{Hamel dimension (algebraic dimension)} changes under repeated dualization. In what
follows $\dim E$ always denotes the Hamel dimension of the linear space $E$, and we write
$\mathfrak c=2^{\aleph_0}$. This subsection collects only the \emph{basic} facts that are actually
used later; the dimension comparison theorem and its relation to the assertions in the
literature are postponed to \S\ref{sec:dim}.

\begin{lemma}[Upper bound for the dimension]\label{lem:dim-bound}
Let $X$ be an infinite-dimensional normed space. Then
\begin{equation}\label{eq:dim-bound}
  \dim X^{*}\ \le\ |X^{*}|\ \le\ \mathfrak c^{\dim X}\ =\ 2^{\dim X}.
\end{equation}
\end{lemma}

\begin{proof}
The first inequality is obvious (the Hamel dimension does not exceed the cardinality). Take a
Hamel basis $B$ of $X$; then every vector of $X$ is a linear combination of finitely many
vectors of $B$, so $\varphi\in X^{*}$ is uniquely determined by its restriction $\varphi|_B$,
i.e. $\varphi\mapsto\varphi|_B$ is an injection $X^{*}\to\K^{B}$. Hence
$|X^{*}|\le|\K^{B}|=\mathfrak c^{|B|}=\mathfrak c^{\dim X}$; and from $\dim X\ge\aleph_0$ we get
\[
  \mathfrak c^{\dim X}=(2^{\aleph_0})^{\dim X}=2^{\aleph_0\cdot\dim X}=2^{\dim X}.\qedhere
\]
\end{proof}

\begin{lemma}[The dimension of an infinite-dimensional Banach space is its cardinality]\label{lem:dim-card}
Let $E$ be an infinite-dimensional Banach space. Then $\dim E=|E|$; in particular
$\dim E\ge\mathfrak c$ and $\dim E\ne\aleph_0$.
\end{lemma}

\begin{proof}
The inequality $\dim E\le|E|$ is clear, since a Hamel basis of $E$ is a subset of $E$. For the
reverse inequality we give a self-contained argument (the statement is due to Halbeisen and
Hungerb\"uhler \cite{halbeisen-hungerbuehler2000}, whose opening remarks make the same point
as the Baire argument below). The Baire category theorem rules out $\dim E=\aleph_0$: a
countable Hamel basis $\{e_1,e_2,\dots\}$ would exhibit
$E=\bigcup_{n\ge1}\Span\{e_1,\dots,e_n\}$ as a countable union of finite-dimensional, hence
closed and nowhere dense, subspaces, contradicting completeness. This yields only the strict
inequality $\dim E>\aleph_0$; the passage to ``at least $\mathfrak c$'' needs a separate
argument.

We construct a linearly independent family of cardinality $\mathfrak c$. We claim that there
are unit vectors $u_n\in E$ and functionals $f_n\in E^{*}$ satisfying
$f_n(u_m)=\delta_{nm}$. Indeed, suppose $u_0,\dots,u_{n-1}$ and $f_0,\dots,f_{n-1}$ have been
chosen; the common kernel $\bigcap_{k<n}\ker f_k$ has codimension at most $n$, hence is
infinite-dimensional (as $E$ is) and in particular contains a unit vector $u_n$. The vectors
$u_0,\dots,u_n$ are linearly independent, so putting $f_n(u_k):=0$ for $k<n$ and
$f_n(u_n):=1$ defines a bounded functional on $\Span\{u_0,\dots,u_n\}$, which by Hahn--Banach
has an extension to $E$; this gives the claim.

Now for $t\in(0,1)$ the series
\[
  x_t:=\sum_{n\ge0}2^{-n}t^{n}u_n
\]
converges absolutely in $E$, because $\|2^{-n}t^{n}u_n\|=2^{-n}t^{n}$ and $t<1$. If distinct
$t_1,\dots,t_m\in(0,1)$ satisfy a finite linear relation
$\sum_{j=1}^{m}a_jx_{t_j}=0$, then applying $f_0,\dots,f_{m-1}$ gives the Vandermonde system
\[
  \sum_{j=1}^{m}a_jt_j^{n}=0\qquad(n=0,1,\dots,m-1).
\]
Its matrix $(t_j^{n})_{0\le n\le m-1,\ 1\le j\le m}$ has determinant
$\prod_{i<j}(t_j-t_i)\ne0$, so $a_1=\cdots=a_m=0$. Thus $\{x_t:t\in(0,1)\}$ is a linearly
independent family of cardinality $\mathfrak c$, whence $\dim E\ge\mathfrak c$.

Finally, for an infinite-dimensional vector space $E$ over $\K$ one always has
$|E|=\max\{\mathfrak c,\dim E\}$: the inequality $|E|\ge\max\{\mathfrak c,\dim E\}$ is clear
(a Hamel basis of $E$, and the multiples of a nonzero vector, are subsets of $E$), while every
vector is a finite linear combination of basis vectors, so $E$ injects into the set of pairs
consisting of a finite subset of a Hamel basis and a $\K$-valued function on it; that set has
cardinality $\dim E\cdot\mathfrak c=\max\{\mathfrak c,\dim E\}$. Since $\dim E\ge\mathfrak c$, this yields
$|E|=\dim E$; in particular $\dim E\ge\mathfrak c$ and $\dim E\ne\aleph_0$.
\end{proof}

\begin{lemma}[Cardinality and density]\label{lem:cardinality-density}
Let $E$ be an infinite-dimensional Banach space and put $\nu=\dens(E)$. Then
\[
|E|=\nu^{\aleph_0},\qquad
\dim E=\nu^{\aleph_0}.
\]
\end{lemma}

\begin{proof}
Choose a dense subset $D\subset E$ with $|D|=\nu$. Let $\mathcal C_D\subset D^{\N}$ be the set of sequences that converge in $E$. The limit map $\mathcal C_D\to E$ is surjective, and therefore $|E|\leq\nu^{\aleph_0}$.

For the reverse inequality, choose a maximal $1/2$-separated subset $S$ of the unit sphere of $E$. Its closed linear span is $E$: otherwise Riesz's lemma provides a unit vector at distance greater than $1/2$ from that closed span, contradicting maximality. The set $S$ is infinite, and finite linear combinations with coefficients in a countable dense subfield of the scalar field give $\dens(E)\leq|S|$. Conversely, the open balls of radius $1/4$ about the points of $S$ are disjoint; each meets $D$, so $|S|\leq|D|=\nu$. Thus write $S=\{u_\alpha:\alpha<\nu\}$.

For $s\in\nu^{\N}$ define
\[
F(s)=\sum_{n=0}^{\infty}10^{-n}u_{s(n)}.
\]
The series converges in $E$. If $s$ and $t$ first differ at index $m$, then
\[
\|F(s)-F(t)\|
\geq10^{-m}\left(\frac12-\frac29\right)>0.
\]
Hence $F$ is injective, and $|E|\geq\nu^{\aleph_0}$. This proves the cardinality formula. The dimension formula follows from Lemma \ref{lem:dim-card}.
\end{proof}

\begin{remark}\label{rem:dim-noaleph0}
The fact that $\dim E\ne\aleph_0$ is a property of Banach spaces: an \emph{incomplete} normed
space may have countably infinite Hamel dimension (for instance the space of finitely supported
sequences $c_{00}$, see the analogous space in Example \ref{ex:einf}).
Therefore statements such as ``taking the dual only raises the dimension from $\aleph_0$ to
$2^{\aleph_0}$'' can hold only for incomplete spaces; \S\ref{sec:dim} will show that it indeed
fails in the complete case.
\end{remark}

\begin{lemma}[\texorpdfstring{The dimension of the dual of $\ell^{\infty}$}{The dimension of the dual of l-infinity}]\label{lem:linf-dual}
$\dim\bigl(\ell^{\infty}(\mathbb N)\bigr)^{*}=2^{\mathfrak c}$.
\end{lemma}

\begin{proof}
\emph{Lower bound.} By the classical correspondence $(\ell^{\infty})^{*}=\ba(\N)$ (the
space of bounded finitely additive measures on $\mathbb N$, see \cite{yosida1952}), for every
ultrafilter $U\in\beta\mathbb N$ put
\[
  \mu_U(A):=\lim\nolimits_{U}\chi_A=\begin{cases}1,&A\in U,\\ 0,&A\notin U,\end{cases}
  \qquad A\subseteq\mathbb N,
\]
so that $\mu_U$ is a $0$-$1$-valued finitely additive measure. Distinct ultrafilters give
distinct $\mu_U$; moreover, for distinct $U_1,\dots,U_m$ one can choose pairwise disjoint sets
$A_1,\dots,A_m$ with $A_i\in U_i$, so that $\mu_{U_i}(A_j)=\delta_{ij}$; hence
$\{\mu_U\}_{U\in\beta\mathbb N}$ is linearly independent (here we use that finitely many distinct
ultrafilters are separated by the atoms of the finite Boolean algebra generated by those
ultrafilters). From the classical result
$|\beta\mathbb N|=2^{2^{\aleph_0}}=2^{\mathfrak c}$ \cite{engelking1989} we get
$\dim(\ell^{\infty})^{*}\ge2^{\mathfrak c}$.

\emph{Upper bound.} $|\ell^{\infty}(\mathbb N)|=\mathfrak c^{\aleph_0}=\mathfrak c$, so by
Lemma \ref{lem:dim-card} we get $\dim\ell^{\infty}(\mathbb N)=\mathfrak c$, and then by inequality
\eqref{eq:dim-bound},
$\dim(\ell^{\infty})^{*}\le2^{\dim\ell^{\infty}}=2^{\mathfrak c}$.
\end{proof}

The same two ingredients (Lemma \ref{lem:linf-dual} for the first dual, and the duality
$(\ell^{\infty})^{*}=\ba(\N)$) determine the \emph{second} dual as well; we isolate this as a
proposition because it is used repeatedly below.

\begin{proposition}[\texorpdfstring{The dimension of the bidual of $\ell^{\infty}$}{The dimension of the bidual of l-infinity}]\label{prop:linf-bidual}
$\dim\bigl(\ell^{\infty}(\mathbb N)\bigr)^{**}=2^{2^{\mathfrak c}}$.
\end{proposition}

\begin{proof}
Write $E:=\ell^{\infty}(\N)$, so that $E^{*}=\ba(\N)$ and $E^{**}=\ba(\N)^{*}$.

\emph{Upper bound.} Apply \eqref{eq:dim-bound} to the Banach space $\ba(\N)$, whose dimension is
$\dim\ba(\N)=\dim(\ell^{\infty})^{*}=2^{\cc}$ by Lemma \ref{lem:linf-dual}:
\[
  \dim E^{**}=\dim\ba(\N)^{*}\ \le\ 2^{\dim\ba(\N)}=2^{2^{\cc}} .
\]
(One must not apply \eqref{eq:dim-bound} to $E=\ell^{\infty}$ here: that only bounds the
\emph{first} dual $E^{*}$ by $2^{\cc}$ and says nothing about $E^{**}$.)

\emph{Lower bound.} For $U\in\beta\N$ let $\mu_U\in\ba(\N)$ be the $0$-$1$-valued measure from the
proof of Lemma \ref{lem:linf-dual}. For distinct $U_1,\dots,U_m$ choose pairwise disjoint sets
$A_j\in U_j$; then $\mu_{U_j}(A_i)=\delta_{ij}$. For all scalars $a_1,\dots,a_m$, since the
$\mu_{U_j}$ are $0$-$1$-valued, the partition $A_1,\dots,A_m,\ \N\setminus\bigcup_jA_j$ of $\N$
and the triangle inequality give
\[
  \Bigl\|\sum_{j}a_{j}\mu_{U_{j}}\Bigr\|=\sum_{j}|a_{j}| .
\]
Thus the linear map $\ell^{1}(\beta\N)\to\ba(\N)$, $e_U\mapsto\mu_U$, is an isometric embedding
(its closed linear span is an isometric copy of $\ell^{1}(\beta\N)$). Its transpose is therefore
a linear surjection $\ba(\N)^{*}\to\ell^{1}(\beta\N)^{*}=\ell^{\infty}(\beta\N)$: indeed, the
transpose of an isometric embedding is surjective, since every bounded functional on the closed
subspace $\ell^{1}(\beta\N)$ extends by Hahn--Banach to a bounded functional on $\ba(\N)$. For a
surjective linear map $V\to W$, one has $\dim W\le\dim V$, so with
$|\beta\N|=2^{\cc}$ \cite{engelking1989} and Lemma \ref{lem:dim-card},
\[
  \dim E^{**}=\dim\ba(\N)^{*}\ \ge\ \dim\ell^{\infty}(\beta\N)=|\ell^{\infty}(\beta\N)|
  =\cc^{2^{\cc}}=2^{2^{\cc}} .
\]
Combining the two bounds gives $\dim E^{**}=2^{2^{\cc}}$.
\end{proof}

At this point the basic theory of dual chains (the two limit spaces, the two limit functors,
one example that can be computed completely, and the basic facts about dimension) has been fully
set up. The next section returns with this toolkit to the three preprints, giving
counterexamples and correct substitute statements one by one; the correct inequalities for the
change of dimension are given in \S\ref{sec:dim}.

\section{A discussion of, and corrections to, three papers of Ugle\v{s}i\'c}\label{sec:corrections}

\subsection{The dimension theorem}\label{sec:dim}

The notation of this subsection is the same as in Section \ref{sec:prelim}: $\dim$ always denotes the Hamel dimension (algebraic dimension), $X^{(n)}=D^{n}(X)$, $\cc=2^{\aleph_0}$, $\ba(\N)$ denotes the space of bounded finitely additive measures on $\N$, and the superscript $\cdot^{\N}$ always denotes the countably infinite product. Throughout this subsection we use the basic facts of Subsection \ref{sub:dim-basic}: the dimension bound \eqref{eq:dim-bound}, Lemma \ref{lem:dim-card} (the dimension of an infinite-dimensional Banach space equals its cardinality), Lemma \ref{lem:linf-dual} ($\dim(\ell^{\infty})^{*}=2^{\cc}$) and Proposition \ref{prop:linf-bidual} ($\dim(\ell^{\infty})^{**}=2^{2^{\cc}}$).

The assertions concerning dimension in the three preprints are concentrated in Theorem 4, Theorem 5, Corollary 4, Corollary 5 and Remark 3 of [U2], in Lemma 4(ii) and Question 1 of [U1], and in Theorem 2(iii) of [U3]. Behind all of these assertions lies a single picture:
\begin{quote}
\emph{$\dim X^{*}$ is determined by $\dim X$.}
\end{quote}
This subsection shows that this picture is false, and supplies the correct replacement. The
\emph{density character} $\dens(X)$ yields the general upper bound $\dim X^{*}\le|X^{*}|\le\cc^{\dens(X)}$, which is sharp in the two extreme examples but does not \emph{determine} $\dim X^{*}$. In particular, the Hamel dimension of $X$ does not determine the exact value of $\dim X^{*}$, although in the infinite-dimensional case it supplies the bounds $\dim X\le\dim X^{*}\le2^{\dim X}$. Following the uniform pattern of this section, we first quote the original statements (\S\ref{sub:dim-claims}), then give counterexamples (\S\ref{sub:dim-counter}), then write down the correct quantitative picture (\S\ref{sub:dim-correct}), and finally explain which conclusions are unaffected (\S\ref{sub:dim-unaffected}).

\subsubsection{The original assertions}\label{sub:dim-claims}

We first fix two terms.

\begin{definition}[$\dim^{*}$-stable and bidual-like]\label{def:dim-stable}
A normed space $X$ is called $\dim^{*}$-\emph{stable} if $\dim X^{*}=\dim X$.
Following the terminology of [U1] and [U3], a Banach space $X$ is called \emph{bidual-like} (or \emph{bidualic}) if there exists a bounded linear isomorphism $X\to X^{**}$; the stronger condition that such an isomorphism can be chosen to be isometric is called \emph{parareflexivity} in [U3]. Reflexivity requires the particular canonical map $J_X$ to be surjective.
\end{definition}

This paper does not develop shape theory. In order to give the discussion below a precise starting point, we recall the definition of the invariant in question and the two implications that [U1] asserts for it.

\begin{definition}[finite quotient shape, {[U1]}]\label{def:shape}
The \emph{finite quotient shape} $\Sh_{\bar0}(X)$ of a normed space $X$ is the inverse system whose objects are the finite-dimensional quotients $X/F$ of $X$---here $F$ ranges over the closed subspaces of finite codimension and $X/F$ carries the quotient norm---and whose morphisms are the continuous linear maps; it is taken up to isomorphism in the corresponding pro-category. The indexing set consists of the closed finite-codimensional subspaces of $X$, ordered by reverse inclusion; if $F\subseteq G$, the bonding map is the quotient map $X/F\to X/G$. The ambient category has finite-dimensional normed spaces as objects and all continuous linear maps as morphisms, with no uniform bound imposed on the norms of the component maps. The invariant is the isomorphism class of this inverse system. Thus $\Sh_{\bar0}(X)=\Sh_{\bar0}(Y)$ means that these two inverse systems are isomorphic.
\end{definition}

The two implications asserted in [U1] are the following.
\begin{description}
  \item[(H1)] ([U1, Theorem 2 (i)]) If $\dim X=\dim Y$, then $\Sh_{\bar0}(X)=\Sh_{\bar0}(Y)$;
    this direction is asserted to require \emph{no} additional hypothesis.
  \item[(H2)] ([U1, Theorem 2 (ii)]) The converse implication
    $\Sh_{\bar0}(X)=\Sh_{\bar0}(Y)\Rightarrow\dim X=\dim Y$ requires additional hypotheses:
    either $\max\{\dim X,\dim Y\}\le\aleph_0$, or $\min\{\dim X,\dim Y\}>\aleph_0$ and
    \emph{each} of $X,Y$ is separable or bidual-like
    (the original condition is $X,Y\in\Ob(\text{sB}\cup\text{bB})$, where sB and bB denote
    respectively separable Banach spaces and bidual-like Banach spaces).
\end{description}

As will become clear below, \emph{(H1) is false as stated}: the invariant that classifies finite quotient shape is not the Hamel dimension of the space itself but the Hamel dimension of its continuous dual, and two spaces of equal dimension may well have duals of different dimensions. It is therefore convenient to replace (H1) and (H2) at once by the following exact criterion, which we prove from Definition \ref{def:shape} and use throughout. (The converse (H2), on the other hand, is \emph{true} under its stated hypotheses, and the criterion below in fact proves it; see Corollary \ref{cor:shape-H2}.)

\begin{theorem}[classification of finite quotient shape]\label{thm:shape-classification}
For normed spaces $X,Y$ one has
\begin{equation}\label{eq:shape-criterion}
  \Sh_{\bar0}(X)=\Sh_{\bar0}(Y)\quad\Longleftrightarrow\quad \dim X^{*}=\dim Y^{*}.
\end{equation}
Consequently (H1) fails: it is \emph{not} true that $\dim X=\dim Y$ implies $\Sh_{\bar0}(X)=\Sh_{\bar0}(Y)$.
\end{theorem}

\begin{proof}
Let $F$ run over the closed subspaces of $X$ of finite codimension, so that $X/F$ runs over the finite-dimensional quotients of $X$.

($\Rightarrow$) Dualizing a finite-dimensional quotient gives a finite-dimensional subspace of the dual, namely $(X/F)^{*}\cong F^{\perp}\subseteq X^{*}$, and $F\subseteq F'$ corresponds to $F'^{\perp}\subseteq F^{\perp}$; thus dualizing turns the inverse system of finite-dimensional quotients into a \emph{direct system of finite-dimensional subspaces of $X^{*}$} whose bond maps are the transposes of the quotient maps. Its \emph{algebraic} direct limit is all of $X^{*}$: every $f\in X^{*}$ vanishes on $\Ker f$ and therefore factors as $f=\bar f\circ\pi$ through the quotient $X/\Ker f$; that quotient is finite-dimensional because $f$ has finite rank, and indeed $f\in(\Ker f)^{\perp}$, which is one of the layers. Hence an isomorphism of the inverse systems $\Sh_{\bar0}(X)=\Sh_{\bar0}(Y)$ induces, after dualizing and passing to the algebraic direct limit, a linear isomorphism $X^{*}\to Y^{*}$; in particular $\dim X^{*}=\dim Y^{*}$.

($\Leftarrow$) Conversely, let $u:X^{*}\to Y^{*}$ be a linear isomorphism. For a finite-dimensional subspace $G\subseteq X^{*}$ put $G_{\perp}:=\{x\in X:\ g(x)=0\ \text{for all }g\in G\}$. Then $G\mapsto G_{\perp}$ is inverse to $F\mapsto F^{\perp}$: it is a bijection between the finite-dimensional subspaces of $X^{*}$ and the closed finite-codimensional subspaces of $X$, with $X/G_{\perp}\cong G^{*}$. Since $u$ maps finite-dimensional subspaces to finite-dimensional subspaces and preserves inclusions, it carries the direct system $\{F^{\perp}\}$ to the direct system $\{u(F^{\perp})\}$, compatibly with the inclusions. For such a $G$, put $u_G:=u|_G:G\to u(G)$; under the natural identifications $X/G_{\perp}\cong G^{*}$ and $Y/(u(G))_{\perp}\cong(u(G))^{*}$ we use the inverse of the adjoint $u_G^{*}:(u(G))^{*}\to G^{*}$ to obtain a compatible isomorphism $X/G_{\perp}\to Y/(u(G))_{\perp}$, which is automatically continuous because its domain is finite-dimensional. Hence $\Sh_{\bar0}(X)=\Sh_{\bar0}(Y)$.
\end{proof}

\begin{corollary}\label{cor:shape-H2}
If $\Sh_{\bar0}(X)=\Sh_{\bar0}(Y)$, then $\dim X=\dim Y$ under either of the following hypotheses: (i) $X$ and $Y$ are normed spaces of at most countable Hamel dimension; or (ii) $X$ and $Y$ are infinite-dimensional Banach spaces, each of which is separable or bidual-like.
\end{corollary}

\begin{proof}
Theorem \ref{thm:shape-classification} gives $\dim X^{*}=\dim Y^{*}$. In case (i), the dual of a finite-dimensional space has the same finite dimension, whereas the dual of a countably infinite-dimensional normed space has Hamel dimension $\cc$. Thus the two spaces have the same finite dimension or are both countably infinite-dimensional. In case (ii), a separable infinite-dimensional Banach space $E$ satisfies $\dim E=\dim E^{*}=\cc$ (Proposition \ref{prop:dim-correct}(ii)). For a bidual-like Banach space, $\dim E=\dim E^{**}$ and Lemma \ref{lem:dens-dual} applied to $E$ and $E^{*}$ gives
\[
  \dim E\ \le\ \dim E^{*}\ \le\ \dim E^{**}=\dim E .
\]
Hence each of $X$ and $Y$ has the same dimension as its own dual, and the result follows.
\end{proof}

Note also that the additional hypotheses in (H2) are necessary, not redundant: since a single finite quotient shape class can contain both a space of dimension $\aleph_0$ and a space of dimension $\cc$ (see \S\ref{sub:dim-counter}), the implication in (H2) does fail once both hypotheses are dropped.

The original assertions are as follows (the English text is quoted verbatim from the original; the location is given in brackets).

\begin{pquote}
\textbf{[U2, Theorem 4]} (p.~22): ``Every normed vectorial space $X$ having (algebraic)
$\dim X\neq\aleph_0$ is $\dim^*$-stable. Especially, every Banach space is
$\dim^*$-stable.''\\[4pt]
\textbf{[U2, Theorem 5]} (p.~23): ``For every $X\in\Ob(\NF)$,
$\dim X^*=2^{\dim X}>\dim X\iff\dim X=\aleph_0$. Equivalently,
$\dim X^*=\dim X\iff\dim X\neq\aleph_0$. Therefore, for every $n\in\N$,
$D^n(X)$ is $\dim^*$-stable, $\dim D^n(X)=\dim X^*$ and
$\Sh_{\bar0}(D^n(X))=\Sh_{\bar0}(X)$.''\\[4pt]
\textbf{[U2, Corollary 4]} (p.~23): ``For every $X\in\Ob(\NF)$, the following properties are
equivalent: (i) $X$ is $\dim^*$-stable; (ii) $(\forall k\in\{0\}\cup\N)$
$D^k(X)$ is $\dim^*$-stable, i.e., $\dim D^{k+1}(X)=\dim D^k(X)$.''\\[4pt]
\textbf{[U2, Corollary 5 (ii)]} (p.~24): ``$|L(X,Y)|=\dim L(X,Y)=\dim D^n(X)
=\dim X=|X|$, $n\in\N$, whenever $\dim X>\aleph_0$.''\\[4pt]
\textbf{[U2, Remark 3]} (p.~24): ``$|X^{**}|=\dim X^{**}=\dim\Cl(X)=\dim X=|X|$,
whenever $\dim X>\aleph_0$.''\\[4pt]
\textbf{[U2, Lemma 4]} (p.~19): ``Let $X,Y\in\Ob(\NF)$ such that
$\Sh_{\bar0}(X)=\Sh_{\bar0}(Y)$. If $Y$ is $\dim^*$-stable, i.e., $\dim Y^*=\dim Y$,
then so is $X^*\equiv D(X)$ and $\dim X^{**}=\dim X^*=\dim Y$.''\\[4pt]
\textbf{[U3, Theorem 2(iii)]} (p.~15): ``$\dim\widetilde X=\dim X$ ($\dim X\neq\aleph_0$);
$\dim\widetilde X=2^{\aleph_0}$ ($\dim X=\aleph_0$).''\\[4pt]
\textbf{[U1, Lemma 4 (ii)]} (p.~11): ``Further, if $X=(V,\|\cdot\|)$ is a separable
or bidual-like normed space ($X^{**}\cong X$), then \dots\ (ii)
$|W(X)|=|Z_{\mathrm{cl}}(X)|=|X|=\dim X$.''\\[4pt]
\textbf{[U1, Question 1 and the assertion following it]} (p.~12): ``Does Lemma 4 (ii) hold true for
every normed (Banach) space $X$ having $\dim X\neq\aleph_0$? Namely, the author can
prove that $X=\ell^\infty$ (non-separable and non-bidual-like) is an example towards
the affirmative answer.''
\end{pquote}

The proof of [U3, Theorem 2(iii)] consists of only two sentences (p.~15): if $\dim X<\infty$ one may take $\widetilde X=X$; if $\dim X=\infty$, one invokes [U2, Theorem 5] and points out that ``a countable union (of quotients) in $\VF$ cannot raise an infinite dimension''. For an increasing sequence of vector spaces, the Hamel dimension of the algebraic union is the supremum of the dimensions of its layers; in particular, a countable direct limit of spaces of a common infinite dimension has dimension at most that common dimension. The failure in the dimension argument of [U3, Theorem 2(iii)] lies in its use of the incorrect dimension-stability assertion from [U2, Theorem 5]. For a completed direct limit there is a further issue: completion may increase Hamel dimension, as Remark \ref{rem:beth-tower} demonstrates.

\subsubsection{Counterexamples}\label{sub:dim-counter}

The first counterexample requires only one space.

\begin{theorem}[first counterexample]\label{thm:dim-counter}
Let $X=\ell^{\infty}(\N)$. Then $X$ is a Banach space with $\dim X=\cc\ne\aleph_0$, yet
\begin{equation}\label{eq:linf-notstable}
  \dim X^{*}=\dim(\ell^{\infty})^{*}=2^{\cc}\ \ne\ \cc=\dim X .
\end{equation}
Consequently [U2, Theorem 4] is false, and so are the ``$\Rightarrow$'' direction of the first $\iff$ in [U2, Theorem 5] and the ``Equivalently'' sentence that follows it.
\end{theorem}

\begin{proof}
$\dim\ell^{\infty}=|\ell^{\infty}|=\cc^{\aleph_0}=\cc$ (Lemma \ref{lem:dim-card} together with $\dim\ell^\infty\ge\cc$), while $\dim(\ell^{\infty})^{*}=2^{\cc}$ (Lemma \ref{lem:linf-dual}). Combining the two gives \eqref{eq:linf-notstable}. Thus $\ell^\infty$ is a Banach space that is not $\dim^{*}$-stable, so the conclusion of Theorem 4 (``every Banach space is $\dim^*$-stable'') fails; moreover $\dim X^*=2^{\dim X}>\dim X$ holds while $\dim X=\aleph_0$ does not, hence the ``$\Rightarrow$'' direction of the first $\iff$ in Theorem 5 fails, and its ``Equivalently'' sentence (namely $\dim X^*=\dim X\iff\dim X\ne\aleph_0$) fails with it.
\end{proof}

The same computation also refutes two further assertions of [U2]. Substitute \eqref{eq:linf-notstable} into [U2, Corollary 5 (ii)] (which asserts, for $\dim X>\aleph_0$, that $\dim D^{n}(X)=\dim X=|X|$; take $X=\ell^\infty$, $n=1$) and into [U2, Remark 3] (which asserts $|X^{**}|=\dim X^{**}=\dim X=|X|$; for $X=\ell^\infty$ we have $\dim X^{**}=2^{2^{\cc}}$ by Proposition \ref{prop:linf-bidual}, while $\dim X=|X|=\cc$); both assertions then fail.

The failure of the last sentence of [U2, Theorem 5], ``$\dim D^n(X)=\dim X^*$ for all $n$'', occurs even one level lower, and can be seen with $c_0$.

\begin{example}[second counterexample: $c_0$ and $\ell^{1}$]\label{ex:dim-chain}
The first five terms of the dual chain of $c_0$---each arrow denoting the passage from a space to its dual, \emph{not} the canonical embedding $J$ (which would skip every second level)---are
\[
  c_0\ \longrightarrow\ \ell^{1}\ \longrightarrow\ \ell^{\infty}
  \ \longrightarrow\ \ba(\N)\ \longrightarrow\ \ell^{\infty**}\ \longrightarrow\ \cdots ,
\]
and the corresponding Hamel dimensions are, in order,
\[
  \cc,\quad \cc,\quad \cc,\quad 2^{\cc},\quad 2^{2^{\cc}} .
\]
Here $\dim c_0=\dim\ell^{1}=\dim\ell^{\infty}=\cc$, while $\dim\ba(\N)=\dim(\ell^{\infty})^{*}=2^{\cc}$ (Lemma \ref{lem:linf-dual}). The dimension of the fifth term, namely $\dim\ell^{\infty**}=\dim\ba(\N)^{*}$, is also determined: it equals $2^{2^{\cc}}$ (Proposition \ref{prop:linf-bidual})---the exponent is $2^{\cc}$ and not $\cc$ because the dimension formula \eqref{eq:dim-bound} has to be applied to $\ba(\N)$ and \emph{not} to $\ell^{\infty}$. Thus $c_0$ is $\dim^{*}$-stable, but $D^{2}(c_0)=\ell^{\infty}$ is not (Theorem \ref{thm:dim-counter}), so (i)$\Rightarrow$(ii) of [U2, Corollary 4] fails for $c_0$; and $\dim D^{3}(c_0)=\dim\ba(\N)=2^{\cc}\ne\cc=\dim c_0^{*}$, so the last sentence of Theorem 5 fails at $n=3$. Likewise $\ell^{1}$ is stable while $D^{2}(\ell^{1})=\ba(\N)$ has dimension $2^{\cc}$.
\end{example}

Note that this example is complementary to the previous one: what $\ell^{\infty}$ refutes is ``stability'', while what $c_0$ refutes is ``preservation of stability along iteration''. Together they show that the dimension may either remain constant or jump inside the dual chain, and that \emph{the place where it jumps cannot be predicted from $\dim X$}.

The third counterexample concerns the equality of [U3, Theorem 2(iii)]. Here no controversial argument is needed: all that is required is the fact already proved in Subsection \ref{sub:tower}: the embeddings of the layers into the hyperdual space are \emph{isometric} (Theorem \ref{thm:tower-exists}), and an isometric embedding does not decrease the Hamel dimension.

\begin{proposition}[lower bound for the dimension of the hyperdual space]\label{prop:dim-tower-lower}
Let $X$ be a Banach space and let $\widetilde X$ be its hyperdual space (Definition \ref{def:tower-space}).
Then
\begin{equation}\label{eq:dim-tower-lower}
  \dim\widetilde X\ \ge\ \sup_{n\ge0}\ \dim X^{(2n)} .
\end{equation}
\begin{proof}
By Theorem \ref{thm:tower-exists}, for every $n$ there is an isometric embedding $\kappa_{n,\infty}:X^{(2n)}\to\widetilde X$, which as an injective linear map does not decrease the dimension, that is, $\dim\widetilde X\ge\dim X^{(2n)}$; taking the supremum over $n$ gives \eqref{eq:dim-tower-lower}.
\end{proof}
\end{proposition}

\begin{example}[third counterexample: the hyperdual space of $c_0$]\label{ex:dim-tower}
Take $X=c_0$. By Proposition \ref{prop:linf-bidual}, $c_0^{(4)}=\ell^{\infty**}$ has dimension $2^{2^{\cc}}$, so \eqref{eq:dim-tower-lower} gives
\[
  \dim\widetilde{c_0}\ \ge\ \dim c_0^{(4)}\ =\ 2^{2^{\cc}}\ >\ \cc=\dim c_0 ,
\]
contradicting the assertion of [U3, Theorem 2(iii)] that $\dim\widetilde X=\dim X$ (when $\dim X\ne\aleph_0$). Passing to the completion (our convention, see Definition \ref{def:tower-space}) does not affect this conclusion: the isometric embedding and the dimension lower bound hold equally for the completed $\widetilde X$.
\end{example}

\begin{remark}[the dual chain of $\ell^{\infty}$ and the beth numbers]\label{rem:beth-tower}
Put $X=\ell^\infty(\mathbb N)$ and define
\[
\beth_0=\aleph_0,\qquad \beth_{n+1}=2^{\beth_n},\qquad
\beth_\omega=\sup_{n<\omega}\beth_n.
\]
Let $\Gamma_0=\mathbb N$, and let $\Gamma_{k+1}=\operatorname{Ult}(\Gamma_k)$ be the set of all ultrafilters on $\Gamma_k$. When the Stone--\v{C}ech construction is applied at the next stage, $\Gamma_k$ is again given the discrete topology. This is thus an iteration on sets, rather than repeated compactification of an already compact space. The standard ultrafilter cardinality formula yields
\[
|\Gamma_{k+1}|=2^{2^{|\Gamma_k|}},\qquad |\Gamma_k|=\beth_{2k}.
\]

For an infinite discrete set $\Gamma$, identify $\ell^\infty(\Gamma)$ with $C(\beta\Gamma)$. In the measure space $M(\beta\Gamma)$, the point masses span an isometric copy of $\ell^1(\operatorname{Ult}(\Gamma))$. It is the range of the contractive atomic-part projection
\[
\mu\longmapsto\sum_{t\in\beta\Gamma}\mu(\{t\})\delta_t.
\]
For each finite measure, only countably many coefficients in this sum are nonzero, and their absolute sum is bounded by $\|\mu\|$. Write $Q_\Gamma$ for the atomic-part projection and $\iota_\Gamma:\ell^1(\operatorname{Ult}(\Gamma))\hookrightarrow M(\beta\Gamma)$ for the inclusion, so that $Q_\Gamma\iota_\Gamma=\id$ and $\|Q_\Gamma\|\le1$; a contraction with a right inverse is a quotient map, and the transpose of a quotient map is isometric. Passing to adjoints therefore exhibits an isometric, norm-one complemented copy of $\ell^\infty(\operatorname{Ult}(\Gamma))$ in $\ell^\infty(\Gamma)^{**}$, namely the range of the isometric embedding $Q_\Gamma^{*}$, a range retracted by the contraction $\iota_\Gamma^{*}$.

A norm-one complemented isometric embedding remains such after passing to biduals: if $\iota:E\to F$ is isometric and $P:F\to E$ is a contraction with $P\iota=\id_E$, then $\iota^{**}:E^{**}\to F^{**}$ is again isometric and $P^{**}\iota^{**}=(P\iota)^{**}=\id_{E^{**}}$ with $\|P^{**}\|=\|P\|\le1$. Starting with $X$ itself and repeating the preceding construction therefore gives norm-one complemented isometric embeddings
\[
\ell^\infty(\Gamma_k)\longrightarrow X^{(2k)},\qquad
\ell^\infty(\Gamma_k)^{*}\longrightarrow X^{(2k+1)},
\]
the second being the transpose of the contraction that retracts the first.
For an infinite set $\Gamma$, the binary-valued functions give a $1$-separated subset of $\ell^\infty(\Gamma)$ of cardinality $2^{|\Gamma|}$. Thus
\[
\dens(\ell^\infty(\Gamma))
=|\ell^\infty(\Gamma)|
=\dim\ell^\infty(\Gamma)=2^{|\Gamma|}.
\]
The ultrafilter measures, by the argument of Lemma \ref{lem:linf-dual} applied to $\Gamma$, give
\[
\dim\ell^\infty(\Gamma)^*\geq2^{2^{|\Gamma|}}.
\]
It follows that $\dim X^{(2k)}\geq\beth_{2k+1}$ and $\dim X^{(2k+1)}\geq\beth_{2k+2}$. The upper bound of Lemma \ref{lem:dim-bound}, applied at every step and starting with $\dim X=\beth_1$, gives the matching inequalities. Consequently
\[
\dim X^{(n)}=\beth_{n+1}\qquad(n\geq0).
\]
The same complemented embeddings show that
\[
\dens(X^{(2k)})\geq\beth_{2k+1}.
\]
The reverse inequality follows from $\dens(X^{(2k)})\leq|X^{(2k)}|=\dim X^{(2k)}$, so equality holds.

Since the layers embed isometrically into $\widetilde X$ via the structure maps $\kappa_{k,\infty}$ and have dense union,
\[
\dens(\widetilde X)
=\sup_k\dens(X^{(2k)})=\beth_\omega.
\]
Here the lower bound uses the hereditary density inequality for metric spaces, and the upper bound follows by taking the union of dense subsets of the layers. Lemma~\ref{lem:cardinality-density} therefore yields
\[
\dim\widetilde X=(\beth_\omega)^{\aleph_0}.
\]
To evaluate this expression, put $\nu=\beth_\omega$. Partition a set $A$ of cardinality $\nu$ into countably many sets $A_n$ of cardinality less than $\nu$. Since $\nu$ is a strong limit cardinal, $|\mathcal P(A_n)|<\nu$ for every $n$. Restriction to the pieces gives
\[
2^\nu=|\mathcal P(A)|\leq\prod_n|\mathcal P(A_n)|
\leq\nu^{\aleph_0}\leq(2^\nu)^{\aleph_0}=2^\nu.
\]
Thus
\[
\dim\widetilde{\ell^\infty}=2^{\beth_\omega}=\beth_{\omega+1}.
\]
By contrast, the uncompleted increasing union has Hamel dimension
\[
\dim U_X=\sup_k\dim X^{(2k)}=\beth_\omega.
\]
Completion therefore strictly increases the Hamel dimension in this example.
\end{remark}

The next counterexample concerns [U2, Lemma 4]. The classification criterion of Theorem \ref{thm:shape-classification} allows us to choose spaces satisfying its finite-quotient-shape hypothesis, and shows that the lemma already fails within the class of separable Banach spaces.

\begin{proposition}[{[U2, Lemma 4]} fails]\label{prop:lemma4-fail}
Take $X=\ell^{1}(\N)$ and $Y=\ell^{2}(\N)$. Then the hypotheses of [U2, Lemma 4] are satisfied---$\Sh_{\bar0}(X)=\Sh_{\bar0}(Y)$, and $Y$ is $\dim^{*}$-stable---but its conclusion fails: $X^{*}=\ell^{\infty}$ is \emph{not} $\dim^{*}$-stable, and
\[
  \dim X^{**}=\dim(\ell^{\infty})^{*}=2^{\cc}\ne\cc=\dim Y=\dim X^{*}.
\]
\end{proposition}

\begin{proof}
Both $X$ and $Y$ are infinite-dimensional and separable, and $\dim X^{*}=\dim\ell^{\infty}=\cc$ while $\dim Y^{*}=\dim\ell^{2}=\cc$ (Lemma \ref{lem:dim-card}), so $\dim X^{*}=\dim Y^{*}$; by the criterion \eqref{eq:shape-criterion} this is exactly the statement $\Sh_{\bar0}(X)=\Sh_{\bar0}(Y)$, which is therefore satisfied. And $(\ell^{2})^{*}\cong\ell^{2}$ shows that $Y$ is $\dim^{*}$-stable. On the other hand $\dim X^{*}=\dim\ell^{\infty}=\cc$, so $X^{*}$ is not $\dim^{*}$-stable by Theorem \ref{thm:dim-counter}; equivalently, $\dim X^{**}=\dim(\ell^{\infty})^{*}=2^{\cc}\ne\cc=\dim X^{*}$. Finally $\dim Y=\dim\ell^{2}=\cc$, so $\dim X^{**}=2^{\cc}\ne\cc=\dim Y$. Thus both parts of the conclusion ``$X^{*}$ is $\dim^{*}$-stable and $\dim X^{**}=\dim X^{*}=\dim Y$'' of Lemma 4 fail. Note also that separability does not repair Lemma 4: both spaces here are separable Banach spaces.
\end{proof}

\begin{remark}[{[U1, Lemma 4 (ii)]} and the answer to Question 1]\label{rem:Z-count}
We now examine the cardinality assertion in [U1, Lemma 4 (ii)] directly. This counting argument requires no assumption about equality of finite quotient shapes. Writing
\[
  \mathcal Z(X):=\bigl\{Z\trianglelefteq X:\ Z\ \text{closed},\ \dim Z=\dim X,\
  \dim(X/Z)<\aleph_0\bigr\},
\]
note that the definition of $\mathcal Z(X)$ allows the codimension to be \emph{any} finite number, so besides closed hyperplanes it also contains closed subspaces of codimension $>1$, and one cannot count only hyperplanes as the original argument does; the correct approach is to use annihilators. Put $\mathcal F:=\{F\trianglelefteq X^{*}:\dim F<\aleph_0\}$. For $Z\in\mathcal Z(X)$, the inequality $\dim(X/Z)<\aleph_0$ implies that $Z^{\perp}\cong(X/Z)^{*}$ is finite-dimensional, and the closedness of $Z$ gives $Z=(Z^{\perp})^{\perp}$, so $Z\mapsto Z^{\perp}$ is an injection $\mathcal Z(X)\hookrightarrow\mathcal F$; moreover every $F\in\mathcal F$ is spanned by some finite tuple $(f_{1},\dots,f_{n})\in(X^{*})^{n}$, whence $|\mathcal F|\le\aleph_{0}\cdot|X^{*}|=|X^{*}|$ (as $|X^{*}|$ is infinite). Thus
\[
  N\ \le\ |\mathcal Z(X)|\ \le\ |\mathcal F|\ \le\ |X^{*}| ,
\]
where $N$ is the number of one-dimensional subspaces (i.e.\ closed hyperplanes $\Ker f$, $f\ne0$): for $f\ne0$ we have $X/\Ker f\cong\K$, hence $\dim(X/\Ker f)=1<\aleph_0$, and since $\dim X$ is infinite, $\dim X=\dim\Ker f+1$ gives $\dim\Ker f=\dim X$, that is, $\Ker f\in\mathcal Z(X)$; while $\Ker f=\Ker g$ if and only if $g=\lambda f$ ($\lambda\in\K\setminus\{0\}$), and each class contains $\cc$ nonzero functionals, so $N\cdot\cc=|X^{*}\setminus\{0\}|=|X^{*}|$. For $X=\ell^{\infty}$ we have $|X^{*}|=2^{\cc}>\cc$: then $N$ is infinite (otherwise $N\cdot\cc=\cc$), so $N\cdot\cc=\max(N,\cc)=2^{\cc}$ gives $N=2^{\cc}$, and substituting into the above inequality yields $|\mathcal Z(\ell^{\infty})|=2^{\cc}>|X|=\dim X=\cc$. This is precisely the failure of the equality $|\mathcal Z(X)|=|X|$ asserted in [U1, Lemma 4 (ii)]; since $\dim\ell^{\infty}=\cc\ne\aleph_0$, $\ell^{\infty}$ is a counterexample to the affirmative answer sought by [U1, Question 1], and that answer is therefore \emph{negative}.
\end{remark}

The implication from equal finite quotient shapes to equal Hamel dimensions fails without additional hypotheses. For example, equip $c_{00}$ with the supremum norm. Both $c_{00}$ and $c_0$ have continuous dual $\ell^{1}$, so Theorem \ref{thm:shape-classification} gives equal finite quotient shapes, whereas $\dim c_{00}=\aleph_0$ and $\dim c_0=\cc$.

\subsubsection{The correct quantitative picture}\label{sub:dim-correct}

The counterexamples provide negative information; now we give positive statements that can replace that group of assertions of the original. A functional on a dual space is determined by its values on any dense subset, so the density character $\dens(X)$---and not the Hamel dimension---is the quantity that yields a \emph{general upper bound} for the size of $X^{*}$. Note at the outset that this bound is only an upper bound: the density character does \emph{not} determine $|X^{*}|$ (see the end of this subsection).

\begin{lemma}[the density character gives an upper bound for the dual]\label{lem:dens-dual}
Let $X$ be a normed space. Then
\begin{equation}\label{eq:dens-dual}
  \dim X\ \le\ \dim X^{*}\ \le\ |X^{*}|\ \le\ \cc^{\dens(X)} .
\end{equation}
\end{lemma}

\begin{proof}
If $X$ is finite-dimensional, then $\dim X=\dim X^*$; the remaining estimates follow from the restriction argument below. We therefore assume that $X$ is infinite-dimensional. The Hahn--Banach theorem, applied to finite-dimensional subspaces of arbitrarily large dimension, shows that $X^*$ is infinite-dimensional.

Let $D\subset X$ be dense with $|D|=\dens(X)$. A continuous functional is uniquely determined by its restriction to $D$, so restriction gives an injection $X^*\to\K^D$. Consequently
\[
\dim X^*\leq|X^*|\leq\mathfrak c^{\dens(X)}.
\]

To prove the left-hand inequality, we first show that
\[
\dens(X)\leq\dens(X^*).
\]
Choose a norm-dense subset $S$ of the unit sphere of $X^*$ with $|S|=\dens(X^*)$. For every $f\in S$, choose $x_f\in X$ with $\|x_f\|=1$ and $|f(x_f)|>1/2$. Put
\[
Y=\overline{\Span}\{x_f:f\in S\}^{\,X}.
\]
If $Y$ were a proper closed subspace of $X$, Hahn--Banach would provide $g\in X^*$ with $\|g\|=1$ and $g|_Y=0$. Choose $f\in S$ with $\|f-g\|<1/4$. Then
\[
\frac12<|f(x_f)|=|(f-g)(x_f)|<\frac14,
\]
a contradiction. Therefore $Y=X$. Finite linear combinations of the $x_f$ with coefficients in a countable dense subfield form a dense subset of cardinality at most $|S|$. This proves the density inequality.

Every point of $X$ is the limit of a sequence in $D$, so the map sending each sequence in $D^{\N}$ that converges in $X$ to its limit is onto $X$. Hence $|X|\leq\dens(X)^{\aleph_0}$. Applying Lemma~\ref{lem:cardinality-density} to the infinite-dimensional Banach space $X^*$ now gives
\[
\dim X\leq|X|
\leq\dens(X)^{\aleph_0}
\leq\dens(X^*)^{\aleph_0}
=|X^*|=\dim X^*.
\]
Together with the restriction estimate, this proves the assertion.
\end{proof}

\begin{theorem}[comparison of dimensions]\label{thm:dim-compare}
Let $X$ be an infinite-dimensional Banach space. Then
\begin{enumerate}
  \item[(i)] along the dual chain the Hamel dimension is non-decreasing:
    \[
      \dim X\ \le\ \dim X^{**}\ \le\ \dim X^{(4)}\ \le\ \cdots\ \le\ \dim\widetilde X ;
    \]
  \item[(ii)] along the inverse chain we have $\dim X^{(2n+1)}\le\dim\overleftarrow X$ ($n\ge0$);
  \item[(iii)] the levelwise bound $\dim X^{(n+1)}\le2^{\dim X^{(n)}}$ ($n\ge0$), and in particular $\dim X^{*}\le2^{\dim X}$.
\end{enumerate}
\end{theorem}

\begin{proof}
(i) The canonical embedding $J_{X^{(2n)}}:X^{(2n)}\to X^{(2n+2)}$ between two adjacent layers is an injective linear map, while Theorem \ref{thm:tower-exists} provides the isometric embeddings $\kappa_{n,\infty}:X^{(2n)}\to\widetilde X$; the former makes the sequence $\dim X^{(2n)}$ non-decreasing, and the latter shows that $\dim\widetilde X$ is an upper bound for that non-decreasing sequence.

(ii) By Theorem \ref{thm:dir-dual-inv}, $\overleftarrow X\equiv(\widetilde X)^{*}$. Since $\kappa_{n,\infty}$ is an isometric embedding, the Hahn--Banach theorem implies that the restriction map $(\widetilde X)^{*}\to(X^{(2n)})^{*}=X^{(2n+1)}$ is a linear surjection; a linear surjection does not increase the Hamel dimension, so $\dim X^{(2n+1)}\le\dim(\widetilde X)^{*}=\dim\overleftarrow X$.

(iii) This is \eqref{eq:dim-bound} applied level by level.
\end{proof}

\begin{proposition}[the correct quantitative picture]\label{prop:dim-correct}
Let $X$ be an infinite-dimensional Banach space.
\begin{enumerate}
  \item[(i)] General upper bound: $\dim X^{*}\le|X^{*}|\le\cc^{\dens(X)}$, and equality can occur (Theorem \ref{thm:dim-counter}); in the separable case it reduces to $\dim X=\dim X^{*}=\cc$;
  \item[(ii)] If $X$ is moreover separable, then $\dim X=\dim X^{*}=\cc$ while $\dim X^{**}\ge\cc$; but $\dim X^{**}$ has no such upper bound $\cc$ and can be strictly larger: for instance, for $X=\ell^{1}$ we have $\dim X^{**}=\dim\ba(\N)=2^{\cc}$ (Example \ref{ex:dim-chain});
  \item[(iii)] Consequently $\dim X^{*}$ is not a function of $\dim X$: both $\dim X=\dim X^{*}=\cc$ and $\dim X=\cc,\ \dim X^{*}=2^{\cc}$ occur, the former for $\ell^{1}(\N)$, $c_0$ and the James space (Theorem \ref{thm:james-tower}), the latter for $\ell^{\infty}(\N)$.
\end{enumerate}
\end{proposition}

\begin{proof}
(i) This is Lemma \ref{lem:dens-dual}; for $\ell^{\infty}$ we have $\dens(\ell^\infty)=\cc$ (it contains the image of the uncountable discrete set $\{0,1\}^{\N}$, see below), so $\cc^{\dens}=2^{\cc}=|(\ell^\infty)^*|$ and equality is indeed attained.

(ii) If $X$ is infinite-dimensional and separable, then $\dens(X)=\aleph_0$, so \eqref{eq:dens-dual} gives $\dim X^{*}\le|X^{*}|\le\cc^{\aleph_0}=\cc$; on the other hand $X$ is an infinite-dimensional Banach space, so $\dim X=\cc$ (Lemma \ref{lem:dim-card}) and $\dim X^{*}\ge\dim X=\cc$, whence $\dim X^{*}=\cc$. As for $X^{**}$, the non-decreasing chain of Theorem \ref{thm:dim-compare}(i) only gives the lower bound $\dim X^{**}\ge\dim X=\cc$; one \emph{cannot} repeat the above upper-bound argument for $X^{**}$: the input density controlling $X^{**}$ is $\dens(X^{*})$, not $\dens(X)$, and $\dens(X^{*})$ need not be countable (a separable $X$ does not imply a separable $X^{*}$), so \eqref{eq:dens-dual} applied to $X^{*}$ gives no upper bound $\cc$ for $X^{**}$. Indeed, for $X=\ell^{1}$ we have $X^{*}=\ell^{\infty}$, which is non-separable, and $\dim X^{**}=\dim(\ell^{\infty})^{*}=2^{\cc}$ (Example \ref{ex:dim-chain}), which is already strictly larger than $\cc$.

(iii) This follows at once from (i), (ii) and Theorem \ref{thm:dim-counter}; for the separability of the James space and $J\equiv J^{**}$ see Example \ref{ex:james}, and for the fact that all of its layers have dimension $\cc$ see Theorem \ref{thm:james-tower}.
\end{proof}

\begin{remark}[the $\aleph_0$ half is right, but only for incomplete spaces]\label{rem:dim-aleph0}
[U2, Theorem 5] asserts that $\dim X^{*}=2^{\dim X}\iff\dim X=\aleph_0$. Its ``$\Leftarrow$'' direction is correct, for elementary reasons, for \emph{incomplete} normed spaces: in that case $X$ has a countable Hamel basis, hence is separable, so $|X^{*}|\le\cc$ (the proof of Lemma \ref{lem:dens-dual}), while $X^{*}$ is an infinite-dimensional Banach space, so $\dim X^{*}=\cc=2^{\aleph_0}$. In the \emph{complete} case, however, this direction never arises---because an infinite-dimensional Banach space has dimension $\ge\cc$ (Lemma \ref{lem:dim-card}), so $\dim X=\aleph_0$ never arises. Thus the ``$\aleph_0$ side'' of this assertion is really about \emph{incomplete} spaces, while the ``$\Rightarrow$'' side (that the dimension stops growing in the uncountable case) is directly refuted by $\ell^{\infty}$.
\end{remark}

\begin{remark}[the root of the matter]\label{rem:dim-root}
The root of the problem here lies not in some cardinality computation, but in treating the Hamel dimension as the \emph{input} of the dual operator. $D$ converts topological information into an algebraic dimension, whereas the density character at the input end only \emph{bounds} $|X^{*}|$ from above---it does not determine it:
\begin{equation}\label{eq:dim-philosophy}
  \dim X\ \le\ \dim X^{*}\ \le\ |X^{*}|\ \le\ \cc^{\dens(X)} ,
  \qquad
  \sup_{n}\dim X^{(2n)}\ \le\ \dim\widetilde X .
\end{equation}
For $\ell^{\infty}$ (with $\dens=\cc$) the right-hand side of the first formula gives $2^{\cc}$, and for separable spaces it gives $\cc$; the bound is sharp in these two examples, but it is a genuine upper bound in general and does not determine $|X^{*}|$ (Remark \ref{rem:dens-counter}). This also explains why (H2) of [U1] must retain the hypothesis ``separable or bidual-like''---separability is precisely the hypothesis that keeps $\dens$ under control; once it is dropped, ``shape $\Rightarrow$ dimension'' fails (Proposition \ref{prop:lemma4-fail}). The second formula of \eqref{eq:dim-philosophy} is the general lower bound we adopt on the hyperdual side; in general it is only a lower bound, and it may be \emph{strict}: for $\ell^{\infty}$ passing to the completion raises the dimension. As for the exact values, they turn out to be computable after all: $\dim\ell^{\infty**}=2^{2^{\cc}}$ (Proposition \ref{prop:linf-bidual}), and along the dual chain of $\ell^{\infty}$ the dimensions are given by the beth numbers, the completed hyperdual satisfying $\dim\widetilde{\ell^{\infty}}=\beth_{\omega+1}$ while the union of the uncompleted dual chain has dimension $\beth_{\omega}$ (Remark \ref{rem:beth-tower}); see also \S\ref{sec:open}.
\end{remark}

\begin{remark}[density does not determine the dimension of the dual]\label{rem:dens-counter}
Let $|\Gamma|=\cc$. Both $\ell^{2}(\Gamma)$ and $\ell^{\infty}(\N)$ have density $\cc$, while their duals have Hamel dimensions $\cc$ and $2^{\cc}$, respectively. Indeed, for the Hilbert space every vector has countable support, so $|\ell^{2}(\Gamma)|=\cc$ and Lemma \ref{lem:dim-card} gives $\dim\ell^{2}(\Gamma)=\cc$, while reflexivity gives the same value $\cc$ for its dual; and $\dens(\ell^{2}(\Gamma))=\cc$ because $\Gamma$ itself is dense and of cardinality $\cc$, whereas $\dens(\ell^{\infty}(\N))=\cc$ by Lemma \ref{lem:cardinality-density} (it contains the image of the uncountable discrete set $\{0,1\}^{\N}$). For $\ell^{\infty}(\N)$ we have $\dim\ell^{\infty}=\cc$ and $\dim(\ell^{\infty})^{*}=2^{\cc}$ (Theorem \ref{thm:dim-counter}). Thus the two spaces have the same density character $\cc$ but duals of \emph{different} Hamel dimensions $\cc$ and $2^{\cc}$: the density character gives the stated general upper bound $2^{\cc}$ rather than an exact determination of $\dim X^{*}$.
\end{remark}

Organizing the above into \emph{replacement statements} for the original, we adopt the following version:

\begin{enumerate}
  \item[(a)] \textbf{Retain and emphasize what is correct}: for any normed space
    \[
      \dim X\ \le\ \dim X^{*}\ \le\ |X^{*}|\ \le\ \cc^{\dens(X)} ,
    \]
    and if $X$ is moreover infinite-dimensional and separable, then $\dim X^{*}=\cc$ (its dual
    is an infinite-dimensional Banach space with $|X^{*}|\le\cc^{\aleph_0}=\cc$); if in addition
    $X$ is \emph{complete}, then also $\dim X=\cc$, while $\dim X^{**}$ is only guaranteed to
    be at least $\cc$ and may be strictly larger (for instance, for $X=\ell^{1}$ we have
    $\dim X^{**}=2^{\cc}$, see Example \ref{ex:dim-chain}; Proposition \ref{prop:dim-correct}(ii)).
    Completeness is a sufficient hypothesis for the equality $\dim X=\dim X^{*}$ asserted
    above; the incomplete separable space $c_{00}$ satisfies $\dim c_{00}=\aleph_0$ whereas
    $\dim c_{00}^{*}=\dim\ell^{1}=\cc$, which shows that completeness cannot simply be omitted
    from this general statement.
  \item[(b)] \textbf{Rewrite [U2, Theorem 4] as}: ``$\dim X=\aleph_0\Rightarrow\dim X^{*}=2^{\aleph_0}$ (in the incomplete case); in general $\dim X^{*}\le\cc^{\dens(X)}$, and equality can be attained ($\ell^{\infty}$).''
  \item[(c)] \textbf{Delete the second equivalence of [U2, Theorem 5]}, or add the hypothesis ``$X$ is an infinite-dimensional separable Banach space'', in which case $\dim X=\dim X^{*}=\cc$ also holds. (The variant ``$\dim X^{*}=\cc=2^{\dim X}$'' is a definite error: for such a space $\dim X=\cc$, hence $2^{\dim X}=2^{\cc}>\cc$, and the incomplete case must be excluded as in (a).) Correspondingly, the last sentence ``$\dim D^{n}(X)=\dim X^{*}$ for all $n$'' must be deleted (Example \ref{ex:dim-chain}).
  \item[(d)] \textbf{If [U2, Corollary 4] is to be used, one must add the hypothesis} ``all iterated duals of $X$ are separable''; otherwise it can only be written as the chain of inequalities $\dim D^{k+1}(X)\le\cc^{\dens(D^{k}(X))}$.
  \item[(e)] \textbf{[U2, Lemma 4] fails even for separable Banach spaces.} Proposition \ref{prop:lemma4-fail} gives $X=\ell^{1}(\N)$ and $Y=\ell^{2}(\N)$ with equal finite quotient shapes and with $Y$ $\dim^{*}$-stable, but
    \[
      \dim X^{**}=2^{\cc}>\cc=\dim X^{*}=\dim Y.
    \]
    Thus separability is not a sufficient repair of this lemma.
  \item[(f)] \textbf{Delete the equality of [U3, Theorem 2(iii)]}, and replace it by the lower bound \eqref{eq:dim-tower-lower}.
  \item[(g)] \textbf{[U1, Lemma 4 (ii)] retains the partial hypothesis ``separable or bidual-like''}, and the sentence about $\ell^{\infty}$ following Question 1---where its author claims that $X=\ell^{\infty}$ ``is an example towards the affirmative answer''---must be \emph{deleted}: the very counting argument belonging to Lemma 4 (ii) gives the opposite, $|\mathcal Z(\ell^{\infty})|=2^{\cc}>|X|=\dim\ell^{\infty}=\cc$ (Remark \ref{rem:Z-count}). Hence [U1, Question 1] has a \emph{negative} answer, and $\ell^{\infty}$ is a counterexample to it rather than an example in its favour.
  \item[(h)] \textbf{The side (H1), ``same dimension $\Rightarrow$ same finite quotient shape'', must be \emph{deleted}}: by the classification \eqref{eq:shape-criterion} the quantity that classifies the finite quotient shape is $\dim X^{*}$, not $\dim X$, and two spaces of the same dimension may well have duals of different dimensions. The instance at $\ell^\infty$ and $\ell_p$ that (H1) would cover is thereby decided in the \emph{negative}: $\dim\ell^{\infty}=\dim\ell_p=\cc$, so (H1) would assert $\Sh_{\bar0}(\ell^\infty)=\Sh_{\bar0}(\ell_p)$, whereas $\dim(\ell^{\infty})^{*}=2^{\cc}\ne\cc=\dim(\ell_{p})^{*}$ for $1\le p<\infty$, so $\Sh_{\bar0}(\ell^{\infty})\ne\Sh_{\bar0}(\ell_{p})$ (Theorem \ref{thm:shape-classification}). This confirms the doubt expressed by the author of [U1] in \cite[the text before Remark 3.6]{uglesic2018lp}, where he wrote that ``it seems that $\Sh_{\bar0}(l^\infty)$ might differ from (the unique) type $\Sh_{\bar0}(l_p)$'', which apparently conflicts with his own Theorem 2 (i); what that paper actually proves about $\ell^\infty$ is only $\Sh_{\aleph_0}(l^\infty)=\Sh_{\bar0}(l^\infty)$, which is perfectly compatible with the above. What the original calls (H2), on the other hand, is retained: it is proved by the same classification (Corollary \ref{cor:shape-H2}).
\end{enumerate}

\subsubsection{Conclusions that are unaffected}\label{sub:dim-unaffected}

The above counterexamples are directed only at the group of assertions saying that ``the algebraic dimension is not raised''. The technical conclusions of the three preprints that have nothing to do with the algebraic dimension are unaffected; several of them have already been re-established in Section \ref{sec:prelim} in the form of independent proofs: the existence and the universal property of the hyperdual space ([U3, Theorem 1], i.e.\ Theorem \ref{thm:tower-exists} of this paper), the \emph{isometric embeddings} of the various $X^{(2n)}$ into $\widetilde X$ and of $\widetilde X$ into $D^{2}(\widetilde X)$ (the embedding part of [U3, Theorem 2(i),(ii)], i.e.\ \eqref{eq:cat-hyperdual-embed} of this paper), and the identity $\bigl(\widetilde X\bigr)^{*}\equiv\varprojlim X^{(2n+1)}$ ([U3, Corollary 2], i.e.\ Theorem \ref{thm:dir-dual-inv} of this paper).

Moreover, the \emph{positive} examples listed in [U2] are all correct: finite-dimensional spaces ($\dim X^{*}=\dim X<\infty$), separable spaces (Proposition \ref{prop:dim-correct}(ii)), and $\ell_p\ (1<p<\infty)$ together with finite-dimensional $L_p$ spaces (thus [U2, Example 2] in the original, whose notation differs from ours). For the James space $J$, every finite iterated dual is separable and infinite-dimensional, so $\dim J^{(n)}=\cc$ for every finite $n$; this is the dimension-stability property used here. (Here the deficiency direction of the first layer is the vector $u_1$ of Theorem \ref{thm:james-tower}, so that $J_1=J^{**}=J_J(J)\oplus\R u_1$ is a topological direct sum (Theorem \ref{thm:james-tower}); the symbol $R$ is reserved for the infinite-dimensional complement \emph{inside the hyperdual space} $\widetilde J$.)

\subsection{\texorpdfstring{Fixing the category, and the commutation of $D^{2}$ with direct limits}{Fixing the category, and the commutation of D2 with direct limits}}\label{sec:cat}

The category-theoretic notation of this subsection is the same as in Section \ref{sec:prelim}: $\Ban$, $\Ban_1$, $\iBan$, and, after dropping the completeness requirement, $\NF$ and $\iNF$ (Definition \ref{def:ban-categories}). Three assertions in [U3] share a single picture:

\begin{quote}
\emph{the even-order dual functor $D^{2k}$ commutes with direct limits, and it commutes with them ``isometrically''.}
\end{quote}

In this section we show that the ``commutes'' half of this picture does not hold; and the failure is due not to some computational slip but to the choice of \emph{the category in which the direct limit is taken}---a point passed over in a single clause of the original. Following the uniform pattern of this section, we first quote the original statements (\S\ref{sub:cat-claims}), then explain why the choice of category is the heart of the matter (\S\ref{sub:cat-category}), give counterexamples (\S\ref{sub:cat-counter}), write down the correct picture (\S\ref{sub:cat-correct}), and finally explain which conclusions are unaffected (\S\ref{sub:cat-unaffected}).

\subsubsection{The original assertions}\label{sub:cat-claims}

We first identify how the object of the direct limit is constructed in [U3] (the English text is quoted verbatim from the original; the location is given in brackets).

\begin{pquote}
(the construction of Lemma 1, pp.~5--6) ``\dots\ an equivalence relation on $\bigsqcup_{n\in\N}X_n$.
Let \dots\ be the corresponding quotient set.''\\[3pt]
(p.~6) ``where $x=[x_{n(x)}]$ and $x_{n(x)}\in X_{n(x)}$ is the \emph{grain} of $x$.
Then, clearly, $\|x\|=\|x_{n(x)}\|_{n(x)}=\|x_n\|$, $n\ge n(x)$. \dots\ Thus,
$X\equiv(X,\|\cdot\|)$ is a \emph{normed space} \dots''
\end{pquote}

Note the last sentence: the object explicitly constructed out of $\bigsqcup/\sim$ is only asserted to be a \emph{normed} space, and \emph{no completion is taken}. The three assertions below all refer to this uncompleted object.

\begin{pquote}
\textbf{[U3, Lemma 4]} (p.~9): ``\dots\ such that, for every
$X=(X_n,i_{nn'},\N)\in\Ob(\text{seq-}\iNF)$,\ $\varinjlim D^{2k}[X]\equiv(X',i'_n)
\cong(D^{2k}(X),D^{2k}(i_n))$ belongs to $D^{2k}[\iNF]$.''\\[4pt]
\textbf{[U3, Theorem 1 (i)]} (p.~11): ``Each restriction functor
$D^{2k}:\iNF\to D^{2k}[\iNF]$ \dots\ is continuous, i.e., it commutes with the direct limit:
$D^{2k}(\varinjlim X)\cong\varinjlim D^{2k}[X]$ isometrically.''\\[4pt]
\textbf{[U3, Theorem 2 (i)]} (p.~15): ``$\widetilde X$ is a bidualic Banach space,
i.e., $D^2(\widetilde X)\cong\widetilde X$ in $\NF$ \dots''
\end{pquote}

Two sentences in the proof of Lemma 4 are the key to the whole matter (pp.~9--10):

\begin{pquote}
(p.~9) ``\dots\ (\emph{implying that $X'$ is isomorphic to $D^{2k}(X)$ in $\iNF$,
and hence, a Banach space}).''\\[3pt]
(p.~10) ``\dots\ it follows that $X'\cong D^{2k}(X)$ in
$D^{2k}[\iNF]\subseteq\iBF$, and the statement for objects follows in general.''
\end{pquote}

Moreover, the proof of Theorem 2 (i) declares at the outset that ``it suffices to prove it for the canonical direct limit space'', that is, it suffices to discuss the uncompleted object $\bigsqcup/\sim$ above. Finally, [U3] itself raises a reservation in its Remark 1 (b) (p.~17): ``If it could hold an appropriate analogue of Lemma 5 for the isometries (in the very special case of the proof of Theorem 2 (i)), then $\widetilde X$ would be a parareflexive space.'' In other words, its author himself lists the \emph{isometric} version of Theorem 2 (i) as undecided.

\subsubsection{Fixing the category}\label{sub:cat-category}

The relevant assertions of [U3] are all made in $\NF$/$\iNF$, and it is inside them that the question ``is the direct limit complete?'' is discussed. In the framework of this paper, this is unnecessary, and it is precisely the source of the difficulty. There is only one reason:

\begin{quote}
\emph{the dual of a normed space is \emph{always} a Banach space.}
\end{quote}

Consequently, from the first dual onward the whole dual chain $X^{(2)}\subset X^{(4)}\subset\cdots$ lies in the category of Banach spaces. The hyperdual space is the direct limit of this chain, and the identification of this limit rests on the following three facts, all established in Section \ref{sec:prelim}:

\begin{enumerate}
  \item[(i)] by Lemma \ref{lem:isometric-colim}(iii), the direct limit in $\iBan$ of an isometric direct system is isometrically isomorphic to the \emph{completion} of $\bigsqcup/\sim$ (equipped with the grain norm); in particular it is a Banach space;
  \item[(ii)] by Remark \ref{rem:tower-isometric}, all the connecting maps of the dual chain are isometries, so it is indeed an isometric direct system in $\iBan$, and (i) can be applied directly;
  \item[(iii)] by Theorem \ref{thm:ban1-limits}, all direct and inverse limits exist in $\Ban_1$; for the isometric direct systems used here, Lemma \ref{lem:isometric-colim} additionally gives the direct limit in $\iBan$. These are separate assertions and are not combined into a single general existence statement: in particular, no general existence theorem for inverse limits in $\iBan$ is being invoked. Nor need such limits exist in $\Ban$ (Remark \ref{rem:why-ban1} and \eqref{eq:colim-hom}).
\end{enumerate}

Hence this paper adopts the convention stated at the end of Section \ref{sec:prelim}: the direct limit of the dual chain is always taken in $\iBan$, i.e.\ \emph{the object is completed}; the inverse limit is taken in $\Ban_1$, with the explicit description given in Remark \ref{rem:inv-limit-form}. This is not a technical preference but a necessity: if one retreats to $\iNF$, the limit object is merely a normed space, whereas $D^{2k}(\cdot)$ is always complete, and the two cannot even be isomorphic (see \S\ref{sub:cat-counter}). Conversely, once the object is taken to be the completion, the ``Banach space'' part of [U3, Theorem 2 (i)] follows directly from the construction of Theorem \ref{thm:tower-exists}: the completed limit is a Banach space by construction, and no separate argument is needed. Two points must, however, be kept separate from this. Completing the limit does \emph{not} by itself place it in $D^{2k}[\iNF]$, that is, does not make it \emph{isometrically a dual space}: completeness is necessary but not sufficient for being a dual space (\S\ref{sub:cat-correct}). And the ``Banach space'' half of [U3, Theorem 2 (i)] is thereby retained, whereas the accompanying $\cong$ is a separate assertion that does not follow from the convention (\S\ref{sub:cat-correct}).

Let us now draw the ``two orders of taking limits'' on a single diagram. Let $(X_n,\kappa_{n,m})$ be an isometric direct system in $\iBan$, fix $k\ge0$, and put $X_\infty:=\varinjlim_n X_n$. Taking the limit first and then applying $D^{2k}$, and applying $D^{2k}$ first and then taking the limit, yield two different objects, between which there is a family of natural comparison maps:

\begin{equation}\label{eq:cat-square}
\begin{tikzcd}[column sep=6.0em, row sep=3.4em]
  X_n \arrow[r, hook, "\kappa_{n,\infty}"] \arrow[d, hook, "J^{(k)}_{X_n}"']
    & X_\infty=\varinjlim_n X_n \arrow[d, hook, "J^{(k)}_{X_\infty}"] \\
  D^{2k}(X_n) \arrow[r, hook, "D^{2k}(\kappa_{n,\infty})"']
    & D^{2k}(X_\infty)
\end{tikzcd}
\end{equation}

Here a distinction of \emph{role} has to be observed, because the same symbol $D^{2k}$ appears in the diagram in two different senses. The \emph{lower row} is the upper row with the \emph{functor} $D^{2k}$ applied to it: the functor acts on the objects $X_n$, $X_\infty$ and on the morphism $\kappa_{n,\infty}$, and it is precisely its action on that morphism that the lower horizontal arrow records. The two \emph{vertical} arrows are something else altogether: they are concrete linear maps between concrete objects, namely the iterated canonical embeddings
\[
  J^{(k)}_{E}:=J_{E^{(2k-2)}}\circ\cdots\circ J_{E^{(2)}}\circ J_{E}:E\longrightarrow E^{(2k)}=D^{2k}(E)\qquad(k\ge1),
\]
with $J^{(0)}_{E}:=\id_E$; thus $J^{(1)}_{E}=J_E$, and for $k=0$ the diagram collapses to the identity and carries no information. These $J^{(k)}_{E}$ are \emph{not} the functor $D^{2k}$: labelling a vertical arrow ``$D^{2k}$'' would conflate the functor's action on objects and morphisms with a particular linear map between objects. By naturality of the canonical embedding the square commutes,
\[
  D^{2k}(\kappa_{n,\infty})\circ J^{(k)}_{X_n}=J^{(k)}_{X_\infty}\circ\kappa_{n,\infty},
\]
as follows from iterating the identity $T^{**}J_E=J_FT$; the two directions in which $D^{2k}$ may be read are exactly the two axes of \eqref{eq:cat-square}.

Both the upper and the lower rows of Diagram \eqref{eq:cat-square} consist of isometric embeddings (Lemma \ref{lem:isometry}), and so does each column. Writing out the universal property of this family of comparison maps is precisely the content of Proposition \ref{prop:colim-embed}: there is a unique isometric embedding

\begin{equation}\label{eq:cat-iota}
  \iota_k:\ \varinjlim_n D^{2k}[X_n]\ \hookrightarrow\ D^{2k}(X_\infty)
  =D^{2k}\Bigl(\varinjlim_n X_n\Bigr),
\end{equation}

such that $D^{2k}(\kappa_{n,\infty})=\iota_k\circ\lambda_n$, where $\lambda_n:D^{2k}(X_n)\to\varinjlim_n D^{2k}[X_n]$ is the canonical map of the direct limit. Proposition \ref{prop:colim-embed} also states that $\iota_k$ is \emph{in general not surjective}: therefore the asserted equality ``$D^{2k}(\varinjlim)\cong\varinjlim D^{2k}$'' does not hold. This is where the problem with [U3, Theorem 1 (i)] lies, and the next subsection gives a counterexample.

\subsubsection{Counterexamples}\label{sub:cat-counter}

The counterexample requires only finite-dimensional spaces.

\begin{example}[$D^{2}$ does not commute with direct limits]\label{ex:cat-cn}
Take the direct sequence in $\iBan$
\[
  X_1=(\C^1,\|\cdot\|_\infty)\xhookrightarrow{\ i_{12}\ }
  X_2=(\C^2,\|\cdot\|_\infty)\xhookrightarrow{\ i_{23}\ }\cdots,
  \qquad X_n=(\C^n,\|\cdot\|_\infty),
\]
where $\|(\xi_1,\dots,\xi_n)\|_\infty=\max_{1\le j\le n}|\xi_j|$ and the connecting maps are the \emph{zero-padding} isometric embeddings
\[
  i_{n,n+1}(\xi_1,\dots,\xi_n)=(\xi_1,\dots,\xi_n,0).
\]
On finite-dimensional spaces $T^{**}=T$, so $D^2(X_n)=X_n$ and $D^2(i_{nn'})=i_{nn'}$, i.e.\ $D^2[X]=X$; in particular the bidual of every connecting map is again the zero-padding map, and the bidual system coincides with the original direct system. The algebraic union of the system is $c_{00}$ with its supremum norm, and its completion is $c_0$. Hence
\[
  \varinjlim_n D^2[X_n]=\varinjlim_n X_n
  =\begin{cases}
     c_{00}\ \text{(supremum norm, incomplete)}, & \text{in }\iNF,\\[3pt]
     c_{0}, & \text{in }\iBan\text{ (completion taken)},
   \end{cases}
\]
and both are \emph{separable}. On the other hand, the bidual of the completed colimit is
\[
  D^2(X_\infty)=D^2(c_0)=c_0^{**}=\ell^{\infty},
\]
which is non-separable (Example \ref{ex:c0}). Thus the completed colimit of the bidual system is $c_0$, whereas the bidual of the completed colimit is $c_0^{**}=\ell^{\infty}$: these spaces differ in separability and therefore are not isomorphic as Banach spaces. For the uncompleted colimit in $\iNF$, the colimit of the finite-dimensional bidual system is $c_{00}$, whereas the bidual of $c_{00}$ with the supremum norm is $\ell^{\infty}$; here completeness alone already separates the two normed spaces. So the asserted commutation with the ambient direct limit fails under both conventions, and
\[
  D^{2}\Bigl(\varinjlim_n X_n\Bigr)\ \not\cong\ \varinjlim_n D^{2}[X_n].
\]
\textbf{In particular, both [U3, Lemma 4] and [U3, Theorem 1 (i)] fail.}
\end{example}

The branch of Example \ref{ex:cat-cn} in which the limit is taken in $\iBan$ shows that \emph{adding the completion does not repair the failure}: the reason for the failure is not completeness but that $D^{2}$ does not commute with direct limits at all. This is a difficulty independent of the possible non-existence of colimits in $\Ban$ (Remark \ref{rem:why-ban1}).

\begin{example}[``and hence, a Banach space'' is false]\label{ex:cat-banach}
Take again the sequence of Example \ref{ex:cat-cn}, but look only at the object $X'=c_{00}$ in $\iNF$, equipped with the supremum norm. The parenthetical claim of Lemma 4 is that $X'$ is isomorphic to $D^{2k}(X)$, ``and hence, a Banach space''. In the supremum norm the closure of $c_{00}$ is $c_{0}$: the finite truncations of $(1/j)_{j\ge1}$ form a Cauchy sequence in $c_{00}$ whose limit lies in $c_{0}\setminus c_{00}$, so $c_{00}$ is \emph{incomplete}. On the other hand, for $k\ge1$ every space $D^{2k}(X)$ is the dual of a Banach space (Lemma \ref{lem:dual-banach}) and is therefore always complete. Completeness is preserved under (linear) isomorphism, so the two cannot be isomorphic, and the clause ``hence a Banach space'' is simply false. (The constant sequence $\mathbf1$ stays at distance $1$ from $c_{00}$ in $\ell^{\infty}$, which separately shows that $c_{00}$ is not dense in $\ell^{\infty}$.)
\end{example}

The last counterexample puts the above phenomenon back into the dual chain itself, and it is the basis for the failure of Theorem 2 (i) in the next subsection.

\begin{example}[the canonical hyperdual of a non-reflexive Banach space is not a Banach space]\label{ex:cat-canonical}
Let $X$ be a non-reflexive \emph{Banach} space. By Theorem \ref{thm:chain} and Corollary \ref{cor:strict-chain}, the dual chain $X\subsetneq X^{(2)}\subsetneq X^{(4)}\subsetneq\cdots$ is strictly increasing, whose connecting maps are isometric embeddings. Equip the union $U=\bigcup_nX^{(2n)}$ with the grain norm (this is precisely the construction of the ``canonical direct limit'' in [U3, Lemma 1]). Each $X^{(2n)}$ is complete, hence a closed subspace of $U$; and it is a proper closed subspace of the next term, hence has no interior point in $U$, i.e.\ it is nowhere dense. If $U$ were complete, then $U$ could be written as a countable union of nowhere dense sets, contradicting the Baire category theorem. Hence $U$ is \emph{incomplete}: \textbf{the canonical hyperdual (the uncompleted version) of a non-reflexive Banach space is not a Banach space.}
\end{example}

The hypothesis ``Banach'' is indispensable and cannot be weakened to ``normed'': the non-reflexive normed space $E^{\infty}$ of Example \ref{ex:einf} has dual chain $E^{\infty}\subsetneq\ell^{2}=\ell^{2}=\cdots$, which stabilizes from the second term onward, so its union is the complete space $\ell^{2}$ and the Baire argument above does not apply (indeed, the dual chain is not strictly increasing). It is completeness of $X$ that forces the dual chain to keep increasing at every level, which is exactly what the Baire category argument uses.

\subsubsection{The comparison map and the canonical embedding}\label{sub:cat-correct}

The uncompleted object $U_X$ and its completion $\widetilde X$ must be treated separately. For a non-reflexive Banach space $X$, Example \ref{ex:cat-canonical} proves that $U_X$ is incomplete. It therefore refutes the Banach-space assertion for the uncompleted object used in [U3, Theorem 2 (i)]. Passing to the completion produces a Banach space, but changes the object under discussion. Example \ref{ex:cat-cn} shows that completion also leaves the general commutation assertion false.

A universal property in an image category concerns only the objects and morphisms of that category. It must be distinguished from the universal property in the ambient category. In particular, an identification obtained in $D^{2k}[\iNF]$ does not by itself identify the explicitly constructed ambient direct limit. Instead, we use the comparison map proved in Proposition \ref{prop:colim-embed}.

Put $A_n=X^{(2n)}$, $a_n=J_{A_n}$, and
\[
T=\varinjlim(A_n,a_n)=\widetilde X,
\qquad
S=\varinjlim(A_n^{**},a_n^{**}).
\]
Write $\kappa_{n,\infty}:A_n\to T$ and $\lambda_n:A_n^{**}\to S$ for the structure maps. Proposition \ref{prop:colim-embed} gives an isometric embedding
\begin{equation}\label{eq:cat-correct}
\iota:S\longrightarrow T^{**},\qquad
\iota\lambda_n=\kappa_{n,\infty}^{**}.
\end{equation}
Although $A_n^{**}=A_{n+1}$ at the level of objects, $a_n^{**}$ is generally different from $a_{n+1}$. For example, take $X=c_0$ and let $\mu\in(\ell^\infty)^*$ be the functional defined by a free ultrafilter on $\N$. Then $\mu|_{c_0}=0$ and $\mu(\mathbf 1)=1$, so
\[
\bigl((J_{c_0})^{**}\mathbf 1\bigr)(\mu)=0,
\qquad
\bigl(J_{\ell^\infty}\mathbf 1\bigr)(\mu)=1.
\]
Thus $S$ is not identified with $T$ merely by deleting the first object of the dual chain; the connecting maps have to be included in any such identification.

Naturality gives
\[
a_n^{**}a_n=a_{n+1}a_n.
\]
Consequently $\phi_n=a_n:A_n\to A_n^{**}$ defines a morphism between the two direct systems and induces an isometric embedding $\bar\phi:T\to S$. For every $n$,
\[
\iota\bar\phi\kappa_{n,\infty}
=\iota\lambda_nJ_{A_n}
=\kappa_{n,\infty}^{**}J_{A_n}
=J_T\kappa_{n,\infty}.
\]
The ranges of the $\kappa_{n,\infty}$ have dense union in $T$, and hence
\[
\iota\bar\phi=J_T.
\]
The resulting embedding of the dual chain into its bidual is therefore exactly its canonical embedding:
\begin{equation}\label{eq:cat-hyperdual-embed}
\widetilde X\ \xhookrightarrow{\ J_{\widetilde X}\ }\ D^{2}(\widetilde X).
\end{equation}
By Proposition \ref{prop:limit-reflexive},
\[
J_{\widetilde X}\text{ is surjective}
\quad\Longleftrightarrow\quad
\widetilde X\text{ is reflexive}
\quad\Longleftrightarrow\quad
X\text{ is reflexive}.
\]
In particular, the canonical embedding of the completed hyperdual space is not surjective when $X$ is non-reflexive.

The existence of some bounded linear isomorphism $\widetilde X\to\widetilde X^{**}$, or of some isometric isomorphism between these spaces, is a different assertion. The second of these assertions is the property called parareflexivity in [U3]. The calculation above determines the canonical embedding, rather than all possible isomorphisms. Likewise, the failure of commutation for the finite-dimensional system of Example \ref{ex:cat-cn} does not by itself settle an abstract isomorphism assertion for every completed hyperdual space. Any such additional assertion requires its own proof or counterexample. The isometric question mentioned in [U3, Remark 1 (b)] concerns the existence of an isometric isomorphism, not the surjectivity of the specified canonical embedding considered here.

\subsubsection{Our replacement statements}\label{sub:cat-alt}

\begin{enumerate}
  \item[(a)] \textbf{Take the object of the direct limit to be ``the completion of $\bigsqcup/\sim$''}, i.e.\ the convention of this paper; in this way the hyperdual space is indeed a Banach space, and the ``Banach space'' part of [U3, Theorem 2 (i)] is retained (which is also consistent with the practice of \cite{semadeni1965,gwizdek2023}).
  \item[(b)] \textbf{Replace the equality of [U3, Theorem 1 (i)] by an isometric embedding}: by Example \ref{ex:cat-cn}, $D^{2k}(\varinjlim X)\cong\varinjlim D^{2k}[X]$ does not hold; the correct statement is Proposition \ref{prop:colim-embed}---$D^{2k}$ sends the morphisms of a direct system to morphisms of the direct limit (functoriality), and the comparison map $\iota_k$ is an isometric embedding which is in general not surjective.
  \item[(c)] \textbf{Downgrade the $\cong$ of [U3, Theorem 2 (i)] to \eqref{eq:cat-hyperdual-embed}}: for the uncompleted object of [U3], Example \ref{ex:cat-canonical} disproves the Banach-space assertion whenever $X$ is a non-reflexive Banach space. For the completed hyperdual space used in this paper, Theorem \ref{thm:tower-exists} establishes completeness, and \S\ref{sub:cat-correct} identifies the resulting specified embedding with $J_{\widetilde X}$. Its surjectivity is completely characterized by Proposition \ref{prop:limit-reflexive}: it holds exactly when $X$ is reflexive (Remark \ref{rem:reflexive-vs-iso}), which is strictly stronger than the \emph{parareflexivity} of [U3, Remark 1 (b)]---the latter asks only for the existence of \emph{some} isometric isomorphism $T(\widetilde X)\equiv T(\widetilde X)^{**}$ (compare Example \ref{ex:james}). The existence of an arbitrary bounded or isometric isomorphism between $\widetilde X$ and $\widetilde X^{**}$ is a separate question and is not established by this comparison-map argument; see also the discussion of the sandwiching between $\widetilde X$ and $D^{2}(\widetilde X)$ in \cite{dualchain}.
  \item[(d)] \textbf{Parts that need restating accordingly}: the proof of [U3, Theorem 2 (i)] (which depends on Theorem 1 (i)) and Theorem 3 (iii) (where the original says that it suffices to apply the proof of Lemma 4 to $\widetilde D$) must both be rewritten.
  \item[(e)] \textbf{Positive results that can be retained}: even though commutativity fails, ``every $X^{(2n)}$ embeds isometrically into $\widetilde X$'' (the universal property of Theorem \ref{thm:tower-exists}) and ``$\widetilde X$ contains a chain of retracts'' ([U3, Theorem 2 (ii)], see \S\ref{sec:wording}) do not depend on Lemma 4 and are retained; the statement about maps on p.~10 (``$\varinjlim f$ preserves isometries/isomorphisms'') and the identity $\|\varinjlim f\|=\|f\|=\lim\|f_n\|$ on p.~9 are also correct.
\end{enumerate}

\subsubsection{Conclusions that are unaffected}\label{sub:cat-unaffected}

The above counterexamples are directed only at the group of assertions saying that ``$D^{2k}$ commutes with direct limits''. Section \ref{sec:prelim} has already re-established from the positive side the part that is actually needed, and none of it depends on Lemma 4 or Theorem 1 (i):

\begin{enumerate}
  \item[(i)] the existence of the hyperdual space, the fact that it is a Banach space, and its universal property (Theorem \ref{thm:tower-exists}, corresponding to the correct part of [U3, Theorem 1]);
  \item[(ii)] $\bigl(\widetilde X\bigr)^{*}\equiv\varprojlim_nX^{(2n+1)}$ (Theorem \ref{thm:dir-dual-inv}, corresponding to [U3, Corollary 2]);
  \item[(iii)] the functoriality and norm preservation of the limit functors $\Dir,\Inv$, the natural transformation $\eta$, and $\Dir(T)^{*}=\Inv(T)$, $\Dir(X)\equiv\Dir(X^{**})$ (Theorem \ref{thm:functors});
  \item[(iv)] the operator properties preserved by the limit: direct-sum decompositions, the spectrum and the Fredholm index (Proposition \ref{prop:preserved});
  \item[(v)] the fact that non-reflexivity does not disappear after taking the limit (Proposition \ref{prop:limit-reflexive}) and the passage of separability (Proposition \ref{prop:limit-separable}).
\end{enumerate}

The independently established statements retained here are the construction and universal property of the completed hyperdual space, the isometric embeddings of its layers, the duality between the completed direct limit and the bounded inverse limit, and the properties proved in Section \ref{sec:prelim} or explicitly cited there---the preservation statement of Proposition \ref{prop:preserved} is quoted from Gwizdek \cite{gwizdek2023}. The commutation assertion is refuted by Example \ref{ex:cat-cn}, and the completeness of the uncompleted direct limit is refuted by Example \ref{ex:cat-canonical} for non-reflexive Banach spaces. These counterexamples are kept separate from the existence of an arbitrary isomorphism between a completed hyperdual space and its bidual. The preservation assertions concerning parareflexivity and quasi-reflexivity in [U3, Theorem 3 (vii)] are \emph{not} used as established results here, because no independent proof of those assertions is supplied in this paper. Reflexivity itself is covered by Proposition \ref{prop:limit-reflexive}.

\subsection{Questions of formulation}\label{sec:wording}

This section deals with two issues at the level of \emph{formulation}; neither of them is a substantive error. First, one turn of phrase in [U3, Theorem 2 (ii)] concerning direct-sum decompositions admits several literal readings; we adopt the \emph{correct} reading, and explain why the remaining readings cannot stand. Second, we flag a number of misprints in the three preprints, purely as a caution, so that the reader is not perplexed when checking the originals. At the end of this section we give the table comparing the \emph{terminology and notation} used in this paper with that of the three preprints, as promised in Remark \ref{rem:naming}.

\subsubsection{The original text and its location}

\begin{pquote}
\textbf{[U3, Theorem 2 (ii)]} (pp.~15--16): ``all $D^{2k}(X)$, $k \in \N$, embed
isometrically into $\widetilde X$ making an increasing sequence of retracts and
retracts of $\widetilde X$ in $(\BF)_1$, implying that, for each $k \in \N$,
$\widetilde X \cong D^{2k}(X) \dotplus N(r_{2k})$ isometrically''
\end{pquote}
where $i_{2k} : D^{2k}(X) \to \widetilde X$ is the isometric embedding, $r_{2k} \in (\BF)_1(\widetilde X, D^{2k}(X))$ is a retraction of $i_{2k}$ ($\|r_{2k}\| \le 1$), and $N(r_{2k}) = \Ker r_{2k}$.

Outside of quotations we always switch to the notation of this paper. The correspondence between the two systems is: the original's $D^{2k}(X)$ is our $X^{(2k)}$; the original's isometric embedding $i_{2k}:D^{2k}(X)\to\widetilde X$ into the hyperdual corresponds to our structure map $\kappa_{k,\infty}:X^{(2k)}\to\widetilde X$ (Definition \ref{def:tower-space}), whereas the original's \emph{canonical embedding} $j_{2k}:D^{2k}(X)\to D^{2k+2}(X)$ corresponds to our adjacent embedding $J_{X^{(2k)}}:X^{(2k)}\to X^{(2k+2)}$ (Definition \ref{def:canonical}); and the original's $\widetilde X$ is the ``hyperdual space'', which differs from our hyperdual space $\widetilde X$ (Definition \ref{def:tower-space}) only by a completion. The complete comparison is given at the end of this subsection.

\subsubsection{The topological direct sum and the norm geometry of the decomposition}

\begin{proposition}[Retractions onto higher layers]\label{prop:tower-layer-retractions}
Let $X$ be a Banach space and let $X_n=X^{(2n)}$. For every $k\geq1$ there exists a contraction
\[
r_{2k}:\widetilde X\longrightarrow X_k,
\qquad r_{2k}\kappa_{k,\infty}=\id_{X_k}.
\]
Consequently $P_k=\kappa_{k,\infty}r_{2k}$ is a contractive projection and
\[
\widetilde X=\kappa_{k,\infty}(X_k)\oplus\ker r_{2k}
\]
is a topological direct sum.
\end{proposition}
\begin{proof}
For $n\geq1$, the space $X_n$ is the dual of $X^{(2n-1)}$. Set
\[
\rho_n=(J_{X^{(2n-1)}})^*:X_{n+1}\longrightarrow X_n.
\]
Then $\|\rho_n\|\leq1$ and $\rho_nJ_{X_n}=\id_{X_n}$, by $J_F^*J_{F^*}=\id_{F^*}$. Fix $k\geq1$ and define
\[
q_{k,n}=\begin{cases}
\kappa_{n,k},&n<k,\\
\id_{X_k},&n=k,\\
\rho_k\rho_{k+1}\cdots\rho_{n-1},&n>k.
\end{cases}
\]
These maps are contractions and satisfy $q_{k,n+1}\kappa_{n,n+1}=q_{k,n}$. The universal property of the completed hyperdual space gives a contraction $r_{2k}$ with $r_{2k}\kappa_{n,\infty}=q_{k,n}$. Taking $n=k$ proves the retraction identity.

It follows that $P_k^2=P_k$, its range is $\kappa_{k,\infty}(X_k)$, and its kernel is $\ker r_{2k}$. Both subspaces are closed. The addition map from their product to $\widetilde X$ (the product being equipped with the maximum norm) is bounded, and its inverse is
\[
x\longmapsto(P_kx,(\id_{\widetilde X}-P_k)x),
\]
which is bounded as well. This proves the asserted topological direct sum.
\end{proof}

Let $P$ be a bounded projection on a Banach space $E$. The topological decomposition $E=R(P)\oplus\ker P$ is isometric for the maximum direct-sum norm exactly when
\[
\|x\|=\max\{\|Px\|,\|(\id_E-P)x\|\}\qquad(x\in E).
\]
Such a projection is an $M$-projection. It is isometric for the sum norm exactly when
\[
\|x\|=\|Px\|+\|(\id_E-P)x\|\qquad(x\in E),
\]
in which case $P$ is an $L$-projection. Both conditions imply $\|P\|\leq1$ and $\|\id_E-P\|\leq1$; the pair of inequalities alone is the weaker condition of bicontractivity.

For example, on $E=\mathbb R^2$ use
\[
\|(a,b)\|=\max\{|a|,|b|,|a-b|\},\qquad P(a,b)=(a,0).
\]
Then $\|P\|=\|\id_E-P\|=1$. The vector $(1,-1)$ disproves the $M$-projection identity, and $(1,1)$ disproves the $L$-projection identity. Therefore bicontractivity does not supply either of the specified isometric direct-sum norms.

For the dual chain, the bounded retraction gives the topological decomposition
\[
\widetilde X=\kappa_{k,\infty}(X_k)\oplus\ker r_{2k},
\qquad P_k=\kappa_{k,\infty}r_{2k}.
\]
The inclusions of its two summands preserve their inherited norms. A maximum-norm isometric decomposition through this particular range would require $P_k$ to be an $M$-projection, and Proposition \ref{prop:no-M-projection} disproves the universal existence of such a choice.

\subsubsection{The interpretation of the direct-sum wording}

Read literally, the sentence quoted above admits several readings of the symbol ``$\dotplus$'' and of the word ``isometrically''. We retain the topological direct-sum interpretation, with the inherited norms on the two subspaces: ``$\dotplus$'' denotes the topological direct sum, and ``isometrically'' qualifies the two embeddings $\kappa_{k,\infty}$ and the inclusion of $\ker r_{2k}$. Requiring the maximum norm or the sum norm is a separate geometric assertion. Allowing an arbitrary transported norm on an external direct sum merely restates the original norm and introduces no additional geometric conclusion. Accordingly, every isometric direct-sum assertion in this paper specifies the external norm under consideration.

Under this topological interpretation the original sentence is merely \emph{loosely worded}, not substantively wrong, and its conclusion holds unconditionally by Proposition \ref{prop:tower-layer-retractions}. The strengthened readings that impose an isometric direct-sum norm, by contrast, require the extra geometry of an $M$- or an $L$-projection and are false in general (the two-dimensional example above); and the reading that only asks for ``\emph{some} isometric decomposition'' demands a construction that [U3] does not provide. This is precisely why we classify this item as a ``question of formulation'' rather than a ``substantive error'', in contrast to \S\ref{sec:dim} and \S\ref{sec:cat}.

\subsubsection{Aligning the conventions on norms}

This paper uses the grain norm of the dual chain for $\widetilde X$, that is, the completion norm of the direct limit of normed spaces (Definition \ref{def:tower-space}); the original's $\widetilde X$ is an uncompleted quotient space, and the two differ by a completion. Under the topological interpretation, ``$\dotplus$'' only asserts the algebraic decomposition into two closed subspaces and the boundedness of the projection, and introduces no new norm, so the two conventions do not conflict: the grain norm on $\widetilde X$ is always the one in force, and the direct sum is read only as a topological direct sum.

\subsubsection{The statement we adopt}

Let $r_{2k} : \widetilde X \to X_k$ satisfy $r_{2k}\kappa_{k,\infty} = \id_{X_k}$. Then
    \[
      \widetilde X = \kappa_{k,\infty}(X_k) \oplus \ker r_{2k}
    \]
is a topological direct sum, and both $\kappa_{k,\infty}$ and the inclusion of $\ker r_{2k}$ are isometric. This is a corollary of Proposition \ref{prop:tower-layer-retractions}, holds unconditionally, and is the correct content of [U3, Theorem 2 (ii)].

\paragraph{Conditional statement about $M$-ideals.}
For a closed subspace $Y\subset E$, being an $M$-ideal means that $Y^\perp\subset E^*$ is the range of an $L$-projection on $E^*$ (Alfsen--Effros \cite{alfsen-effros1972}; Harmand--Werner--Werner \cite{hww1993}; the idea goes back to Kakutani's duality theory for abstract $L$-spaces and $M$-spaces \cite{kakutani1941}). This is an additional property of the inclusion $Y\subset E$, not a property asserted here for every pair of successive layers of the dual chain. When it holds, it gives a decomposition on the dual side. It does not by itself give a projection from $E$ onto $Y$. The standard example is $c_0\subset\ell^\infty$: $c_0$ is an $M$-ideal, but is not complemented in $\ell^\infty$ (Phillips' theorem, see \cite{albiac-kalton2006}). We therefore keep the questions of bounded complementation, bicontractivity, $M$-summands, and $M$-ideals separate.

\noindent We do \emph{not} require an isometric direct sum: under the topological interpretation only a topological direct sum is required, and the isometry is provided by the two embeddings $\kappa_{k,\infty}$ and the inclusion of $\ker r_{2k}$. Whether, for every Banach space $X$ and every $k\ge1$, the projection $P_k = \kappa_{k,\infty} r_{2k}$ can be chosen to be an $M$-projection (and hence give an isometric direct sum) is answered in the \emph{negative} in \S\ref{sec:open}: for $X=c_{0}$ and $k=1$ no such $M$-projection exists (Proposition \ref{prop:no-M-projection}). Since only a topological direct sum is needed, this has no effect on the conclusions of the present section. The weaker question, whether for a given $X$ and $k$ one can choose a retraction $r_{2k}$ satisfying the two \emph{norm inequalities} $\|r_{2k}\|\le1$ and $\|\id_{\widetilde X}-\kappa_{k,\infty}r_{2k}\|\le1$ that occur in [U3, Theorem 2 (ii)], is a different matter and is not settled here.

\subsubsection{The range of validity of the chain of retracts}

Note that the original's assertion about the chain of retracts is made only for $k \in \N$, i.e.\ $k \ge 1$; its proof also says explicitly ``Let us exclude $j_0$ off $\widetilde X$''. For $k \ge 1$, the fact that ``$X^{(2k)}$ is a retract of $\widetilde X$'' depends on [U2, Theorem 1] and [U3, Lemma 2]; since neither of these uses [U3, Lemma 4] or [U3, Theorem 1 (i)]---[U3, Lemma 5] belongs to a different construction and is not the lemma invoked here---this part is not touched by the corrections of \S\ref{sec:cat} and can be retained (it is precisely this part that (e) of \S\ref{sec:cat} refers to). In addition, Proposition \ref{prop:tower-layer-retractions} now supplies an independent proof of the existence of these retractions for every $k\ge1$, so that the conclusion no longer relies on the preprints under review. As for $k = 0$, i.e.\ whether $X$ itself is a retract of $\widetilde X$, the original makes no assertion, and neither do we.

\subsubsection{Misprints in the three papers}

The following misprints do not affect any substantive conclusion; they are listed here only for completeness.

\begin{enumerate}
  \item[(1)] Both items of \textbf{Definition 1} of [U3] are numbered \emph{(i)} (the second should be (ii)); and in Lemma 1 of the same paper the target category is written as $(\iNF)_1$, $(\iNF)^1$ (according to the context it should be $(\NF)_1$, $(\NF)^1$, or a notation of the $\iBF$ type).
  \item[(2)] In the statement of \textbf{Theorem 6} of [U2], the right-hand side of ``$\|f_{n,k}\| = \|f_{n,k}\|$'' is evidently $\|f_n\|$.
  \item[(3)] In \textbf{the proof of Theorem 2 (iii)} of [U2] (p.~17) there appears ``$D(X/Z) \cong Z^9$'', which should be $Z^0$ (the annihilator of $Z$); moreover, in the statement on p.~15 the subscripts of the three annihilators $Z^0$, $D^{2n}(Z)^0$, $D^{2n-1}(X/Z)^0$ are written ambiguously (it is not stated at which level the ``$0$'' is taken), so care is needed when quoting.
  \item[(4)] In \textbf{the proof of Theorem 3} of [U3], in item (vi) the natural transformation is written as $\iota_{2k} : D^{2k} \rightsquigarrow \widetilde D$, whereas elsewhere in the text $j_{2k}$ is used for the canonical embedding; the two systems of symbols are easily confused.
\end{enumerate}

\paragraph{Hypotheses in the original.}
Unlike the misprints above, the next point is not a purely typographical slip but a matter of \emph{hypotheses}, and it is therefore recorded separately rather than included among the misprints. In \textbf{the proof of Lemma 4} of [U2], ``$\dim X = |X|$'' and ``$\dim X^* = |X^*|$'' are used interchangeably, but the two require different hypotheses. The former is the Halbeisen--Hungerb\"uhler formula and needs $X$ to be an \emph{infinite-dimensional Banach} space: it fails for an incomplete normed space (the space $c_{00}$ of finitely supported sequences has $\dim c_{00}=\aleph_0$ while $|c_{00}|=\cc$) and also in finite dimension (there $\dim X=n<\infty$ while $|X|=\cc$). The latter needs no completeness assumption on $X$, because $X^*$ is automatically a Banach space and is infinite-dimensional as soon as $X$ is; this distinction must be noted when quoting.

\subsubsection{Comparison of terminology and notation}

Remark \ref{rem:naming} promised, in this section, a comparison with the terminology used in the three preprints.

\begin{center}
\begin{tabular}{@{}p{0.40\textwidth}p{0.52\textwidth}@{}}
\toprule
This paper & The three preprints\\
\midrule
$X^{(2n)}$ ($2n$-th iterated dual) & $D^{2n}(X)$\\
adjacent embedding $J_{X^{(2n)}} : X^{(2n)} \to X^{(2n+2)}$ &
the canonical embedding $j_{2k} : D^{2k}(X) \to D^{2k+2}(X)$\\
structure map $\kappa_{n,\infty} : X^{(2n)} \to \widetilde X$ &
the isometric embedding $i_{2k} : D^{2k}(X) \to \widetilde X$ into the ``hyperdual space''\\
hyperdual space $\widetilde X$ (completion taken) &
the ``hyperdual space'' $\widetilde X$ of [U3] (${\bigsqcup}/{\sim}$, no completion taken)\\
direct-limit functor $\Dir$ & the hyperdual functor $\widetilde D$ of [U3]\\
retraction $r$ ($r\,J = \id$, $\|r\| \le 1$) & $r_{2k}$\\
inverse hyperdual space $\overleftarrow X = \varprojlim_n X^{(2n+1)}$ &
the inverse limit occurring in [U3, Corollary 2] (the original gives it no fixed name)\\
\bottomrule
\end{tabular}
\end{center}

\noindent Recall also that the \emph{reference numbering differs from paper to paper} in the three preprints, so it is very easy to mix them up when reading: in [U3], [U1] is cited as [15] and [U2] as [16], and there is \emph{no} [17]; [17] occurs only in the reference list of [U2] (where [U1] $=$ [17]). Hence ``Theorem 1 of [16]'' in the text of [U3] means Theorem 1 of [U2], while ``Theorem 2 (i) of [17]'' in the text of [U2] means Theorem 2 (i) of [U1]. One must match the numbers first when checking; otherwise one will mistake a conclusion of one paper for that of another.

\section{Two questions answered in the negative}\label{sec:open}

We collect the consequences established above: the exact dimension identities for the dual chain of $\ell^{\infty}$, the classification of finite quotient shape, and a geometric obstruction to $M$-projections onto canonical layers of the dual chain. The dimension identities follow from Lemma \ref{lem:linf-dual} and Remark \ref{rem:beth-tower}; the negative answer concerning finite quotient shape follows from Theorem \ref{thm:shape-classification}; and Proposition \ref{prop:no-M-projection} supplies the $M$-projection counterexample. Both questions about the objects of the theory are answered in the \emph{negative}, and like the rest of the paper these are \emph{purely mathematical} statements about the objects under discussion, which have nothing to do with the choices made in the three preprints.

\medskip
\noindent\emph{The dimension value.} The Hamel dimension of the second dual of $\ell^{\infty}$ is bounded above by $2^{2^{\cc}}$ (Proposition \ref{prop:linf-bidual}), and this bound is attained, so that the exact value---used throughout \S\ref{sec:dim}---is known:
\begin{equation}\label{eq:open-dim}
  \dim\ell^{\infty**}=2^{2^{\cc}},\qquad
  \dim(\ell^{\infty})^{(n)}=\beth_{n+1}\ \ (n\ge0),\qquad
  \dim\widetilde{\ell^{\infty}}=\beth_{\omega+1}.
\end{equation}
The middle identity is the two-step computation of Remark \ref{rem:beth-tower}, the first is the case $n=2$ of that identity, and the third says that completing the hyperdual space raises the Hamel dimension from $\beth_{\omega}$---the dimension of the incomplete union $\bigcup_k(\ell^{\infty})^{(2k)}$---to $\beth_{\omega+1}$. Hence the general lower bound
\[
  \dim\widetilde X\ \ge\ \sup_{n}\dim X^{(2n)}
  \qquad\text{(Proposition \ref{prop:dim-tower-lower})}
\]
\emph{can be strict}, as announced in Remark \ref{rem:dim-root}; for $X=\ell^{\infty}$ its right-hand side equals $\sup_n\beth_{2n+1}=\beth_\omega$, while the left-hand side is $\beth_{\omega+1}$. Consequently every dimension appearing in the dual chain of $\ell^{\infty}$ is determined, and the exact value $2^{2^{\cc}}$ is the one used in \S\ref{sec:dim}. For an arbitrary space $X$ the bound above is the general lower bound adopted in this paper; we do not list it as a separate numbered problem.

\begin{enumerate}
  \item[\textbf{(O1)}] \textbf{Is $\Sh_{\bar0}(\ell^{\infty})$ equal to the unique $l_p$-type $\Sh_{\bar0}(l_p)$?}
    \emph{Answer: no.} Since $\dim\ell^{\infty}=\dim l_p=\cc$, the assertion (H1) of [U1] (same dimension $\Rightarrow$ the same finite quotient shape) would give $\Sh_{\bar0}(\ell^{\infty})=\Sh_{\bar0}(l_p)$; yet the author of [U1] wrote in \cite[the text preceding Remark 3.6]{uglesic2018lp} that ``it seems that $\Sh_{\bar0}(l^\infty)$ might differ from (the unique) type $\Sh_{\bar0}(l_p)$'', which apparently conflicts with his own Theorem 2 (i). In fact the doubt is justified: by the classification criterion \eqref{eq:shape-criterion} the invariant of the finite quotient shape is $\dim X^{*}$ and not $\dim X$, and for $1\le p<\infty$
    \[
      \dim(\ell^{\infty})^{*}=2^{\cc}\ \ne\ \cc=\dim(l_{p})^{*},
    \]
    so that $\Sh_{\bar0}(\ell^{\infty})\ne\Sh_{\bar0}(l_{p})$ (Theorem \ref{thm:shape-classification}). This is statement (h) of \S\ref{sec:dim}, where the instance of (H1) for $\ell^{\infty}$ and $l_p$ is thereby decided in the negative. What \cite{uglesic2018lp} actually proves about $\ell^{\infty}$ is only $\Sh_{\aleph_0}(l^\infty)=\Sh_{\bar0}(l^\infty)=\Sh_{\bar0}(F^{N}_0(\infty)^{**})$, which is perfectly compatible with the negative answer and does not compare $\ell^{\infty}$ with the type of $l_p$. The mechanism of the answer is \emph{algebraic}: what was needed was an isomorphism invariant of the finite quotient shape, and it is supplied by $\dim X^{*}$.
  \item[\textbf{(O2)}] \textbf{For every Banach space $X$ and every $k\ge1$, can the projection $P = \kappa_{k,\infty} r_{2k}$ on $\widetilde X$ be chosen to be an \emph{$M$-projection}, that is, can one have}
    \[
      \|w\| = \max\bigl\{\|Pw\|,\ \|(\id-P)w\|\bigr\} \qquad (w \in \widetilde X),
    \]
    \textbf{where $r_{2k}\in(\BF)_1(\widetilde X, X^{(2k)})$ is a retraction of the structure map $\kappa_{k,\infty}$ (Definition \ref{def:tower-space}), i.e.\ $r_{2k}\kappa_{k,\infty}=\id$?}
    \emph{Answer: no.} Here $P=\kappa_{k,\infty}r_{2k}$---rather than $J_{X^{(2k)}}r_{2k}$---is the operator on $\widetilde X$ whose range is the canonical copy of $X^{(2k)}$; this more precise notation is necessary because the range of $r_{2k}$ is a layer of the dual chain, not the hyperdual space, and $\kappa_{k,\infty}$ embeds that layer into $\widetilde X$. If $P$ could be chosen to be an $M$-projection, then the direct-sum decomposition $\widetilde X=\kappa_{k,\infty}(X^{(2k)})\oplus N(r_{2k})$ of \S\ref{sec:wording} would be an \emph{isometric} direct sum in the $\ell^{\infty}$ sense. The following counterexample shows that this is in general impossible.
\end{enumerate}

\begin{proposition}[No $M$-projection in general]\label{prop:no-M-projection}
There exist a Banach space $X$ and an index $k$ such that the hyperdual space $\widetilde X$ admits no $M$-projection onto the canonical copy of $X^{(2k)}$: no bounded linear $Q:\widetilde X\to\widetilde X$ with $Q^{2}=Q$, $R(Q)=R\bigl(\kappa_{k,\infty}\bigr)$ and
\[
  \|w\|=\max\bigl\{\|Qw\|,\ \|(\id-Q)w\|\bigr\}\qquad(w\in\widetilde X)
\]
exists. The pair $X=c_{0}$, $k=1$ already has this property.
\end{proposition}

\begin{proof}
Take $X=c_0$ and $k=1$. Put $E=X^{(2)}=\ell^\infty$ and $F=X^{(4)}=E^{**}$. The structure maps of the dual chain satisfy
\[
  \kappa_{2,\infty}J_E=\kappa_{1,\infty}.
\]
Suppose that $Q$ is an $M$-projection on $\widetilde X$ with range $\kappa_{1,\infty}(E)$. Since this range is contained in $\kappa_{2,\infty}(F)$, the formula
\[
  P=\kappa_{2,\infty}^{-1}\,Q\,\kappa_{2,\infty}:F\longrightarrow F
\]
is well defined, where the inverse is taken on $\kappa_{2,\infty}(F)$. Then $P^2=P$, $R(P)=J_E(E)$, and the isometry of $\kappa_{2,\infty}$ gives
\[
  \|w\|=\max\{\|Pw\|,\ \|w-Pw\|\}\qquad(w\in F).
\]
Thus $P$ is an $M$-projection on $F$.

The space $E=\ell^\infty$ is non-reflexive, so $J_E(E)$ is a proper subspace of $F$. Choose $0\ne z\in\Ker P$. Positive finitely additive measures linearly span $\ba(\N)=E^*$, so there is a positive $\mu\in\ba(\N)$ with $z(\mu)\ne0$. After normalising $\mu$ and multiplying $z$ by a scalar of modulus one, arrange that $\|\mu\|=\mu(\N)=1$ and $z(\mu)>0$. Choose $t>0$ with $t\|z\|<1$, and let $e=J_E(\mathbf 1)$. Then $Pe=e$, $Pz=0$ and $\|e\|=1$. The $M$-projection identity gives
\[
  \|e+tz\|=\max\{1,\ t\|z\|\}=1 .
\]
On the other hand, evaluation at the norm-one functional $\mu$ gives
\[
  \|e+tz\|\ \ge\ \bigl|(e+tz)(\mu)\bigr|=1+tz(\mu)>1 ,
\]
a contradiction. Hence the required $M$-projection does not exist.
\end{proof}

\noindent Two remarks on the scope of Proposition \ref{prop:no-M-projection}. First, the counterexample is directed at the \emph{$M$-projection property} only, and its reach is exactly what the proposition states: for $X=c_{0}$ and $k=1$ no retraction $r_{2}$ can be accompanied by an $M$-projection onto the canonical copy of $X^{(2)}$. Whether, for a given $X$ and $k$, one can choose \emph{some} retraction $r_{2k}$ satisfying the \emph{norm inequalities} $\|r_{2k}\|\le1$ and $\|\id_{\widetilde X}-\kappa_{k,\infty}r_{2k}\|\le1$ that occur in [U3, Theorem 2 (ii)] is a weaker and different question---it asks for a bicontractive projection, not for an $M$-projection---and it is not settled here. Under the topological interpretation adopted in \S\ref{sec:wording}, however, neither question is needed: only a \emph{topological} direct sum is required there, and this is provided unconditionally by Proposition \ref{prop:tower-layer-retractions}.

Second, the negative answer explains why the $M$-ideal picture cannot replace the $M$-projection property. An $M$-ideal is characterized on the \emph{dual} side and yields an $L$-summand decomposition there, not a direct-sum decomposition of the original space; no universal assertion that $X^{(2k)}$ is an $M$-ideal in $X^{(2k+2)}$ is used or claimed. The standard instance $X=c_{0}$ illustrates the gap: $c_{0}$ is an $M$-ideal of $\ell^{\infty}$, yet $c_{0}$ is uncomplemented in $\ell^{\infty}$ (Phillips' theorem), so it is certainly not an $M$-summand. Thus neither the norm inequalities nor the $M$-ideal property delivers an isometric decomposition, in agreement with the counterexample above.

\noindent To sum up: the two questions above are answered by quite different mechanisms---(O1) by an \emph{algebraic} invariant of the finite quotient shape, (O2) by a \emph{geometric} argument on the $\ell^{\infty}$-side of the dual chain (a norm-one projection need not be an $M$-projection). Both answers are negative, and the exact dimensions of \eqref{eq:open-dim} complete the picture of \S\ref{sec:dim}.

%=====================================================================
\section*{Acknowledgements}
\addcontentsline{toc}{section}{Acknowledgements}
The first author thanks the members of the QQ group ``八云的魔法书交流群'' (group number: 1080353942) for taking part in discussions of some of the questions in this paper and for checking parts of its content, in particular (in no particular order; the nicknames shown are those used at the time of writing): f1; 群莉莉白，时不时让群友发春；群幽香，时不时来群里抓群友给太阳花施肥；学代数学的；SuppleMarisa；Ludwig II Skłodowsk-Hearn.
The first author also thanks N.~Ugle\v{s}i\'c for candidly accepting the errors pointed out in his earlier papers.

It should be noted that the role of the second author in the preparation and writing of this paper was closer to that of a consultant, or advisor, than to that of a co-author; out of respect for him, and for other reasons as well, he is nevertheless listed in the second position on the byline.

%=====================================================================

\end{document}